\documentclass[11pt, oneside]{amsart}

\usepackage[ 
        pdftitle={AS VMHS}%
      dvips,colorlinks=true]{hyperref}
\hypersetup{pdfpagelabels,
                  plainpages=false,
  bookmarksnumbered=true,
  linkcolor=red,
  citecolor=black,
  pagecolor=black, 
  urlcolor=black,  
}

\renewcommand*{\HyperDestNameFilter}[1]{\jobname-#1} 
\numberwithin{equation}{section}
\usepackage[centering]{geometry}

\usepackage[nameinlink]{cleveref}

\usepackage{amsmath,amssymb,cite,mathrsfs,tikz-cd}
\usepackage[all]{xy}
\usepackage{typearea} 
\usepackage{hyperref} 
\usepackage{color} 
\usepackage[scr=euler,scrscaled=.96]{mathalfa}

\usepackage{epstopdf} 
\usepackage{booktabs}

\newtheorem{teo}{Theorem}[section]
\newtheorem{thm}[teo]{Theorem}
\newtheorem{prop}[teo]{Proposition}
\newtheorem{lemma}[teo]{Lemma}

\newtheorem{cor}[teo]{Corollary}

\newtheorem{def-lemma}[teo]{Definition-Lemma}
\newtheorem{notation}[teo]{Notation}
\theoremstyle{remark}
\newtheorem{rmk}[teo]{Remark}

\theoremstyle{remark}
\newtheorem{defn}[teo]{Definition}

\newtheoremstyle{named}{}{}{\itshape}{}{\bfseries}{.}{.5em}{\thmnote{#3 }#1}
\theoremstyle{named}

\makeatletter
\newcommand{\neutralize}[1]{\expandafter\let\csname c@#1\endcsname\count@}
\makeatother

\numberwithin{equation}{section}

  \newcommand{\C}{\mathbb{C}}

  \newcommand{\G}{\mathbb{G}}

  \newcommand{\N}{\mathbb{N}}
  
  \newcommand{\Q}{\mathbb{Q}}
  \newcommand{\R}{\mathbb{R}}  
  \renewcommand{\S}{\mathbb{S}}
  \newcommand{\V}{\mathbb{V}}
  \newcommand{\Z}{\mathbb{Z}}

  \renewcommand{\cong}{\simeq}
  \renewcommand{\bar}{\overline}
  \renewcommand{\tilde}{\widetilde}

  \providecommand{\frac}[1]{\operatorname{Frac}(#1)}

  \renewcommand{\hom}{\operatorname{Hom}}

  \newcommand{\MT}{\operatorname{MT}}
  \newcommand{\GL}{\operatorname{GL}}

  \renewcommand{\lim}{\operatorname{lim}}

  \newcommand{\lie}{\operatorname{Lie}}
  \newcommand{\Gr}{\operatorname{Gr}}

  \newcommand{\Zar}{\mathrm{Zar}}

  \newcommand{\Ad}{\textnormal{Ad}}
  \newcommand{\an}{\textnormal{an}}
  \newcommand{\ol}{\overline}

\newcommand{\ZZ}{\mathbb{Z}}
  \newcommand{\QQ}{\mathbb{Q}}
  \newcommand{\CC}{\mathbb{C}}

\newcommand{\cD}{\mathcal{D}}

\newcommand{\cF}{\mathcal{F}}

\newcommand{\cM}{\mathcal{M}}
\newcommand{\cO}{\mathcal{O}}
\newcommand{\cQ}{\mathcal{Q}}
\newcommand{\cR}{\mathcal{R}}

\newcommand{\cV}{\mathcal{V}}
\newcommand{\cW}{\mathcal{W}}
\newcommand{\cX}{\mathcal{X}}

\newcommand{\cZ}{\mathcal{Z}}

\newcommand{\pullbackcorner}[1][dr]{\save*!/#1-1.7pc/#1:(-1.5,1.5)@^{|-}\restore}

\newcommand\supervisor[1]{\def\@supervisor{#1}}

\newcounter{elno}

\renewcommand{\cong}{\simeq}

\begin{document}
\title[]{The Ax--Schanuel conjecture for variations of mixed Hodge structures}
\author{Ziyang Gao, Bruno Klingler}

\address{Institute of Algebra, Number Theory and Discrete Mathematics; Leibniz University Hannover; Welfengarten 1, 30167 Hannover, Germany}
\email{ziyang.gao@math.uni-hannover.de}

\address{Dept. of Mathematics, Humboldt Universit\"{a}t, Berlin, Germany}
\email{bruno.klingler@math.hu-berlin.de}


\maketitle
\begin{abstract}
We prove in this paper the Ax--Schanuel conjecture for all admissible variations of mixed Hodge structures.
\end{abstract}

\tableofcontents

\section{Introduction}
In this paper we prove the Ax--Schanuel conjecture for all admissible, graded-polarized, integral variation of mixed Hodge structures over a smooth complex
quasi-projective variety $S$.

Let $(\V_\ZZ,W_\bullet,\cF^\bullet) \rightarrow
S^\an$ be an admissible, graded-polarized, integral
variation of mixed Hodge structures on the complex manifold
$S^\an$ associated to $S$. Let $[\Phi]\colon S^\an \rightarrow \Gamma\backslash
\cM$ be the associated complex analytic period map, where $\cM$
denotes the period domain classifying graded polarized mixed Hodge
structures of the relevant type and $\Gamma$ is an
arithmetic subgroup in the group of automorphisms of $\cM$. The
classifying space $\cM$
admits a natural realization as a real semi-algebraic subset, open in the usual
topology, of a complex algebraic variety $\cM^\vee$. The Ax--Schanuel
conjecture is a functional transcendence statement 
comparing the algebraic structure on $\cM^\vee$ 
and the algebraic structure on $S$, via $[\Phi]$ and $u \colon \cM \rightarrow \Gamma\backslash\cM$ . Consider the commutative diagram
in the category of complex analytic spaces
\[
\xymatrix{
S^\an \times \cM^\vee & S^\an \times \cM \ar@{_(->}[l] & S^\an
\times_{\Gamma\backslash\cM}\cM \ar@{_(->}[l] \ar[r]^>>>>>>{p_{\cM}} \ar[d]_-{p_S} \pullbackcorner & \cM \ar[d]^{u} \\
& & S^\an \ar[r]_{[\Phi]} & \Gamma \backslash \cM
}.
\]
We prove the following result, conjectured in \cite[Conj.~7.5]{Klingler-conj-} (we refer to \Cref{DefnMTDomain} for the
definition of weak Mumford--Tate subdomains of $\cM$):
\begin{thm}\label{ThmAS}
Let $\mathscr{Z}$ be a complex analytic irreducible subset of
$S^\an
\times_{\Gamma\backslash\cM}\cM$. Then
\begin{equation} \label{inequality}
\dim \mathscr{Z}^\Zar - \dim \mathscr{Z} \ge \dim p_{\cM}(\mathscr{Z})^{\mathrm{ws}},
\end{equation}
where $\mathscr{Z}^\Zar$ denotes the Zariski closure of $\mathscr{Z}$
in $S \times \cM^\vee$, and $p_{\cM}(\mathscr{Z})^{\mathrm{ws}}$ is
the smallest weak Mumford--Tate subdomain of $\cM$ containing $p_{\cM}(\mathscr{Z})$.
\end{thm}
\noindent 
In the course of the proof, we also explain how to construct $p_{\cM}(\mathscr{Z})^{\mathrm{ws}}$. Let $S'$ be the Zariski closure of $p_S(\mathscr{Z})$. Let $N$ be the connected algebraic monodromy group of $(\V_\ZZ,W_\bullet,\cF^\bullet)|_{S'} \rightarrow
S'^\an$. Then $p_{\cM}(\mathscr{Z})^{\mathrm{ws}}$ is the $N(\R)^+\cR_u(N)(\C)$-orbit of any point $\tilde{z} \in p_{\cM}(\mathscr{Z})$, where $\cR_u(N)$ is the unipotent radical of $N$; see \Cref{RmkSmallestWMT}.

\medskip
The idea of functional
transcendence statements related to Hodge theory first appeared in the
context of Shimura varieties, where 
$[\Phi]$ is the identity. Motivated by 
Pila's pioneer work \cite{PilaO-minimality-an} on 
the Andr\'e--Oort 
conjecture for copies of moduli curves, the Ax--Lindemann conjecture (a
special case of the Ax--Schanuel conjecture) was proved for various cases in \cite{PilaAbelianSurfaces, UllmoThe-Hyperbolic-, PilaAxLindemannAg} and ultimately for all pure Shimura varieties in 
\cite{KlinglerThe-Hyperbolic-}; this was extended to mixed Shimura varieties in \cite{GaoTowards-the-And}. 
After the proof of  the Andr\'e--Oort
conjecture \cite{TsimermanA-proof-of-the-} (see \cite{GaoAbout-the-mixed} for mixed Shimura varieties), and in order to attack the more general Zilber--Pink
conjecture, \Cref{ThmAS} was proved for copies of moduli curves in \cite{PilaAx-Schanuel-for} and for any pure Shimura
variety in \cite{MokAx-Schanuel-for}; this was extended 
to mixed Shimura varieties of Kuga type in \cite{GaoAxSchanuel}. In
\cite[Conj.~7.5]{Klingler-conj-} the second author suggested that
these functional transcendence statements should hold much more
generally for all admissible, graded
polarizable, integral variation of mixed Hodge structures over a smooth complex
quasi-projective variety $S$ and formulated 
\Cref{ThmAS}; this was proved in \cite{BTAS} if the variation of Hodge structures in question is pure.

All these works have been important ingredients in the proofs 
of various diophantine results: the Andr\'e--Oort conjecture for 
mixed Shimura varieties, results in the direction of the more general Zilber--Pink conjecture \cite{DawRenAppOfAS}, 
use of \cite{MokAx-Schanuel-for} to prove the submersivity of the Betti map in \cite{ACZBetti}, 
use of
\cite{BTAS} for Shafarevich type results in
\cite{LawrenceVenkatesh, LawrenceSawin}, use of \cite{GaoAxSchanuel} to fully study the Betti rank  in \cite{GaoBettiRank} which eventually was applied to prove a rather uniform bound on the number of rational points on curves \cite{DGHUnifML}. Hast \cite{Hast} recently proved a transcendence property of the unipotent Albanese map assuming \Cref{ThmAS}. 
We expect \Cref{ThmAS} to have more
applications in diophantine geometry, for instance in direction of the general
Hodge-theoretical atypical intersection conjecture \cite[Conj.
1.9]{Klingler-conj-} and its special case \cite[Conj.
5.2]{Klingler-conj-}.

\medskip
The strategy for proving \Cref{ThmAS} is similar in spirit to previous works, in particular \cite{BTAS},
\cite{MokAx-Schanuel-for} and \cite{GaoAxSchanuel}. However its implementation in the mixed non-Shimura 
case contains serious new difficulties.

\medskip
For readers' convenience, we start the paper by recalling basic knowledge on variations of mixed Hodge structures and mixed Mumford--Tate domains  in  \Cref{mhs}, \Cref{vmhs}, 
\Cref{SectionMHD} and \Cref{SectionMTDomains}. Unlike for the pure or the Shimura case, references to some of the results recalled hereby are not easy to find. We also give proofs in these sections and \Cref{Appendix} to some results which are surely known to experts but whose proofs we cannot find in existing references. For example, mixed Mumford--Tate domains are complex spaces and are stable under intersection; as an upshot, the classifying space $\cM$ in \Cref{ThmAS} can be replaced by a suitable mixed Mumford--Tate domain $\cD$. 
We also use mixed Hodge data developed in \cite{Klingler-conj-} to prove that we are able to take quotients by normal groups in the category of mixed Mumford--Tate domains, and each such quotient is a holomorphic map. All these results are fundamental to the proof of \Cref{ThmAS}. In fact, with these preparations, we can prove a particular case of \Cref{ThmAS}, called the \textit{logarithmic Ax theorem}, in \Cref{SectionPeriodMapLogAx}.

\medskip
Another formalism we do for our strategy is the fibered structure of mixed Mumford--Tate domains. We also need to discuss the real points of mixed Mumford--Tate domains; they correspond to mixed Hodge structures split over $\R$. This is done in \Cref{SectionFiberedStruRealPoints}.

\medskip
Then we move on to prove \Cref{ThmAS}. We start by some d\'{e}vissages
in \Cref{preparation}, and reduce to the case where the projection of $\mathscr{Z}$ in
$S$ is Zariski-dense in $S$ and that $\mathscr{Z}$ is an irreducible
component of the intersection of its Zariski-closure with $\Delta$:
see \Cref{LemmaSecondDevissage}. In order to obtain a better
group theoretical control of $\mathscr{Z}$, we also replace the classifying space $\cM$ by its refinement $\cD 
$, the mixed Mumford--Tate domain associated to the generic
Mumford--Tate group $P$ of the variation $(\V_\ZZ,W_\bullet,\cF^\bullet)$. 

\medskip
The first step in the proof of \Cref{ThmAS} consists of proving that the inequality~(\ref{inequality})
holds true if the $\Q$-stabilizer of $\mathscr{Z}^\Zar$ 
(for the
action of $P$ on the second factor of $S^\an \times \cD$), 
denoted by $H_{\mathscr{Z}^\Zar}$,  
is
zero dimensional; see \Cref{PropBignessStab}. To do so we use
o-minimal geometry (more precisely the result of \cite{BBKT}
generalizing \cite{BKT} saying that mixed period maps are definable in
some o-minimal structure, and the celebrated Pila-Wilkie theorem \cite[3.6]{PilaO-minimality-an}) to prove a counting result \Cref{ThmCounting}. 

More precisely, take a suitable semi-algebraic fundamental set $\mathfrak{F}$ for $\cD \rightarrow \Gamma\backslash \cD$. As in all proofs of Ax--Schanuel type transcendence results via o-minimality, we start by constructing a definable subset $\Theta$ of $P(\R)$ which contains all integer elements $\gamma \in \Gamma$ such that $\gamma(S\times \mathfrak{F}) \cap \mathscr{Z} \not = \emptyset$. We wish to prove that $\Theta$ contains semi-algebraic curves with arbitrarily many integer elements; this will yield the non-triviality of $H_{\mathscr{Z}^\Zar}$ unless \eqref{inequality} already holds true by induction. The Pila-Wilkie theorem then reduces the question to showing that the number of elements in $\Gamma \cap \Theta$ of height at most $T$ grows at least polynomially in $T$. The latter is precisely \Cref{ThmCounting}.

\medskip
The first main new difficulty lies in the proof of this counting result. It occupies the full section \Cref{SectionBigness} and is
quite technical. While in the pure case it follows from an explicit
description of the semi-algebraic fundamental set $\mathfrak{F}$ for $\Gamma$
in terms of Siegel sets furnished by reduction 
theory and from the non-positive curvature in the horizontal direction
for pure Mumford--Tate domains (see \cite{BTAS}), in the mixed case we
have only an implicit knowledge of
$\mathfrak{F}$:  its construction in \cite{BBKT} relies fundamentally on the
rather mysterious retraction of $\cD$ on its subvariety $\cD_\R$ of real
split mixed Hodge structures furnished by the
$\mathfrak{sl}_2$-splitting of mixed Hodge structures. Instead, we use
the natural fibered structure 
\begin{equation}\label{EqSuccessiveLifting}
 \cD= \cD_m \to 
\cD_{m-1} \to \cdots \to \cD_0
\end{equation}
 of mixed Mumford--Tate domains 
associated to the weight filtration of the variation of Hodge
structures. Each step is a vector bundle. Considering the successive
projections $\mathscr{Z}_k$ of $\mathscr{Z}$ to the
storeys $S \times \cD_{k}$, we proceed as follows:

- assuming that the required estimate holds
for $\mathscr{Z}_k$ we prove that we can  ``lift'' this estimate to
$\mathscr{Z}_{k+1}$: see \Cref{PropCounting} and
\Cref{SubsectionEstimateLifting}. As in \cite{GaoAxSchanuel},
there are two cases to consider for this lifting process, namely the ``horizontal'' case 
 \Cref{LemmaHorizontalGrowth} and the ``vertical'' case  \Cref{LemmaLiftingCaseVertical}. 

- we initiate the process at the
smallest integer $k_0$ such that the projection
of $\mathscr{Z}$ to $\cD_{k_{0}}$ is not a point. If $k_0= 0$ the required estimate
follows from \cite{BTAS} as $\cD_0$ is a pure Mumford--Tate
domain. 
On the other hand there is some non-trivial work to be done if $k_0
>0$ (the unipotent case, or equivalently when the maximal pure quotient of the variation is
constant): see \Cref{SubsectionCountingBaseStep}, more precisely \Cref{LemmaCountingBaseStep}.

\medskip
The second step in the proof of \Cref{ThmAS} consists of dealing with the case where the group
$H_{\mathscr{Z}^\Zar}$ is positive dimensional. In that case one wants
to reduce to the first step by working in the quotient Mumford--Tate
domain $\cD/H_{\mathscr{Z}^\Zar}$. Such a quotient exists as a
Mumford--Tate domain only if the group $H_{\mathscr{Z}^\Zar}$ is normal in the generic
Mumford--Tate group $P$. Following the guideline of \cite{MokAx-Schanuel-for}, we prove in \Cref{normality1} that $H_{\mathscr{Z}^\Zar}$ is normal in the algebraic monodromy group of this variation of mixed Hodge structures. While this immediately implies that $H_{\mathscr{Z}^\Zar}$ is normal in $P$ in the pure case, it
turns out to be more subtle in the mixed case. We solve this problem in \Cref{normality2}, by doing an intermediate quotient $(\cR_u(P)(\mathbb{Q})\cap \Gamma)\backslash \cD$, applying Pila--Wilkie in the unipotent part, and analyzing the unipotent part of the $H_{\mathscr{Z}^\Zar}$ by passing to a suitable quotient space which \textit{a priori} is only a real manifold. This guideline was executed for the universal abelian variety in \cite[$\mathsection$6.3]{GaoAxSchanuel}. A key new input at this step compared with \cite[$\mathsection$6.3]{GaoAxSchanuel}, as for the lifting process of point counting  from the first step explained above, is the retraction map $\cD \rightarrow \cD_{\R}$ from \cite{BBKT} obtained by the $\mathfrak{sl}_2$-splitting.


\medskip
Right before the first version of this paper was publicized, we received a preprint \cite{KennethChiuAS} from Chiu independently proving the same result. Both papers use extensively o-minimality and the Pila--Wilkie counting theorem,  rely on the estimate results for the pure case of Bakker--Tsimerman \cite{BTAS}, use the retraction map $\cD \rightarrow \cD_{\R}$, and use the idea of separating the ``horizontal'' and ``vertical'' cases for point counting as was done in \cite{GaoAxSchanuel}. 

The major differences of the two papers lie in the specific treatments of the two steps of the proof of \Cref{ThmAS}. For the first step, we obtain the desired point counting result by successive liftings explained in the paragraph containing \eqref{EqSuccessiveLifting}, while Chiu separate the unipotent part from the semi-simple part at the beginning. For the second step, we work in the Mumford--Tate domain $\cD$ and prove that the $\Q$-stabilizer $H_{\mathscr{Z}^{\mathrm{Zar}}}$ of $\mathscr{Z}^\Zar$ is positive dimensional unless $\mathscr{Z}$ takes some particular form and that \Cref{ThmAS} easily holds true, and then proceed to prove the normality of $H_{\mathscr{Z}^{\mathrm{Zar}}}$ in the Mumford--Tate group $P$ in \Cref{normality2} in order to do the quotient $P/H_{\mathscr{Z}^{\Zar}}$. Chiu works in the weak Mumford--Tate domain corresponding to a suitable normal subgroup $N$ of $P$ and 
does the estimates directly on $(N/H_{\mathscr{Z}^{\Zar}})(\R)$, and instead of proving the normality of $H_{\mathscr{Z}^\Zar}$ in $P$ he reduces to the case where $\mathscr{Z}$ is contained in one fiber and handles this case in \cite[$\mathsection$8]{KennethChiuAS}. Apart from these, we also include a summary of basic knowledge and results on variations of mixed Hodge structures and mixed Mumford--Tate domains in  \Cref{mhs}, \Cref{vmhs}, 
\Cref{SectionMHD}, \Cref{SectionMTDomains} and \Cref{Appendix}, as the references to some of the results are not easy to find in contrast to the pure  or the Shimura case.

In the end, we would like to point out that our first version had a serious (Hodge-theoretic) mistake in the previous \Cref{normality2}  while Chiu's proof was correct. To fix this mistake
, we had to go back to the argument of the first author's \cite[$\mathsection$6.3]{GaoAxSchanuel} and use again the retraction map $\cD \rightarrow \cD_{\R}$, and this makes our current \Cref{normality2} similar to Chiu's treatment in \cite[$\mathsection$8]{KennethChiuAS}.


\subsection*{Acknowledgements} ZG has received
fundings from the European Research Council (ERC) under the
European Union's Horizon 2020 research and innovation programme (grant
agreement n$^\circ$ 945714). 
BK has received
fundings from the European Research Council (ERC) under the
European Union's Horizon 2020 research and innovation programme (grant
agreement n$^\circ$ 101020009). 
The authors would like to thank Jacob Tsimerman for having pointed out a mistake in a previous version of \Cref{normality2} and for explaining to us the similarities of our new \Cref{normality2} and Chiu's \cite[$\mathsection$8]{KennethChiuAS}.


\section{Mixed Hodge structures, classifying space, and Mumford--Tate
  domains}  \label{mhs}

\subsection{Mixed Hodge structure}
In this subsection we recall some definitions and properties of $\Q$-mixed Hodge structures.
\begin{defn} Let $V$ be a finite dimensional $\Q$-vector space and $V_\C: =  V \otimes_\Q \C$ its complexification.
\begin{enumerate}
\item[(i)] A $\Q$-pure Hodge structure on $V$ of weight $n$ is a decreasing filtration $F^\bullet$ (the {\em Hodge filtration}) on $V_\C$ such that $V_\C =  F^p V_\C \oplus \bar{F^{n+1-p} V_{\C}}$ for all $p \in \Z$.
\item[(ii)] A $\Q$-mixed Hodge structure on $V$ consists of two filtrations, an increasing filtration $W_\bullet$ on $V$ (the {\em weight filtration}) and a decreasing filtration $F^\bullet$ on $V_{\C}$ (the {\em Hodge filtration}) such that for each $k \in \Z$ the $\Q$-vector space $\mathrm{Gr}_k^W V = W_k/W_{k-1}$ is a pure Hodge structure of weight $k$ for the filtration on $\mathrm{Gr}_k^W V \otimes_\Q \C$ deduced from $F^\bullet$. 
\end{enumerate}

The numbers
$ h^{p,q}(V) = \dim_{\C} 
F^p \mathrm{Gr}^W_{p+q}(V_{\C}) / F^{p+1} \mathrm{Gr}^W_{p+q}(V_{\C})$
are called the \textit{Hodge numbers} of $(V, W_\bullet, F^\bullet)$.
\end{defn}

$\QQ$-mixed Hodge structures, defined in terms of two filtrations, can be equivalently described in terms of {\em bigradings}. This is classical in the pure case, where a weight~$n$ $\QQ$-pure Hodge structure on $V$ is equivalently given by a direct sum decomposition $V_{\C} = \oplus_{p+q=n} V^{p,q}$ (the {\em Hodge decomposition}) into $\C$-vector spaces, such that the complex conjugate $\bar{V^{q,p}}$ coincides with $V^{p,q}$ for all $p, q \in \Z$ with $p+q=n$. The relation between the Hodge filtration and the Hodge decomposition is given by $F^pV_{\C} = \oplus_{p'\ge p}V^{p',n-p'}$. In the general mixed case Deligne \cite[1.2.8]{DeligneHodgeII}
proved the following: 
\begin{prop}\label{PropBiGradingMHS}
A $\Q$-mixed Hodge structure on $V$ is the datum of a bigrading
\begin{equation}
  V_{\C} = \bigoplus_{p,q \in \Z}I^{p,q}
\end{equation}
satisfying that each complex vector subspace $W_k V_{\C} = \bigoplus_{p+q\le k}I^{p,q}$ of $V_\C$ is defined  over $\QQ$ and 
\begin{equation}
I^{p,q} \equiv \bar{I^{q,p}} \bmod \bigoplus_{r<p,s<q}I^{r,s}.
\end{equation}
The Hodge filtration is then defined by $F^pV_{\C} = \bigoplus_{r \ge p}I^{r,q}$.
\end{prop}

We will use  a third, more group-theoretic, point of view on $\Q$-mixed Hodge structures. Let $\S = \mathrm{Res}_{\C/\R}\G_{\mathrm{m},\C}$ be the Deligne torus, this is the real algebraic group such that $\S(\R) = \C^*$ and $\S(\C) = \C^* \times \C^*$, with the action of the complex conjugation twisted by the automorphism that interchanges the two factors. The character group of $\S$, denoted by $X_*(\S)$, identifies with $\Z \oplus \Z$ under
$$\begin{array}{lll}
 \Z \oplus \Z &\xrightarrow{\sim}  &  X_*(\S) \\
 (p,q) &\mapsto & \big(z \in \S(\R) = \C^* \mapsto z^{-p}\bar{z}^{-q} \in \C^* \big).
  \end{array}
  $$
\noindent Given a $\Q$-vector space $V$ a bigrading $V_{\C} = \oplus_{p, q \in \ZZ} I^{p,q}$ is thus equivalent to a homomorphism $h \colon \S_{\C} \rightarrow \mathrm{GL}(V_{\C})$. In particular we deduce from the paragraph above that any mixed Hodge structure on $V$ defines a homomorphism $h \colon \S_{\C} \rightarrow \mathrm{GL}(V_{\C})$. In \cite{PinkThesis} Pink identified the conditions such a homomorphism has to satisfy to define a mixed Hodge structure on $V$:

\begin{prop} \cite[1.4 and 1.5]{PinkThesis} \label{propPink}
Let $V$ be a finite dimensional $\Q$-vector space. A morphism $h \colon \S_{\C} \rightarrow \mathrm{GL}(V_{\C})$  
defines a MHS on $V$ if and only if there exists a connected $\Q$-algebraic subgroup $P \subset \mathrm{GL}(V)$ such that $h$ factors through $P_{\C}$ and which satisfies the following conditions: 
\begin{itemize}
\item[(i)] The composite $\S_\C \stackrel{h}{\to} P_\C \to (P/W_{-1})_\C$ is
  defined over $\R$, where $W_{-1}$ denotes the unipotent radical of $P$. Call this composite $\bar{h}$.
\item[(ii)] The composite $\G_{m, \R} \stackrel{w}{\to} \S
  \stackrel{\bar{h}}{\to} (P/W_{-1})_\R$ is a cocharacter of the center of $(P/W_{-1})_\R$ defined over $\Q$.
\item[(iii)] The weight filtration on
  $\lie P$ defined by $\Ad_P \circ h$ satisfies $W_0 \lie P = \lie P$ and $W_{-1}(\lie P) = \lie W_{-1}$.
\end{itemize}
\end{prop}

\noindent
If $h \in \cM$ let us define the {\em Mumford--Tate group} $\MT(h)$ of the $\Q$-mixed Hodge structure $(M, h)$ as the smallest $\Q$-subgroup of $\GL(V)$ whose complexification contains $h(\S_\C)$. One easily checks that the groups $P$ satisfying the conditions of \Cref{propPink} are precisely the ones containing $\MT(h)$.

\medskip
We finish this subsection by recalling the definition of polarizations.
\begin{defn}
Let $(V,W_{\bullet},F^{\bullet})$ be a $\Q$-mixed Hodge structure. A {\em (graded) polarization} is a collection of non-degenerate $(-1)^k$-symmetric bilinear forms
\[
Q_k \colon \mathrm{Gr}_k^W(V) \otimes \mathrm{Gr}_k^W(V) \rightarrow \Q
\]
such that
\begin{enumerate}
\item[(i)] $Q_k(F^p\mathrm{Gr}^W_kV_\C , F^{k-p+1}\mathrm{Gr}^W_k V_\C) = 0$ for each $k$ (first Riemann bilinear relation);
\item[(ii)] the Hermitian form on $\mathrm{Gr}_k^W(V)_{\C}$ given by $Q_k(Cu,\bar{v})$ is positive-definite, where $C$ is the Weil operator ($C|_{I^{p,q}} = i^{p-q}$ for all $p,q$).
\end{enumerate}
\end{defn}

\noindent
One easily checks that the Mumford-Tate group of a polarizable pure $\Q$-Hodge structure is reductive.

\subsection{Classifying space}\label{SubsectionClassifyingSpace}
In this subsection, we discuss the classifying space of all $\Q$-mixed Hodge structures with given weight filtration, graded polarization and Hodge numbers.

Let $V$ be a finite dimensional $\Q$-vector space, endowed with the following additional data:
\begin{enumerate}
\item[(i)] a finite increasing filtration $W_\bullet$ of $V$;
\item[(ii)] a collection of non-degenerate $(-1)^k$-symmetric bilinear forms
$$
Q_k \colon \mathrm{Gr}_k^W(V) \otimes \mathrm{Gr}_k^W(V) \rightarrow \Q\;\; ;$$
\item[(iii)] a partition $\{h^{p,q}\}_{p, q \in \Z}$ of $\dim V_{\C}$ into non-negative integers.
\end{enumerate}

Given these data, one forms the classifying space $\cM$ parametrizing $\Q$-mixed Hodge structures $(V, W_\bullet, F^\bullet)$ with the following properties:
\begin{enumerate}
\item the $(p,q)$-constituent $V^{p,q}:= \Gr^p_F \Gr^W_{p+q} V_\C$ has complex dimension $h^{p,q}$;
\item $Q_k(F^p\mathrm{Gr}^W_kV_\C , F^{k-p+1}\mathrm{Gr}^W_k V_\C) = 0$ for each $k$ (first Riemann bilinear relation);
\item $(V, W_\bullet, F^\bullet)$ is graded-polarized by $Q_k$.
\end{enumerate}

Let us summarize the construction and basic properties of $\cM$; see \cite{Kaplan}, \cite[below (3.7) to Lemma~3.9]{Pearlstein2000} 
for more details.
First one defines the complex algebraic variety $\cM^\vee$ parametrizing mixed Hodge structures satisfying only the conditions (1) and (2) above (see \cite[Lem.~3.8]{Pearlstein2000}). This is a homogeneous space under $P^{\cM}(\C)$, where $P^{\cM}$ is the $\Q$-algebraic group defined as follows: for any $\Q$-algebra $R$, 
\begin{equation}\label{EquationQGroupPM}
P^{\cM}(R) := \{g \in \mathrm{GL}(V_R) : g(W_k) \subseteq W_k\text{ and }\mathrm{Gr}^W_k(g) \in \mathrm{Aut}_{R}(Q_k)\; \text{for all} \, k \in \Z\}.
\end{equation}
The classifying space $\cM$ is defined as the real semi-algebraic open subset of $\cM^\vee$ consisting of mixed Hodge structures which satisfy moreover condition~(3) above (see \cite[Lem.~3.9 and above]{Pearlstein2000}). The fact that $\cM$ is open in $\cM^\vee$ endows $\cM$ with a natural complex analytic structure. The real semi-algebraic group
\begin{equation} \label{l1}
\{g \in P^{\cM}(\C): \mathrm{Gr}^W_k(g) \in \mathrm{Aut}_{\R}(Q_k)\; \text{for all }k \in \Z\}
\end{equation}
identifies with $P^{\cM}(\R)^+W^{\cM}_{-1}(\C)$, where $W^{\cM}_{-1}$ is the unipotent radical of $P^{\cM}$, see \cite[Remark below Lem.~3.9]{Pearlstein2000}. It acts transitively on $\cM$. 


\subsection{Adjoint Hodge structure}
For each $h \in \cM$ \Cref{propPink} defines a natural $\Q$-mixed Hodge structure on $\lie P^{\cM}$ via $\mathrm{Ad}^{\cM} \circ h \colon \S_{\C} \rightarrow P^{\cM}_{\C} \rightarrow \mathrm{GL}(\lie P^{\cM})_{\C}$: the \textit{adjoint Hodge structure} associated with $h$. One easily checks that the corresponding weight filtration and graded polarization are independent of $h$. Indeed the weight filtration $W_\bullet$ on $\lie P^{\cM} \subseteq \mathrm{End}(V) = V \otimes V^\vee$ is the one deduced from the weight filtration $W_\bullet$ on $V$. Similarly for the graded-polarization.


\subsection{(Weak) Mumford--Tate domains}
\Cref{propPink} suggests to attack the problem of classifying mixed Hodge structures by rather considering mixed Hodge structures with prescribed Mumford-Tate group. This leads abstractly to the notion of mixed Hodge data, see \Cref{MHD}; and geometrically to the notion of (weak) Mumford-Tate domain refining the classifying space $\cM$.

\begin{defn}\label{DefnMTDomain}

\begin{enumerate}
\item[(i)] A subset $\cD$ of the classifying space $\cM$ is called a Mumford--Tate domain if there exists an element $h \in \cD$ such that $\cD = P(\R)^+W_{-1}(\C) h$, where $P = \MT(h)$ and $W_{-1} = \cR_u(P)$ is the unipotent radical of $P$.
\item[(ii)] A subset $\cD$ of the classifying space $\cM$ is called a weak Mumford--Tate domain if there exist an element $h \in \cD$ and a normal subgroup $N$ of $P=\MT(h)$ such that $\cD = N(\R)^+\cR_u(N)(\C) h$, where $\cR_u(N)$ is the unipotent radical of $N$.
\end{enumerate}
\end{defn}

In the definition, as $N \lhd P$, we have $\cR_u(N) = W_{-1}\cap N$. 
One easily checks that $\cM$ is a Mumford-Tate domain in itself, for $P = P^\cM$. A closer look at the geometry of general Mumford--Tate domains is given in \Cref{Appendix2}. In particular we will prove the following results (well-known in the pure case):
\begin{prop}\label{PropMTDomainComplex}
Every weak Mumford--Tate domain in $\cM$ is a complex analytic subspace of $\cM$.
\end{prop}

\begin{lemma}\label{LemmaIntersectionMTDomains}
Let $\cD_1$ and $\cD_2$ be Mumford--Tate domains in $\cM$. Then every irreducible component of $\cD_1\cap \cD_2$ is again a Mumford--Tate domain in $\cM$.
\end{lemma}

This lemma has the following immediate corollary.
 \begin{cor}\label{LemmaSmallestMTDomain}
Let  $\cZ$ be a complex analytic irreducible subset of $\cM$. Then there exists a smallest Mumford--Tate domain, denoted by $\cZ^{\mathrm{sp}}$ and called the {\em special closure} of $\cZ$, which contains $\cZ$.
 \end{cor}

We close this subsection with some discussion on the \textit{generic Mumford--Tate group} of a complex analytic irreducible subvariety of $\cM$. In particular the discussion applies to weak Mumford--Tate domains.
The trivial local system $\V= \cM \times V$ underlies a natural family of mixed Hodge structures: for each $h \in \cM$ the triple $(V, (W_\bullet)_h, (\cF^\bullet)_h)$ is a mixed $\Q$-Hodge structure. For any complex analytic irreducible subset $\cZ$ of $\cM$, the first part of the proof of \cite[$\mathsection$4, Lemma~4]{AndreMumford-Tate-gr} applies: for a very general element $h \in \cZ$, the Mumford--Tate group $P(h)$ does not depend on $h$. Such an $h$ is said to be {\em Hodge--generic} in $\cZ$ and its Mumford-Tate group is called the {\em generic Mumford--Tate group} of $\cZ$. We write $\mathrm{MT}(\cZ)$ to denote the generic Mumford--Tate group of $\cZ$. It satisfies the following property: $\mathrm{MT}(h') < \mathrm{MT}(\cZ)$ for any $h' \in \cZ$.

\begin{lemma}\label{LemmaMTDomainOrbitUnderGenericMTGroup}
Let $\cD= P(\R)^+W_{-1}(\CC)h$ be a Mumford-Tate domain in $\cM$ (thus $h \in \cD$, $P= \MT(h)$ and $W_{-1}$ is the unipotent radical of $P$). Then $P= \MT(\cD)$. 
\end{lemma}
\begin{proof}
By definition of $\MT(\cD)$ the group $P$ is a subgroup of $\MT(\cD)$ . Thus we are reduced to proving the converse inclusion.

Each $h' \in \cD$ is of the form $g h g^{-1}$ for some $g \in P(\R)^+W_{-1}(\C)$, and hence the homomorphism $h' = \colon \S_{\C} \rightarrow \mathrm{GL}(V_{\C})$ factors through $g P_{\C} g^{-1} = P_{\C}$. This implies that $\mathrm{MT}(h') < P$ for all $h' \in \cD$. Looking at a Hodge generic point $h'$ we are done.
\end{proof}

The following lemma, whose proof is given \Cref{Appendix2}, is useful to determine when an orbit is a Mumford--Tate domain. 
\begin{lemma}\label{LemmaSameOrbitUnderNormalSubgroup}
Let $P$ be a $\Q$-subgroup of $\mathrm{GL}(V)$ with $W_{-1} = \cR_u(P)$ and let $\cD$ be a $P(\R)^+W_{-1}(\C)$-orbit in $\cM$. If some $h \in \cD$ satisfies that $h \colon \S_{\C} \rightarrow \mathrm{GL}(V_{\C})$ factors through $P_{\C}$  then $\cD$ is a Mumford--Tate domain and $\mathrm{MT}(\cD) \lhd P$.
\end{lemma}


\section{Variation of mixed Hodge structures} \label{vmhs}

Let $f \colon X \to S$ be a morphism of algebraic varieties. If $f$ satisfies a sharp notion of topological local constancy (suffice it to say here it is automatically satisfied if $f$ is proper smooth, and is true over a Zariski-open subset of $S$ for any morphism of varieties), then $f$ gives rise to a family of mixed Hodge structures (pure when $f$ is proper smooth) on $H^n(X_s, \Q)$, as $s$ varies over $S^\an$, subject to certain rules. This leads to the notion of a (graded-polarizable) variation of mixed Hodge structures, which we now recall:

\begin{defn} Let $S$ be a connected complex manifold. A \textit{variation of mixed Hodge structures} (abbreviated VMHS) on $S$ is a triple $(\V_\ZZ,W_\bullet,\cF^\bullet)$ consisting of:
\begin{enumerate}
\item[(i)] a local system $\V_\ZZ$ of free $\Z$-modules of finite rank on $S$;
\item[(ii)] a finite increasing filtration $W_\bullet$ of the local system $\V:= \V_\ZZ\otimes_{\Z_S}\Q_S$ by local subsystems (weight filtration);
\item[(iii)] a finite decreasing filtration $\cF^\bullet$ of the holomorphic vector bundle $\cV := \V_\ZZ \otimes_{\Z}\cO_S$ by holomorphic subbundles (Hodge filtration),
\end{enumerate}
satisfying the following conditions:
\begin{enumerate}
\item for each $s \in S$, the triple $(\V_{s}, W_\bullet(s), \cF^\bullet(s))$ is a mixed Hodge structure;
\item the connection $\nabla \colon \cV \rightarrow \cV \otimes_{\cO_S}\Omega_S^1$ whose sheaf of horizontal sections is $\V_{\C}:= \V \otimes_\QQ \CC$ satisfies the Griffiths' transversality condition
\[
\nabla(\cF^p) \subseteq \cF^{p-1}\otimes \Omega_S^1.
\]
\end{enumerate}
\end{defn}

\begin{defn}
A VMHS $(\V_\ZZ,W_\bullet,\cF^\bullet)$ on $S$ is called {\em graded-polarizable} if the induced variations of pure $\Q$-Hodge structures (VHS) $\mathrm{Gr}^W_k \V$, $k \in \ZZ$, are all polarizable, \textit{i.e.} for each $k \in \ZZ$ there exists a morphism of local systems
\[
\cQ_k \colon \mathrm{Gr}^W_k \V \otimes \mathrm{Gr}^W_k \V \rightarrow \Q_S
\]
inducing on each fiber a polarization of the corresponding $\Q$-Hodge structure of weight~$k$.
\end{defn}

{\em From now on all VMHS are assumed to be graded-polarizable.}

\subsection{Mumford-Tate group and monodromy group}\label{SubsectionMTgroupMongroup}
Let $S$ be a connected complex manifold and $(\V_\ZZ,W_\bullet,\cF^\bullet)$ a VMHS on $S$. The pull-back $\pi^*\V_\ZZ$ of $\V_\ZZ$ along the universal covering map $\pi \colon \tilde{S} \rightarrow S$ is canonically trivialized: $\pi^*\V_\ZZ \cong \tilde{S} \times V_\ZZ$, with $V_\ZZ= H^0(\tilde{S}, \pi^*\V_\ZZ)$. 

For $s \in S$, we denote by $\MT_s \subseteq \mathrm{GL}(\V_s)$ the Mumford--Tate group of the Hodge structure $\V_s$ and by $H_s^{\mathrm{mon}} \subseteq \mathrm{GL}(\V_s)$ the \textit{connected algebraic monodromy group} at $s$,
that is the connected component of identity of the smallest $\QQ$-algebraic subgroup of $\mathrm{GL}(\V_s)$ containing the image under monodromy of $\pi_1(S, s)$.

By definition the algebraic monodromy group $H_s^{\mathrm{mon}}$ is locally constant on $S$.
By \cite[$\mathsection$4, Lemma~4]{AndreMumford-Tate-gr}, following \cite[$\mathsection$ 7.5]{Del} in the pure case, the Mumford-Tate group $\MT_s \subset \GL(\V_{s})$ is locally constant on $S^\circ = S \setminus \Sigma$ where $\Sigma$ denotes a meager subset of $S$; and $H_s^{\mathrm{mon}} $ is a subgroup of $P_s$ for all $s \in S^\circ$ as $(\V_{\Z},W_\bullet,\cF^\bullet)$ is graded-polarizable. We call $S^{\circ}$ the \textit{Hodge-generic locus}. For $s \in S^\circ$ the group $\mathrm{MT}_{s_{0}}$ is called the \textit{generic Mumford--Tate group} $\MT(S)$ of  $(\V_{\Z},W_\bullet,\cF^\bullet)$.

\subsection{Admissible VMHS}

\textit{Admissible} VMHSs are the ones with good asymptotic properties. The concept was introduced by Steenbrick--Zucker \cite[Properties~3.13]{SteenbrinkVariation-of-mi} on a curve and Kashiwara \cite[1.8 and 1.9]{KashiwaraA-study-of-vari} in general. All VMHSs which arise from geometry are admissible \cite{EZ} and all VHSs are automatically admissible. We recall briefly the definition.

\begin{defn}[admissible VMHS]
A VMHS $(\V_{\Z}, W_\bullet, \cF^\bullet)$ over the punctured unit disc $\Delta^*$ is called admissible if
\begin{itemize}
\item[(i)] it is graded-polarizable;
\item[(ii)] the monodromy $T$ around zero is quasi-unipotent and the logarithm $N$ of the unipotent part of $T$ admits a weight filtration $M(N,W_\bullet)$ relative to $W_\bullet$ (see \cite[$\mathsection$3.1]{KashiwaraA-study-of-vari});
\item[(iii)] Let $\ol{\cV}$, resp. $W_k \ol{\cV}$, be Deligne's canonical extension of $\cV$, resp. of $\mathcal{O}_{\Delta^*} \otimes_\QQ W_k \V$, to $\Delta$. The Hodge filtration $\cF^\bullet$ extends to a locally free filtration $\ol{\cF}^\bullet$ of $\ol{\cV}$ such that $\mathrm{Gr}^p_{\ol{\cF}} \mathrm{Gr}_k^W \ol{\cV}$ is locally free.
\end{itemize}

Let $S$ be a connected complex manifold compactifiable by a compact complex analytic space $\bar{S}$. A graded-polarizable variation of mixed Hodge structure $(\V_{\Z},W_\bullet,\cF^\bullet)$ on $S$ is said admissible with respect to $\bar{S}$ if for every holomorphic map $i \colon \Delta \rightarrow \bar{S}$ which maps $\Delta^*$ to $S$, the variation $i^*(\V_{\Z},W_\bullet,\cF^\bullet)$ is admissible.

\end{defn}

Let $S$ be a smooth complex quasi-projective variety. The property for a VHMS on $S^\an$ to be admissible with respect to a smooth projective compactification $\bar{S}^\an$ is easily seen to be independent of the choice of $\bar{S}$. Hence we can and will talk of admissible VMHSs on $S^\an$. {\em From now on, and in order to simplify notations, we will not distinguish between $S$ and $S^\an$, the meaning being clear from the context.}

\medskip
Admissible VMHSs have the following advantage (see Andr\'{e} \cite[$\mathsection$5, Theorem~1]{AndreMumford-Tate-gr}, following \cite[$\mathsection$7.5]{Del} in the pure case):

\begin{thm} (Deligne, Andr\'e) \label{ThmAndreNormal}
Let $(\V_{\Z},W_\bullet,\cF^\bullet)$ be an admissible VMHS over a smooth connected complex quasi-projective variety $S$. Then for any Hodge-generic point $s \in S^\circ$, the connected algebraic monodromy group $H_s^{\mathrm{mon}}$ is a normal subgroup of the derived group $\MT(S)^{\mathrm{der}}$ of the generic Mumford-Tate group of $S$.
\end{thm}


\section{Mixed Hodge data}\label{SectionMHD}
Classifying mixed Hodge structures with prescribed Mumford-Tate group
leads to the formalism of mixed Hodge data introduced
in \cite{Klingler-conj-}, following \cite{PinkThesis} in the
Shimura case. This group theoretical
formalism is useful to relate VMHS and Mumford-Tate domains.

\subsection{Mixed Hodge data} \label{MHD}
\begin{defn}\label{DefnMixedHD}
A connected mixed Hodge datum is a pair $(P,\cX)$, where $P$ is a
connected linear algebraic group over $\Q$ whose unipotent radical we
denote by $W_{-1}$, and $\cX \subseteq
\hom(\S_{\C},P_{\C})$ is a $P(\R)^+W_{-1}(\C)$-conjugacy class
such that one (and then any) $h \in \cX$
satisfies property $(i)$, $(ii)$ and $(iii)$ of \Cref{propPink}.
A morphism $(P, \cX) \to (P', \cX')$ of mixed Hodge data is a morphism
$P \to P'$ of $\Q$-algebraic groups inducing an equivariant map $\cX
\to \cX'$.
\end{defn}

Let $(P,\cX)$ be a mixed Hodge datum.
As a homogeneous space under $P(\R)^+W_{-1}(\C)$, the set
$\cX$ is naturally endowed with a structure of real semi-algebraic
variety. In general however it does not carry any complex structure. 
To relate $\cX$ to complex geometry, let us fix $\rho \colon P
\rightarrow \mathrm{GL}(V)$ a $\Q$-representation. By \Cref{propPink},
for each $h \in
\cX$ the map $\rho \circ h$ endows $V$ with a rational mixed Hodge structure, whose weight filtration and Hodge
numbers are easily seen to be independent of $h \in \cX$. We thus obtain a
$P(\R)^+W_{-1}(\C)$-equivariant map $$\varphi_{\rho}
\colon \cX \rightarrow \cM, $$ for $\cM$ a classifying space as in \Cref{SubsectionClassifyingSpace}.
By \cite[1.7]{PinkThesis}, $\varphi_{\rho}$ factors through a complex
manifold $\cD$ which is independent of $\rho$ \footnote{ Take $\rho$ to be a faithful representation
  of $P$, then we can take $\cD = \varphi_{\rho}(\cX)$.}. From now on
we will just write
\begin{equation}\label{EqMHDToClassifyingSpace}
\varphi \colon \cX \rightarrow \cD
\end{equation}
and call this map {\em the classifying map of the Hodge datum} $(P, \cX)$.
The group $P(\R)^+W_{-1}(\C)$ acts on $\cD$ preserving its complex
structure, and the action of
$W_{-1}(\C)$ on $\cD$ is holomorphic.

\begin{lemma} (\!\!\cite[1.8(b)]{PinkThesis}) \label{below 4.1}
For each $x \in \cD$, the fiber $\varphi^{-1}(x)$ is a principal
homogeneous space under $\exp(F^0_x(\lie W_{-1})_{\C})$.
\end{lemma}

\noindent
In particular
$\varphi$ is an isomorphism in the pure case.

\subsection{Mixed Hodge data and Mumford-Tate domains}
We now relate mixed Hodge data and Mumford-Tate domains by showing
that the complex space $\cD$ in (\ref{EqMHDToClassifyingSpace}) is
a Mumford-Tate domain, and that conversely any Mumford-Tate domain appears as a target
in (\ref{EqMHDToClassifyingSpace}) for some connected mixed Hodge
datum. We start with the case where $\cD=\cM$ is a classifying space.

\begin{lemma}\label{LemmaMHDAssociatedWithClassifyingSpace}
 Let $\cM$ be a classifying space of mixed Hodge structure as in
 \Cref{SubsectionClassifyingSpace}, $P^\cM$ the corresponding group, and 
 $W_{-1}^{\cM} $ its unipotent radical.
 
There exists a mixed Hodge datum $(P^{\cM},\cX^{\cM})$ such that the
classifying map 
(\ref{EqMHDToClassifyingSpace}) for $(P^{\cM},\cX^{\cM})$ reads
$\varphi^{\cM} \colon \cX^{\cM} \rightarrow \cM$. For any $h \in
\cX^{\cM}$, the mixed Hodge structures on $\lie P^{\cM}$ induced by
$h$ and by $\varphi^{\cM}(h)$ coincide. 
\end{lemma}
\begin{proof} Take $h \in \cM$. Then $h \in
  \hom(\S_{\C},P^{\cM}_{\C})$ satisfies conditions $(i)$, $(ii)$ and
  $(iii)$ of \Cref{propPink}. In particular $(P^{\cM}, \cX^{\cM})$ is a mixed
  Hodge datum, where $\cX^{\cM} \colon = P^{\cM}(\R)^+W_{-1}^{\cM}(\C)h
  \subseteq \hom(\S_{\C},P^{\cM}_{\C})$. The existence of
  $\varphi^{\cM}$ follows from \cite[1.7]{PinkThesis}; it is precisely
  the $\varphi$ from (\ref{EqMHDToClassifyingSpace}) for $(P^{\cM},\cX^{\cM})$. 
\end{proof}



\begin{prop}\label{PropMTDomainMHD}
  Let $\cM$ be a classifying space of mixed Hodge structure as in
  \Cref{SubsectionClassifyingSpace}, with associated connected mixed
    Hodge datum $(P^{\cM},\cX^{\cM})$ and classifying map
    $\varphi^{\cM} \colon \cX^{\cM} \rightarrow \cM$ as in \Cref{LemmaMHDAssociatedWithClassifyingSpace}.
\begin{enumerate}
\item[(i)] For each Mumford--Tate domain $\cD$ in $\cM$, there exists
  a sub-mixed Hodge datum $(\mathrm{MT}(\cD),\cX)$ of
  $(P^{\cM},\cX^{\cM})$ such that $\varphi^{\cM}(\cX) = \cD$. Moreover
  $\varphi := \varphi^{\cM}|_{\cX} \colon \cX \rightarrow \cD$ is
  precisely the classifying map (\ref{EqMHDToClassifyingSpace}) for $(\mathrm{MT}(\cD),\cX)$. 
\item[(ii)] Conversely for any sub-mixed Hodge datum $(P,\cX)$ of $(P^{\cM},\cX^{\cM})$, the image $\varphi^{\cM}(\cX)$ is a Mumford--Tate domain in $\cM$ (whose generic Mumford--Tate group is a normal subgroup of $P$).
\end{enumerate}
\end{prop}
\begin{proof} For (i): for simplicity we write $P$ for
  $\mathrm{MT}(\cD)$ and $W_{-1}$ for $\cR_u(P)$. Take a point $x \in
  \cD$; it gives rise to a homomorphism $h_x \colon \S_{\C}
  \rightarrow P_{\C}$. View $h_x \in \cX^{\cM}$, then
  $\varphi^{\cM}(h_x) \in \cD$ by definition of $\varphi^{\cM}$. Let
  $\cX = P(\R)^+W_{-1}(\C)h_x \subset \cM$. As $\varphi^{\cM}$ is
  $P^{\cM}(\R)^+W_{-1}^{\cM}(\C)$-equivariant, we have
  $\varphi^{\cM}(\cX) = P(\R)^+W_{-1}(\C)\varphi^{\cM}(h_x) =
  P(\R)^+W_{-1}(\C)x =  \cD$. By \Cref{propPink} the pair $(P,\cX)$ is
  a mixed Hodge datum and by construction $\varphi = \varphi^{\cM}|_{\cX}$ is precisely the map in (\ref{EqMHDToClassifyingSpace}).

For (ii): Denote by $\cD = \varphi^{\cM}(\cX)$. Then $\cD$ is a $P(\R)^+W_{-1}(\C)$-orbit because the map $\varphi^{\cM}$ is $P^{\cM}(\R)^+W_{-1}^{\cM}(\C)$-equivariant. Moreover for any $x \in \cD$, the corresponding homomorphism $h_x \colon  \S_{\C} \rightarrow \mathrm{GL}(V_{\C})$ factors through $P_{\C}$ by definition of mixed Hodge data. Thus $\cD$ is a Mumford--Tate domain and $\mathrm{MT}(\cD) \lhd P$ by \Cref{LemmaSameOrbitUnderNormalSubgroup}.
\end{proof}

\section{Quotients}\label{SectionMTDomains}

\subsection{Quotient of mixed Hodge datum}
Given a connected mixed Hodge datum $(P,\cX)$ and a normal subgroup $N \lhd P$, the \textit{quotient mixed Hodge datum}
\begin{equation}\label{EqQuotientMHD}
q_N \colon (P,\cX) \rightarrow (P,\cX)/N
\end{equation}
is defined as follows. Given $h \in \cX \subseteq  \hom(\S_{\C},P_{\C})$ we denote by  $\bar{h} \in \hom(\S_{\C},(P/N)_{\C})$ the homomorphism $\S_{\C} \xrightarrow{h} P_{\C} \rightarrow (P/N)_{\C}$. Note that $\cR_u(P/N) = W_{-1}/(W_{-1}\cap N)$.
Denote by $\cX/N = (P/N)(\R)^+(W_{-1}/W_{-1}\cap N)(\C)\bar{h} \subseteq \hom(\S_{\C},(P/N)_{\C})$. One easily checks that $(P,\cX)/N \colon = (P/N, \cX/N)$ is a connected mixed Hodge datum, independent of the choice of $h \in \cX$. The morphism $q_N\colon (P,\cX) \rightarrow (P/N,\cX/N)$ is what we desire. Moreover $q_N \colon \cX \rightarrow \cX/N$ is clearly real algebraic.


\subsection{Quotient of Mumford--Tate domains}

Next we prove that Mumford--Tate domains are stable under taking quotients. This operation is important to understand the structure of Mumford--Tate domains.

Let $V_{\Z}$ be a free finite rank $\Z$-module and $V:= V_{\Z} \otimes_{\Z}\Q$ be the associated $\Q$-vector space. Let $\cM$ be the classifying space of certain polarized mixed Hodge structures and let $P^{\cM}$ be the $\Q$-group, both from \Cref{SubsectionClassifyingSpace}.

\begin{prop}\label{PropQuotientMT}
Let $\cD$ be a Mumford--Tate domain in $\cM$ with $P = \mathrm{MT}(\cD)$, and let $(P,\cX)$ and $\varphi \colon \cX \rightarrow \cD$ be as in (\ref{PropMTDomainMHD})(i). Let $N$ be a normal subgroup of $P$. Then there exists a quotient $p_N \colon \cD \rightarrow \cD/N$, in the category of complex varieties, such that
\begin{enumerate}
\item[(i)] $\cD/N$ is a Mumford--Tate domain in some classifying space of mixed Hodge structures, and $\mathrm{MT}(\cD/N) = P/N$.
\item[(ii)] Each fiber of $p_N$ is an $N(\R)^+(W_{-1}\cap N)(\C)$-orbit, where $W_{-1} = \cR_u(P)$.
\item[(iii)] For the quotient mixed Hodge datum $q_N \colon (P,\cX) \rightarrow (P/N,\cX/N)$ defined in \eqref{EqQuotientMHD}, the classifying map (\ref{EqMHDToClassifyingSpace}) for $(P/N,\cX/N)$ has image $\cD/N$, thus defining $\varphi_{/N}\colon \cX/N \rightarrow \cD/N$.
\item[(iv)] The following commutative diagram commutes
\begin{equation}\label{EqQuotientMHDandMT}
\xymatrix{
\cX \ar[d]_-{\varphi} \ar[r]^-{q_N} & \cX/N \ar[d]^-{\varphi_{/N}} \\
\cD \ar[r]^-{p_N}  & \cD/N.
}
\end{equation}
\end{enumerate}
\end{prop}
\begin{proof} Consider the quotient mixed Hodge datum $q_N \colon (P,\cX) \rightarrow (P/N,\cX/N)$ defined in \eqref{EqQuotientMHD}. Any $\bar{h} \in \cX/N \subseteq \hom(\S_{\C}, (P/N)_{\C})$ induces a $\Q$-mixed Hodge structure on $\lie (P/N)$, via $\mathrm{Ad}_{P/N} \circ \bar{h} \colon \S_{\C} \rightarrow (P/N)_{\C} \rightarrow \mathrm{GL}(\lie (P/N))_{\C}$, which satisfies the three properties listed in \Cref{DefnMixedHD} with $P$ replaced by $P/N$ and $h$ replaced by $\bar{h}$. 

Fix a faithful representation $\bar{\rho} \colon P/N \rightarrow \mathrm{GL}(V')$ defined over $\Q$. Then the morphism $\bar{\rho} \circ \bar{h}$ induces a $\Q$-mixed Hodge structure on $V'$ by \Cref{propPink} for each $\bar{h} \in \cX/N$, and the weight filtration and the Hodge numbers do not depend on the choice of $\bar{h} \in \cX/N$. Thus we obtain a map
\[
\varphi_{/N} \colon \cX/N \rightarrow \{\text{mixed Hodge structures on }V'\}.
\]
Set $\cD/N = \varphi_{/N}(\cX/N)$. Then we get $\varphi_{/N} \colon \cX/N \rightarrow \cD/N$, which by \cite[1.7]{PinkThesis} is $(P/N)(\R)^+(W_{-1}/(W_{-1}\cap N))(\C)$-equivariant  (here $W_{-1} = \cR_u(P)$ and hence $\cR_u(P/N) = W_{-1}/(W_{-1}\cap N)$). This establishes (iii) for the space $\cD/N$.

By \cite[1.12]{PinkThesis} the $\Q$-mixed Hodge structures on $V'$ thus obtained are graded-polarized by some collection of non-degenerate bilinear forms (same for all $\bar{h}$). 
So $\cD/N $ is a contained in some classifying space $\cM'$. This establishes (i).

Now let us construct the map $p_N \colon \cD \rightarrow \cD/N$ and prove properties (ii) and (iv). Take $x \in \cD$, and take any $h_x \in \varphi^{-1}(x)$. Then $\varphi^{-1}(x) = \exp(F^0_x(\lie W_{-1})_{\C}) h_x$ by \Cref{below 4.1}. 
Note that $\exp(F^0_x(\lie W_{-1})_{\C})$ is a 
subgroup of $P_{\C}$. Then $q_N(\varphi^{-1}(x)) = q_N(\exp(F^0_x(\lie W_{-1})_{\C}) h_x) =  \frac{\exp(F^0_x(\lie W_{-1})_{\C})}{N(\C)\cap \exp(F^0_x(\lie W_{-1})_{\C})} q_N(h_x)$.

On the other hand define $\bar{x} := \varphi_{/N}(q_N(h_x))$. Then $\varphi_{/N}^{-1}(\bar{x}) = \exp(F^0_{\bar{x}}(\lie W_{-1}/(W_{-1}\cap N))_{\C}) q_N(h_x)$ again by \Cref{below 4.1}.

We claim that $\frac{\exp(F^0_x(\lie W_{-1})_{\C})}{N(\C)\cap \exp(F^0_x(\lie W_{-1})_{\C})}  = \exp(F^0_{\bar{x}}(\lie W_{-1}/(W_{-1}\cap N))_{\C})$. Indeed it suffices to check for Lie algebras, \textit{i.e.} it suffices to prove $F^0_x(\lie W_{-1})_{\C}/(\lie N_{\C} \cap F^0_x(\lie W_{-1})_{\C}) \cong F^0_{\bar{x}}(\lie W_{-1}/(W_{-1}\cap N))_{\C}$ canonically. As $N \lhd P$, we have $\mathrm{Ad}_P(\lie N) \subseteq \lie N$. So $\lie N$ is a sub-mixed Hodge structure of the adjoint Hodge structure on $\lie P$. Thus $\lie N_{\C} \cap F^0_x(\lie W_{-1})_{\C} = F^0_x(\lie W_{-1}\cap N)_{\C}$. Thus we proved the desired claim.

By the last three paragraphs, we have $q_N(\varphi^{-1}(x)) = \varphi_{/N}^{-1}(\bar{x})$. So the map $\cD \rightarrow \cD/N$, $x \mapsto \bar{x}:=\varphi_{/N}(q_N(h_x))$ is well-defined. Call this map $p_N$. Then property (iv) holds true by construction of $p_N$. Property (ii) then is not hard to check.

Now the map is complex analytic by property (ii).
\end{proof}

\section{Fibered structure and real points}\label{SectionFiberedStruRealPoints}

Let $\cD$ be a Mumford--Tate domain in some classifying space $\cM$
with $P = \mathrm{MT}(\cD)$. Let the connected mixed Hodge datum
$(P,\cX)$ and the $P(\R)^+W_{-1}(\C)^+$-equivariant map $\varphi
\colon \cX \rightarrow \cD$ be as in \Cref{PropMTDomainMHD}.(i). In
particular by \Cref{below 4.1}, the fiber $\varphi^{-1}(x)$ is a principal homogeneous space under $\exp(F^0_x(\lie W_{-1})_{\C})$ for each $x \in \cD$.

\subsection{Fibered structure of Mumford--Tate domains}\label{SubsectionFiberedStructureContinue}
Let $0 = W_{-(m+1)} \subseteq W_{-m} \subseteq \cdots \subseteq W_{-1}$ be the sequence of unipotent normal subgroups of $P$ defined in  (\ref{EqDistSeqSubgp}).

First for each $k \in \{0,\ldots,m\}$, let $\cX_k = \cX/W_{-(k+1)}$ and let
\begin{equation}\label{EqProjMT}
p_k \colon \cD \rightarrow \cD/W_{-(k+1)} =: \cD_k
\end{equation}
be the quotient constructed in \Cref{PropQuotientMT}. Notice that $\cX_m = \cX$ and $p_m$ is the identity on $\cD$.

Observe that we have $(P/W_{-k},\cX_k) = (P/W_{-(k+1)},\cX_{k+1})/(W_{-(k+1)}/W_{-(k+2)})$ and $\cD_k = \cD_{k+1}/(W_{-(k+1)}/W_{-(k+2)})$. Denote by $q_{k+1,k} \colon (P/W_{-(k+1)},\cX_{k+1}) \rightarrow (P/W_{-k},\cX_k)$ and $p_{k+1,k} \colon \cD_{k+1} \rightarrow \cD_k$ the quotients.  Then by \Cref{PropQuotientMT} we have the following commutative diagram
\begin{equation}\label{EqDiagXandDFiber}
\xymatrix{
\cX=\cX_m \ar[r]^-{q_{m,m-1}} \ar[d]_-{\varphi_m := \varphi} & \cX_{m-1} \ar[r]^-{q_{m-1,m-2}} \ar[d]_-{\varphi_{m-1}} & \cX_{m-2} \ar[r]^-{q_{m-2,m-3}} \ar[d]_-{\varphi_{m-2}} & \cdots \ar[r]^-{q_{2,1}} & \cX_1 \ar[r]^-{q_{1,0}} \ar[d]^{\varphi_1} & \cX_0 \ar[d]^-{\varphi_0} \\
\cD =\cD_m \ar[r]^-{p_{m,m-1}} & \cD_{m-1} \ar[r]^-{p_{m-1,m-2}} & \cD_{m-2} \ar[r]^-{p_{m-2,m-3}} & \cdots \ar[r]^-{p_{2,1}} & \cD_{1} \ar[r]^-{p_{1,0}} & \cD_{0}
}.
\end{equation}

By \Cref{below 4.1}, $\varphi_0$ is bijective. But the other $\varphi_i$'s are not injective in general.

Let $k \in \{0,\ldots, m-1\}$. Recall that $W_{-(k+1)}/W_{-(k+2)} = \lie W_{-(k+1)}/ W_{-(k+2)}$ is a vector group. Thus for any $x_k \in \cD_{k}$, the notation $F^0_{x_k}(W_{-(k+1)}/W_{-(k+2)})_{\C}$ makes sense.

\begin{lemma}\label{LemmaFibersOfVarphikPkkplus1}
For each $k \in \{0,\ldots, m\}$ and any point $x_k \in \cD_k$, we have that
\begin{enumerate}
\item[(i)] the fiber  $\varphi_k^{-1}(x_k)$ is a principal homogeneous space under $F^0_{x_k}(W_{-(k+1)}/W_{-(k+2)})_{\C}$.
\item[(ii)] (for $k \le m-1$) the fiber $p_{k+1,k}^{-1}(x_k)$ is a principal homogeneous space under 
\[
(W_{-(k+1)}/W_{-(k+2)})(\C) / F^0_{x_k}(W_{-(k+1)}/W_{-(k+2)})_{\C}.
\]
\end{enumerate}
\end{lemma}
\begin{proof}
Part (i) follows directly from \Cref{below 4.1}.

For (ii): By \cite[1.8(a)]{PinkThesis}, each fiber of $q_{k+1,k}$ is a principal homogeneous space under $(W_{-(k+1)}/W_{-(k+2)})(\C)$. Combined with part (i) we can conclude.
\end{proof}

\subsection{Real points}\label{SubsectionRealPt}
Define $\cD_{\R}$ to be the set of $x \in \cD$ such that the mixed Hodge structure parametrized by $x$ is split over $\R$. Namely, $\cD_{\R} = \varphi(\cX_{\R})$ with $\cX_{\R} = \{ h \colon \S_{\C} \rightarrow P_{\C} : h\text{ is defined over }\R\} \subseteq \cX$.

It is known that $\cD_{\R} = P(\R)^+x$ for some $x \in \cD$; see \cite[last~Remark~of~$\mathsection$3]{Pearlstein2000}.

Moreover for any $x \in \cD_{\R}$, it is not hard to check that $F^0_x(\lie W_{-1})_{\C} \cap \lie P_{\R} = \{0\}$. So by \Cref{below 4.1}, $p_0 \colon P \rightarrow G = P/W_{-1}$ induces
\begin{equation}\label{EqStabRealPoint}
\mathrm{Stab}_{P(\R)^+}(x) \cong \mathrm{Stab}_{G(\R)^+}(p_0(x)).
\end{equation}

Consider the real semi-algebraic $P(\R)^+$-equivariant retraction induced by the $\mathfrak{sl}_2$-splitting \cite[Thm.~2.18]{SL2Splitting} (see also \cite[Cor.~3.12]{BBKT}) 
\begin{equation}\label{EqRetraction}
r \colon \cD \rightarrow \cD_{\R}.
\end{equation}

For each $k \in \{0,\ldots,m-1\}$, $\cD_k$ is a Mumford--Tate domain and hence we can define $\cD_{k,\R}$ as above. Then $\cD_{k,\R}$ is a $(P/W_{-(k+1)})(\R)^+$-orbit, and there is a real semi-algebraic $(P/W_{-(k+1)})(\R)^+$-equivariant retraction $r_k \colon \cD_k \rightarrow \cD_{k,\R}$ induced by the $\mathfrak{sl}_2$-splitting. 

Let $p_k \colon \cD \rightarrow \cD_k$ be from \eqref{EqProjMT}. 
 The following diagram is commutative by \cite[Lem.~6.6]{BBKT}:
\begin{equation}\label{EqDiagramRetractionProjection}
\xymatrix{
\cD \ar[r]^-{p_k} \ar[d]_{r} & \cD_k \ar[d]^{r_k} \\
\cD_{\R} \ar[r]^-{p_k|_{\cD_{\R}}} & \cD_{k,\R}.
}
\end{equation}



We close this subsection with the following proposition, which states that $\cD_{\R}$ can be split (non-canonically) into the product of a Mumford--Tate domain for pure Hodge structures and some vector spaces.

\begin{prop}\label{PropProductStrMHD}
There exists a real algebraic isomorphism
\begin{equation}\label{EqBundleStrOnXreal}
\cD_{\R} \cong \cD_{0} \times (W_{-1}/W_{-2})(\R) \times \cdots \times (W_{-(m-1)}/W_{-m})(\R) \times W_{-m}(\R)
\end{equation}
with the following properties.
\begin{enumerate}
\item[(i)] For any $g = (g_0, w_1,\ldots, w_m) \in P(\R)^+$ under the identification \eqref{EqDecompGroup} and any $x = (x_0,x_1,\ldots,x_m) \in \cD_{\R}$ under \eqref{EqBundleStrOnXreal}, the action of $P(\R)^+$ on $\cD_{\R}$ is given by the formula
\small
\begin{equation}\label{EqFormulaActionBCH}
gx =  (g_0 x_0, w_1 + g_0 x_1, w_2+g_0 x_2 + \mathrm{calb}_2(w_1, g_0 x_1), \ldots, w_m + g_0 x_m + \mathrm{calb}_m(\mathbf{w}_{m-1}, g_0\mathbf{x}_{m-1}))
\end{equation}
\normalsize
where $\mathbf{w}_k = (w_1,\ldots, w_k)$ and $\mathbf{x}_k = (x_1,\ldots,x_k)$ for all $k \ge 1$, and the $\mathrm{calb}_k$'s are the $\Q$-polynomials of degree at most $k-1$ given by \Cref{LemmaCalbFactorProdUni}.
\item[(ii)] The decomposition \eqref{EqBundleStrOnXreal} is compatible with taking quotients of $W_{-(k+1)}$ on both sides for each $k \in \{0,\ldots,m-1\}$, \textit{i.e.}, the following diagram commutes
\[
\xymatrix{
\cD_{\R} \ar[r]^-{\sim} \ar[d]_{p_k|_{\cD_{\R}}} & \cD_{0} \times (W_{-1}/W_{-2})(\R) \times \cdots \times (W_{-(m-1)}/W_{-m})(\R) \times W_{-m}(\R) \ar[d] \\
\cD_{k,\R} \ar[r]^-{\sim} & \cD_{0} \times (W_{-1}/W_{-2})(\R) \times \cdots \times (W_{-k}/W_{-(k+1)})(\R)
}
\]
where the top arrow is \eqref{EqBundleStrOnXreal}, the bottom arrow is \eqref{EqBundleStrOnXreal} applied to $\cD_{k,\R}$, and the right arrow is omitting the last $m-k$ factors.
\end{enumerate}
\end{prop}
\begin{proof}
First note that $\cD_{0,\R} = \cD_0$ because every pure Hodge structure is split over $\R$. Now \eqref{EqDecompGroup} and \eqref{EqStabRealPoint} together induce a real algebraic isomorphism as in \eqref{EqBundleStrOnXreal}. Part (ii) is clear. Part (i) follows from the group law given by \eqref{EqFormulaGroupLaw}.
\end{proof}


\section{Period Map and Logarithmic Ax}\label{SectionPeriodMapLogAx}

\subsection{Period map} Let $S$ be an irreducible algebraic variety defined over $\C$. Assume that $S$ carries a graded-polarized VMHS $(\V_\Z,W_\bullet,\cF^\bullet) \rightarrow S$. Then it induces a period map $[\Phi] \colon S \rightarrow \Gamma\backslash \cM$ where $\cM$ is the classifying space and $\Gamma$ is an arithmetic subgroup of $P^{\cM}(\Q)$.

The period map $[\Phi]$ factors through another quotient space in the following way. In the context of \Cref{ThmAS}, we have a complex analytic irreducible subset $\mathscr{Z}$ of $S \times_{\Gamma\backslash\cM}\cM = \{(s,x) \in S \times \cM : [\Phi](s) = u(x)\}$, where $u \colon \cM \rightarrow \Gamma\backslash\cM$. For the projection $p_{\cM} \colon S\times \cM \rightarrow \cM$, we have that $p_{\cM}(\mathscr{Z})$ is irreducible and is contained in $u^{-1}([\Phi](S))$. Let $\tilde{S}$ be a complex analytic irreducible component of $u^{-1}([\Phi](S))$ which contains $p_{\cM}(\mathscr{Z})$. Then $\mathscr{Z} \subseteq S \times \tilde{S}$. Let $\cD = \tilde{S}^{\mathrm{sp}}$, the smallest Mumford--Tate domain containing $\tilde{S}$; see \Cref{LemmaSmallestMTDomain}. Let $P = \mathrm{MT}(\tilde{S})$ and $W_{-1} = \cR_u(P)$, then $\cD$ is a $P(\R)^+W_{-1}(\C)$-orbit. Now we have $[\Phi](S) \subseteq u(\cD)$.

Let $\Gamma_P = \Gamma \cap P(\Q)$, then $[\Phi]$ factors through $S \rightarrow \Gamma_P \backslash \cD$. The inclusion $\cD \subseteq \cM$ induces a finite map $\Gamma_P \backslash\cD \rightarrow \Gamma\backslash \cM$.


Let $\Delta = S \times _{\Gamma_P\backslash\cD}\cD$. So to prove \Cref{ThmAS}, it suffices to work  in the following diagram and assume $\mathscr{Z} \subseteq \Delta$
\begin{equation}\label{EqSetupDiag}
\xymatrix{
S \times \cD \supseteq & \Delta \ar[r] \ar[d]_-{u_S} \pullbackcorner & \cD \ar[d]^-{u} \\
& S \ar[r]^-{[\Phi]} & \Gamma_P \backslash \cD
}.
\end{equation}
This is our setup for the rest of the paper.

\subsection{Quotient for the period map}
Assume $N \lhd P$. We have constructed the quotient Mumford--Tate domain $p_N \colon \cD \rightarrow \cD/N$ in \Cref{PropQuotientMT}. For the arithmetic group $\Gamma_{P/N} := \Gamma_P/(\Gamma_P \cap N(\Q))$, we then have a map $[p_N] \colon \Gamma_P \backslash \cD \rightarrow \Gamma_{P/N}\backslash(\cD/N)$. Composing with $[\Phi] \colon S \rightarrow \Gamma_P \backslash\cD$, we obtain
\begin{equation}\label{EqQuotPerMap}
[\Phi_{/N}] \colon S \rightarrow \Gamma_{P/N}\backslash(\cD/N).
\end{equation}
\Cref{PropQuotientMT} says that $\cD/N$ is a Mumford--Tate domain in the classifying space of some mixed Hodge structures. Thus $[\Phi_{/N}]$ is again a period map.

Let us summarize the notations involving this operation of taking quotient in the following diagram:
\begin{equation}\label{EqEqQuotPerMapDiag}
\xymatrix{
& \cD \ar[r]^-{p_N} \ar[d]_-{u} & \cD/N \ar[d]^-{u_{/N}} \\
S \ar[r]^-{[\Phi]} \ar@/_0.6cm/[rr]|-{[\Phi_{/N}]} & \Gamma_P\backslash\cD \ar[r]^-{[p_N]} & \Gamma_{P/N}\backslash(\cD/N) 
}
\end{equation}

\subsection{Bi-algebraic system}\label{SubsectionBiAlg}
Recall that $\cM$ is a semi-algebraic open subset in some algebraic variety $\cM^\vee$ over $\C$. So $\cD$ is a semi-algebraic open subset in some algebraic variety $\cD^\vee$ over $\C$.
\begin{defn}\label{DefnBiAlg}
\begin{enumerate}
\item[(i)] A subset of $\cD$ is said to be \textit{irreducible algebraic} if it is a complex analytic irreducible component of $U \cap \cD$, with $U$  an algebraic subvariety of $\cD^\vee$.
\item[(ii)] An irreducible algebraic subset $W$ of $\cD$ is said to be \textit{bi-algebraic} if $[\Phi]^{-1}(u(W))$ is algebraic.
\end{enumerate}
\end{defn}
By \cite[Cor.~6.7]{BBKT}, every weak Mumford--Tate domain is bi-algebraic.

\subsection{Logarithmic Ax}
In this subsection we prove a particular case of \Cref{ThmAS}. Retain $\mathscr{Z}$ as in \Cref{ThmAS} and the notations in and above \eqref{EqEqQuotPerMapDiag}. As discussed before, we have $\mathscr{Z} \subseteq \Delta \cap (S \times \tilde{S})$.

\begin{thm}\label{ThmLogAx}
There is a smallest weak Mumford--Tate domain in $\cD$, denoted by $\tilde{S}^{\mathrm{ws}}$, which contains $\tilde{S}$. Moreover,
\begin{enumerate}
\item[(i)] $\mathscr{Z}^{\mathrm{Zar}} \subseteq S \times \tilde{S}^{\mathrm{ws}}$.
\item[(ii)] \Cref{ThmAS} holds if $u_S(\mathscr{Z}) = S$.
\end{enumerate}
\end{thm}
In the proof, we will see that  $\tilde{S}^{\mathrm{ws}}$ is an $N(\R)^+(W_{-1}\cap N)(\C)$-orbit, where $N$ is the connected algebraic monodromy group of $(\V,W_\bullet,\cF^\bullet) \rightarrow S$.
\begin{proof}
Let $N$ be the connected algebraic monodromy group of $(\V,W_\bullet,\cF^\bullet) \rightarrow S$. Then $N \lhd P$ by \Cref{ThmAndreNormal}. Thus $N(\R)^+(W_{-1}\cap N)(\C)\tilde{s}$ is a weak Mumford--Tate domain, for any $\tilde{s} \in \tilde{S}$.

As $N \lhd P$, we have the quotient period map $[\Phi_{/N}] \colon S \rightarrow \Gamma_{S/N}\backslash(\cD/N)$ constructed in \eqref{EqQuotPerMap}. 
Note that $[\Phi_{/N}]$ gives rise to a new VMHS over $S$, whose  connected algebraic monodromy group is trivial. So $[\Phi_{/N}](S)$ is a point by \cite[Thm.~7.12]{brylinski1998overview}. Thus using the notations in \eqref{EqEqQuotPerMapDiag}, we have that $p_N(\tilde{S})$ is a point. So $\tilde{S} \subseteq N(\R)^+(W_{-1}\cap N)(\C)\tilde{s}$ for any $\tilde{s} \in \tilde{S}$.

In particular $N(\R)^+(W_{-1}\cap N)(\C)\tilde{s}$ is independent of the choice of $\tilde{s} \in \tilde{S}$.


Let us start by proving part (ii). In the course of this proof, we will also show the existence of $\tilde{S}^{\mathrm{ws}}$.

Assume $u_S(\mathscr{Z}) = S$. Since $\mathscr{Z} \subseteq S \times \tilde{S}$, the following is true:  For each $s \in S$, there exists $\tilde{s} \in \tilde{S}$ such that $(s,\tilde{s}) \in \mathscr{Z}$. 

The group $P(\R)^+W_{-1}(\C)$ acts on $S \times \cD$ via its action on the second factor. Let $\rho \colon \pi_1(S,s) \rightarrow \mathrm{GL}(V)$ be the monodromy representation. Then $\mathrm{Im}(\rho)$ is a subgroup of $\Gamma_P$. By construction of $\tilde{S}$, we have $\mathrm{Im}(\rho)(s,\tilde{s}) \subseteq \mathscr{Z}$ for any $(s,\tilde{s}) \in \mathscr{Z}$. Taking Zariski closures of both sides and recalling that $N = (\mathrm{Im}(\rho)^{\mathrm{Zar}})^{\circ}$, we have $\{s\} \times N(\R)^+(W_{-1}\cap N)(\C)\tilde{s} \subseteq \mathscr{Z}^{\mathrm{Zar}}$. 
As this holds true for each $s \in S$, we then have $S \times N(\R)^+(W_{-1}\cap N)(\C)\tilde{s} \subseteq \mathscr{Z}^{\mathrm{Zar}}$.

To sum it up, we have $\mathscr{Z} \subseteq S \times \tilde{S} \subseteq S \times N(\R)^+(W_{-1}\cap N)(\C)\tilde{s} \subseteq \mathscr{Z}^{\mathrm{Zar}}$. 
By taking Zariski closures, we have $\mathscr{Z}^{\mathrm{Zar}} = S \times N(\R)^+(W_{-1}\cap N)(\C)\tilde{s}$ and $\tilde{S}^{\mathrm{Zar}} = N(\R)^+(W_{-1}\cap N)(\C)\tilde{s}$.

By definition, $N(\R)^+(W_{-1}\cap N)(\C)\tilde{s}$ is a weak Mumford--Tate domain. Moreover if $\cW$ is a weak Mumford--Tate domain which contains $\tilde{S}$, then $\cW$ contains $\tilde{S}^{\mathrm{Zar}} = N(\R)^+(W_{-1}\cap N)(\C)\tilde{s}$ because $\cW$ is algebraic. So  $N(\R)^+(W_{-1}\cap N)(\C)\tilde{s}$ is the smallest weak Mumford--Tate domain which contains $\tilde{S}$. Thus $\tilde{S}^{\mathrm{ws}}$ exists and is precisely $N(\R)^+(W_{-1}\cap N)(\C)\tilde{s}$. Now part (ii) is established.

Now part (i) is immediately true because $\mathscr{Z} \subseteq S \times \tilde{S}$ and $S \times \tilde{S}^{\mathrm{ws}}$ is algebraic.
\end{proof}

\begin{rmk}\label{RmkSmallestWMT}
If we assume $S = u_S(\mathscr{Z})^{\mathrm{Zar}}$, then $\tilde{S}^{\mathrm{ws}}$ is the smallest  weak Mumford--Tate domain which contains $p_{\cD}(\mathscr{Z})$. Indeed, we have $p_{\cD}(\mathscr{Z}) \subseteq \tilde{S}^{\mathrm{ws}}$ by \Cref{ThmLogAx}.(i). So it suffices to prove the following statement: for any $W$ a weak Mumford--Tate domain in $\cD$ which contains $p_{\cD}(\mathscr{Z})$, we have $\tilde{S}^{\mathrm{ws}} \subseteq W$. This is true: $u(W) \supseteq u(p_{\cD}(\mathscr{Z})) = [\Phi](u_S(\mathscr{Z}))$, so $[\Phi]^{-1}(u(W)) \supseteq u_S(\mathscr{Z})$, so $[\Phi]^{-1}(u(W)) \supseteq S$ because $[\Phi]^{-1}(u(W))$ is algebraic (by \cite[Cor.~6.7]{BBKT}) and $S = u_S(\mathscr{Z})^{\mathrm{Zar}}$. Therefore $\tilde{S}^{\mathrm{ws}} \subseteq W$ and hence we are done.

\end{rmk}

\section{D\'{e}vissage and Preparation} \label{preparation}
In this section, we do some preparations. Recall the setup \eqref{EqSetupDiag}
\[
\xymatrix{
S \times \cD \supseteq & \Delta \ar[r]^-{p_{\cD|_{\Delta}}}\ar[d]_-{u_S} \pullbackcorner & \cD \ar[d]^-{u} \\
& S \ar[r]^-{[\Phi]} & \Gamma_P \backslash \cD
}.
\]

\begin{lemma}\label{LemmaSecondDevissage}
  If \Cref{ThmAS} holds true under the following two additional assumptions:
\begin{enumerate}
\item[(i)] $S = u_S(\mathscr{Z})^{\mathrm{Zar}}$.
\item[(ii)] $\mathscr{Z}$ is a complex analytic irreducible component of $\mathscr{Z}^{\mathrm{Zar}} \cap \Delta$.
\end{enumerate}

then it holds true in full generality.
\end{lemma}

\begin{proof}
Let $\mathscr{Z}$ be as in \Cref{ThmAS}.  
Notice that $\mathscr{Z}^{\mathrm{Zar}} \subseteq u_S(\mathscr{Z})^{\mathrm{Zar}} \times \cD$. The assumptions and the conclusion of \Cref{ThmAS} do not change if we replace $S$ by $u_S(\mathscr{Z})^{\mathrm{Zar}}$. So we may assume $S = u_S(\mathscr{Z})^{\mathrm{Zar}}$.

Let $\mathscr{Z}'$ be a complex analytic irreducible component of $\mathscr{Z}^{\mathrm{Zar}} \cap \Delta$ which contains $\mathscr{Z}$. Note that $\mathscr{Z} \subseteq \mathscr{Z}' \subseteq \mathscr{Z}^{\mathrm{Zar}}$. Thus by taking the Zariski closures, we obtain  $\mathscr{Z}'^{\mathrm{Zar}} = \mathscr{Z}^{\mathrm{Zar}}$.



Thus $p_{\cD}(\mathscr{Z}'^{\mathrm{Zar}}) = p_{\cD}(\mathscr{Z}^{\mathrm{Zar}})$, for the projection $p_{\cD} \colon S \times \cD \rightarrow \cD$. 
So for the algebraic structure on $\cD$ defined by \Cref{DefnBiAlg}, we have $p_{\cD}(\mathscr{Z}')^{\mathrm{Zar}} = p_{\cD}(\mathscr{Z})^{\mathrm{Zar}}$ because the projection $p_{\cD}$ is algebraic. But each weak Mumford--Tate domain is algebraic. So
\[
p_{\cD}(\mathscr{Z}') \subseteq p_{\cD}(\mathscr{Z}')^{\mathrm{Zar}} = p_{\cD}(\mathscr{Z})^{\mathrm{Zar}} \subseteq p_{\cD}(\mathscr{Z})^{\mathrm{ws}} = \tilde{S}^{\mathrm{ws}},
\]
where the last equality follows from \Cref{RmkSmallestWMT}. 
But $p_{\cD}(\mathscr{Z}) \subseteq p_{\cD}(\mathscr{Z}')$ because $\mathscr{Z} \subseteq \mathscr{Z}'$. So every weak Mumford--Tate domain containing $p_{\cD}(\mathscr{Z}')$ must also contain $p_{\cD}(\mathscr{Z})$, and thus contains $\tilde{S}^{\mathrm{ws}}$ by \Cref{RmkSmallestWMT}. Combined with the inclusion above, we get that $\tilde{S}^{\mathrm{ws}}$ is also the smallest weak Mumford--Tate domain which contains $p_{\cD}(\mathscr{Z}')$. So
\[
\dim \mathscr{Z}^{'\mathrm{Zar}} - \dim \mathscr{Z}' \ge \dim p_{\cD}(\mathscr{Z}')^{\mathrm{ws}} \Longrightarrow \dim \mathscr{Z}^{\mathrm{Zar}} - \dim \mathscr{Z} \ge \dim p_{\cD}(\mathscr{Z})^{\mathrm{ws}}
\]
as $\dim \mathscr{Z} \le \dim \mathscr{Z}'$ and $p_{\cD}(\mathscr{Z})^{\mathrm{ws}} = p_{\cD}(\mathscr{Z}')^{\mathrm{ws}} = \tilde{S}^{\mathrm{ws}}$. Replacing $\mathscr{Z}$ by $\mathscr{Z}'$, it is thus enough to prove \Cref{ThmAS} assuming furthermore (ii).
\end{proof}

Thus our main theorem is reduced to the following theorem, which we will prove in the rest of the paper.
\begin{thm}\label{ThmASDevissage}
\Cref{ThmAS} holds true under the additionnal assumption that $\mathscr{Z}$ is a complex analytic irreducible component of $\mathscr{Z}^{\mathrm{Zar}} \cap \Delta$ and $S = u_S(\mathscr{Z})^{\mathrm{Zar}}$.
\end{thm}

The rest of the paper is devoted to prove \Cref{ThmASDevissage}.

\section{Bigness of the $\Q$-stabilizer}\label{SectionBigness}
Recall our setup
\begin{equation}
\xymatrix{
S \times \cD \supseteq & \Delta \ar[r]^-{p_{\cD}|_{\Delta}} \ar[d]_-{u_S} \pullbackcorner & \cD \ar[d]^-{u} \\
& S \ar[r]^-{[\Phi]} & \Gamma_P \backslash \cD
}.
\end{equation}
We consider a subset $\mathscr{Z}$ of $\Delta$ satisfying the following properties: (i) $\mathscr{Z}$ is a complex analytic irreducible component of $\mathscr{Z}^{\mathrm{Zar}}\cap \Delta$; (ii) $S = u_S(\mathscr{Z})^{\mathrm{Zar}}$.

Let $H_{\mathscr{Z}^{\mathrm{Zar}}}$ be the $\Q$-stabilizer of $\mathscr{Z}^\Zar$, namely
\begin{equation}\label{EqQStabZZar}
H_{\mathscr{Z}^{\mathrm{Zar}}} = \left(\mathrm{Stab}_{P(\R)}(\mathscr{Z}^{\mathrm{Zar}}) \cap \Gamma_P \right)^{\mathrm{Zar},\circ} = \left(\{\gamma \in \Gamma_P : \gamma\mathscr{Z}^{\mathrm{Zar}} = \mathscr{Z}^{\mathrm{Zar}}\}^{\mathrm{Zar}}\right)^{\circ}.
\end{equation}

In this section we prove the following case of \Cref{ThmASDevissage}:
\begin{prop}\label{PropBignessStab}
\Cref{ThmASDevissage} holds true under the additional assumption $H_{\mathscr{Z}^{\mathrm{Zar}}}$ is the trivial group.
\end{prop}

\subsection{Auxiliary set}\label{SubsectionAuxiliarySet}
Our proof of \Cref{PropBignessStab} heavily uses o-minimality. We are able to work in this framework thanks to the  following theorem proved by the second-named author, Bakker, Brunebarbe, and Tsimerman. In the pure case this theorem is the main result of \cite{BKT}.
\begin{thm}[{\!\!\cite[Prop.~3.13 and Thm.~4.4]{BBKT}}]\label{ThmFundSet}
Let $r \colon \cD \rightarrow \cD_{\R}$ be the retraction defined in \eqref{EqRetraction}, and identify $\cD_{\R}$ with $\cD_0 \times \prod_{1\le k \le m} (W_{-k}/W_{-k-1})(\R)$ under the real-algebraic isomorphism defined in \eqref{EqBundleStrOnXreal}.

There exist an $\R_{\mathrm{alg}}$-definable subset $\mathfrak{F}_0$  of $\cD_0$ and a real number $M > 0$ such that 
\begin{equation}\label{EqFundD}
\mathfrak{F}_{\R} := \mathfrak{F}_0 \times \prod_{1 \le k \le m} (-M,M)^{\dim (W_{-k}/W_{-(k+1)})(\R)},
\end{equation}
which is a $\R_{\mathrm{alg}}$-definable subset of  $\cD_{\R}$, 
satisfies the following properties:
\begin{enumerate}
\item[(i)] $u|_{r^{-1}(\mathfrak{F}_{\R})}$ is surjective;
\item[(ii)] $[\Phi]$ is $\R_{\mathrm{an},\exp}$-definable for the $\R_{\mathrm{alg}}$-structure on $\Gamma_P\backslash\cD$ defined by $r^{-1}(\mathfrak{F}_{\R})$.
\end{enumerate}
\end{thm}

\vskip 0.5em

The following auxiliary set is important for the proof of Ax-Schanuel.
\begin{equation}\label{EqDefnSetTheta}
\Theta := \{g \in P(\R) : \dim (g^{-1} \mathscr{Z}^{\mathrm{Zar}} \cap (S \times\mathfrak{F}) \cap \Delta ) = \dim \mathscr{Z} \},
\end{equation}
with $\mathfrak{F} = r^{-1}(\mathfrak{F}_{\R})$. 
It is clear that $\Theta$ is definable in $\R_{\mathrm{an},\exp}$, and
\[
\{\gamma \in \Gamma_P : \gamma(S\times\mathfrak{F}) \cap \mathscr{Z} \not= \emptyset\} \subseteq \Theta.
\]
Denote for simplicity by $\tilde{Z} = p_{\cD}(\mathscr{Z})$, then
\[
p_{\cD}\left( \gamma(S\times\mathfrak{F}) \cap \mathscr{Z} \right) = p_{\cD}(p_{\cD}^{-1}(\gamma\mathfrak{F}) \cap \mathscr{Z}) = \gamma\mathfrak{F} \cap \tilde{Z}.
\]
Thus for any $\gamma \in \Gamma_P$, we have
\[
\gamma(S\times\mathfrak{F}) \cap \mathscr{Z} \not=\emptyset \Leftrightarrow \gamma\mathfrak{F} \cap \tilde{Z} \not=\emptyset.
\]
Therefore
\begin{equation}\label{EqDistSubsetOfTheta}
\{\gamma \in \Gamma_P : \gamma\mathfrak{F} \cap \tilde{Z} \not= \emptyset\} \subseteq \Theta.
\end{equation}


\begin{thm}\label{ThmCounting}
Assume $\dim \tilde{Z} > 0$. Then there exist constants $\epsilon > 0$, $c_{\epsilon}>0$ and a sequence of real numbers $\{T_i\}_{i\in \mathbb{N}}$ with $T_i \rightarrow \infty$ such that
\begin{equation}\label{EqTheEstimateLeadingToFindingRealAlgCurves}
\#\{\gamma \in \Theta \cap \Gamma_P : H(\gamma) \le T_i\} \ge c_{\epsilon} T_i^{\epsilon}.
\end{equation}
\end{thm}

\subsection{Proof of \Cref{PropBignessStab} assuming \Cref{ThmCounting}}
If $\dim \tilde{Z} = 0$, then $\dim \tilde{Z}^{\mathrm{ws}} = 0$ and hence \Cref{ThmASDevissage} clearly holds true. So we assume $\dim \tilde{Z} > 0$.

We prove \Cref{PropBignessStab} by (downward) induction on $\dim \mathscr{Z}^{\mathrm{Zar}}$. The starting point for this induction is when $\mathscr{Z}^{\mathrm{Zar}} = S \times \tilde{S}^{\mathrm{ws}}$ (see \Cref{ThmLogAx}). In this case, under the assumptions of \Cref{ThmASDevissage} we have $\mathscr{Z} = S\times_{\Gamma_P\backslash\cD} \tilde{S}^{\mathrm{ws}}$, and so $\dim \mathscr{Z} = \dim S$. Thus \Cref{ThmASDevissage} holds true in this case.

Let $c_{\epsilon}>0$, $\epsilon >0$ and $\{T_i\}$ be as in \Cref{ThmCounting}. Then by the Pila--Wilkie counting theorem \cite[3.6]{PilaO-minimality-an}, 
for each $T_i$ there exists a connected semi-algebraic curve $C_i \subseteq \Theta$ which contains $\ge c_{\epsilon}T_i^{\epsilon}$ points in $\Gamma_P$ of height at most $T_i$. 
For $T_i \gg 1$ we have $c_{\epsilon}T_i^{\epsilon} \ge 2$.

Fix $c_0 \in C_i \cap \Gamma_P$. Set $C := c_0^{-1} \cdot C_i$. Then $C$ is a semi-algebraic curve which contains $\ge c_{\epsilon}T_i^{\epsilon}$ in $\Gamma_P$.

For each $c' \in C_i \subseteq \Theta$, we have $\dim (c^{\prime-1}\mathscr{Z}^{\mathrm{Zar}} \cap \Delta) = \dim \mathscr{Z}$ by definition of $\Theta$ from \eqref{EqDefnSetTheta}. But $c_0 \Delta = \Delta$ since $c_0 \in \Gamma_P$. So we have
\begin{equation}\label{EqDimEqualTrans}
\dim (c^{-1}\mathscr{Z}^{\mathrm{Zar}} \cap \Delta) = \dim \mathscr{Z}\qquad \text{ for all }c \in C.
\end{equation}

Notice that $\mathscr{Z}^{\mathrm{Zar}} \subseteq C^{-1}\mathscr{Z}^{\mathrm{Zar}}$. Moroever 
since $C$ is a semi-algebraic curve, we have $\dim (C^{-1}\mathscr{Z}^{\mathrm{Zar}})^{\mathrm{Zar}} \le \dim \mathscr{Z}^{\mathrm{Zar}} + 1$.

We have the following alternative:
\begin{enumerate}
\item[(i)] $\dim (C^{-1}\mathscr{Z}^{\mathrm{Zar}})^{\mathrm{Zar}} = \dim \mathscr{Z}^{\mathrm{Zar}}$;
\item[(ii)] $\dim (C^{-1}\mathscr{Z}^{\mathrm{Zar}})^{\mathrm{Zar}} = \dim \mathscr{Z}^{\mathrm{Zar}}+1$.
\end{enumerate}

Assume we are in case (i). Then $C \subseteq \mathrm{Stab}_{P(\R)}(\mathscr{Z}^{\mathrm{Zar}})$. Hence $\#(\mathrm{Stab}_{P(\R)}(\mathscr{Z}^{\mathrm{Zar}}) \cap \Gamma_P) \ge C \cap \Gamma_P \ge c_{\epsilon}T_i^{\epsilon}$ for each $i$. Letting $T_i \rightarrow \infty$, we get $\#(\mathrm{Stab}_{P(\R)}(\mathscr{Z}^{\mathrm{Zar}}) \cap \Gamma) = \infty$. Hence $\dim H_{\mathscr{Z}^{\mathrm{Zar}}} > 0$. This contradicts the triviality of $H_{\mathscr{Z}^{\mathrm{Zar}}}$.

Thus we are in case (ii). Then there exists $c \in C$ such that $\mathscr{Z}^{\mathrm{Zar}} \not= c^{-1}\mathscr{Z}^{\mathrm{Zar}}$. Thus $\mathscr{Z} \not\subseteq c^{-1}\mathscr{Z}^{\mathrm{Zar}}$; otherwise taking the Zariski closures we get $\mathscr{Z}^{\mathrm{Zar}} \subseteq c^{-1}\mathscr{Z}^{\mathrm{Zar}}$, hence $\mathscr{Z}^{\mathrm{Zar}} = c^{-1}\mathscr{Z}^{\mathrm{Zar}}$ by comparing dimensions, contradicting the choice of $c$. Thus $c^{-1}\mathscr{Z}^{\mathrm{Zar}} \cap \Delta$ varies with $c \in C$. Therefore by \eqref{EqDimEqualTrans}, an irreducible component $\mathscr{Z}' \supseteq \mathscr{Z}$ of $(C^{-1}\mathscr{Z}^{\mathrm{Zar}})^{\mathrm{Zar}} \cap \Delta$ which has dimension $\ge \dim \mathscr{Z} + 1$.

We claim that  $\mathscr{Z}^{\prime\mathrm{Zar}} = (C^{-1}\mathscr{Z}^{\mathrm{Zar}})^{\mathrm{Zar}}$. Indeed, otherwise $\mathscr{Z}^{\prime\mathrm{Zar}} = \mathscr{Z}^{\mathrm{Zar}}$ by dimension comparisons. But then the assumption of \Cref{ThmASDevissage} says that $\mathscr{Z}$ is a component of $\mathscr{Z}^{\mathrm{Zar}} \cap \Delta = \mathscr{Z}^{\prime\mathrm{Zar}}\cap \Delta$. Hence $\mathscr{Z}' \subseteq \mathscr{Z}$. This contradicts $\dim \mathscr{Z}' \ge \dim \mathscr{Z}+1$.

So we can apply the induction hypothesis to $\mathscr{Z}'$ and obtain
\[
\dim \mathscr{Z}^{\prime\mathrm{Zar}} - \dim \mathscr{Z}' \ge \dim p_{\cD}(\mathscr{Z}')^{\mathrm{ws}}.
\]
But the left hand side $\le \dim \mathscr{Z}^{\mathrm{Zar}} - \dim \mathscr{Z}$ and the right hand side is  $\ge \dim p_{\cD}(\mathscr{Z})^{\mathrm{ws}}$. Hence \Cref{ThmASDevissage} holds true for $\mathscr{Z}$.

We are done. 

\subsection{Preparation of the proof of \Cref{ThmCounting}}\label{SubsectionPrepThmCounting}
We will prove \Cref{ThmCounting}, or more precisely \eqref{EqTheEstimateLeadingToFindingRealAlgCurves}, in the rest of this section. The proof is long. It will be divided in several steps for readers' convenience. In this subsection, we fix some notations.


The proof of \eqref{EqTheEstimateLeadingToFindingRealAlgCurves} uses the fibered structure of $\cD$ and the discussion on its real points, both explained in \Cref{SectionFiberedStruRealPoints}. 
We start by recollecting basic knowledge on both aspects.

Recall the sequence of normal subgroups 
\[
0 = W_{-(m+1)} \subseteq W_{-m} \subseteq \cdots W_{-1} = \cR_u(P) 
\]
 of $P$ from \eqref{EqDistSeqSubgp}, and the quotient Mumford--Tate domains $p_k \colon \cD \rightarrow \cD_k := \cD/W_{-k-1}$, for each $k \in \{0,\ldots,m\}$, from \eqref{EqProjMT}. Notice that $p_m$ is the identity map on $\cD$.

Let $r \colon \cD \rightarrow \cD_{\R}$ be the $P(\R)^+$-equivariant retraction of the inclusion $\cD_{\R}\subseteq \cD$ from \eqref{EqRetraction}. Applying \eqref{EqDiagramRetractionProjection} successively to $p_{k,k-1} \colon \cD_{k+1} \rightarrow \cD_k$ (defined in the diagram \eqref{EqDiagXandDFiber}), we obtain the following commutative diagram
\begin{equation}\label{EqDiagQuotientAndRealPoints}
\xymatrix{
\cD \ar[r]^-{p_{m,m-1}} \ar[d]_{r} & \cD_{m-1} \ar[r]^-{p_{m-1,m-2}} \ar[d]_{r_{m-1}} & \cD_{m-2} \ar[r]^-{p_{m-2,m-3}} \ar[d]_{r_{m-2}} & \cdots \ar[r]^-{p_{2,1}} & \cD_{1} \ar[r]^-{p_{1,0}} \ar[d]_{r_1} & \cD_{0} \ar[d]_{r_0} \\
\cD_{\R} \ar[r] & \cD_{m-1,\R} \ar[r] & \cD_{m-2,\R} \ar[r] & \cdots \ar[r] & \cD_{1,\R} \ar[r]   & \cD_{0,\R} 
}
\end{equation}
with each $r_k$ a $(P/W_{-k-1})(\R)^+$-equivariant retraction of $\cD_{k,\R} \subseteq \cD_k$. 
Recall that $\cD_{0}$ is a Mumford--Tate domain in a classifying space of pure Hodge structures, and $r_0$ is the identity map. There is a metric on $\cD_0$; see \cite[beginning of $\mathsection$2.1]{BTAS}.

In the proof, we often need to project subsets of $\cD$ to different levels and consider the real points. So it is convenient to fix the following notations. 

\begin{notation}\label{NotationEstimates}
For each $k \in \{0,1,\cdots,m\}$, 
\begin{itemize}
\item For any subset $A \subseteq \cD$, denote by $A_k := p_k(A) \subseteq \cD_{k}$. As convention $A_m = A$.
\item For any subset $A \subseteq \cD$, denote by $A_{\R} := r(A) \subseteq \cD_{\R}$, and $A_{k,\R} = r_k(A_k) \subseteq \cD_{k,\R}$.
\end{itemize}
\end{notation}

Let $\mathfrak{F} = r^{-1}(\mathfrak{F}_{\R})$ where $\mathfrak{F}_{\R} \subseteq \cD_{\R}$ is given by \Cref{ThmFundSet}, or more precisely by \eqref{EqFundD}.

Before moving on, let us sketch how \eqref{EqTheEstimateLeadingToFindingRealAlgCurves} is proved when $m = 0$, namely when $\cD = \cD_0$ and $P = P/W_{-1}$ is a reductive group. In this case, $\tilde{Z} = \tilde{Z}_0$, which has positive dimension by assumption. For each real number $T > 0$, take $\mathbf{B}_0(T) \subseteq \cD_{0}$ to be the ball centered at a fixed point of radius $\log T$ in $\cD_{0}$. Let $\tilde{Z}_0(T)$ be a complex analytic irreducible component of $\tilde{Z} \cap \mathbf{B}_0(T)$. The following estimate is a direct corollary of Thm.~1.2 and Thm.~4.2 of Bakker--Tsimerman \cite{BTAS}: There exist constants $c_0, \epsilon_0 > 0$, independent of $T$, such that
\[
\#\{\gamma \in \Gamma_P : \gamma \mathfrak{F} \cap \tilde{Z}_0(T) \not= \emptyset, ~ H(\gamma) \le T \} \ge c_0 T^{\epsilon_0}.
\]
See also \cite[Prop.~6.3]{BTASLectureNotes}  for the statement of this estimate. 
By \eqref{EqDistSubsetOfTheta}, the set on the left hand side is a subset of $\#\{\gamma \in \Theta \cap \Gamma_P : H(\gamma) \le T\}$. This yields \eqref{EqTheEstimateLeadingToFindingRealAlgCurves}. 

For a general $m$, we need to generalize this idea. A first thing to do is to find an appropriate generalization of $\mathbf{B}_0(T)$ for $\cD$. To achieve this, we make use of the retractions $r_k$'s (with $r_m = r$) and the following product structure on $\cD_{\R}$ \eqref{EqBundleStrOnXreal} (and the truncated version given by \Cref{PropProductStrMHD}.(ii) for each $k \in \{0,1,\cdots,m\}$)
\begin{equation}\label{EqXProdTrun}
\cD_{k,\R} \cong \cD_{0,\R} \times (W_{-1}/W_{-2})(\R) \times (W_{-2}/W_{-3})(\R) \times \cdots \times (W_{-k}/W_{-k-1})(\R).
\end{equation}
Now we are ready to give the generalization of the $\mathbf{B}_0(T)$ above. For each $k \in \{0,1,\cdots,m\}$ and each real number $T > 0$, define the following subset $\mathbf{B}_k(T) \subseteq \cD_k$ as follows.
\begin{itemize}
\item Let $\mathbf{B}_0(T) = B_0(T) \subseteq \cD_{0}$ be the ball of radius $\log T$ centered at a fixed point in $\tilde{Z}_0$.
\item For each $k \ge 1$, let $B_k(T)$ the $| \cdot |$-ball centered at $0$ of radius $T$ in $(W_{-k}/W_{-k-1})(\R)$, \textit{i.e.} $B_k(T) = \{w \in (W_{-k}/W_{-k-1})(\R): |w| < T\}$. Define $\mathbf{B}_k(T) = r_k^{-1}(\prod_{i=0}^k B_i(T))$. In particular, $p_{k+1,k}(\mathbf{B}_{k+1}(T)) = \mathbf{B}_k(T)$, and $\mathbf{B}_k(T)_{\R} = \prod_{i=0}^k B_i(T)$.
\end{itemize}

Next, we need to generalize the set $\tilde{Z}_0(T)$. For each $k \ge 0$:
\begin{itemize}
\item Let $\tilde{Z}_k(T)$ be a complex analytic irreducible component of $\tilde{Z}_k \cap \mathbf{B}_k(T) \subseteq \cD_k$.
\item We may choose such $\tilde{Z}_k(T)$'s that $p_{k+1,k} (\tilde{Z}_{k+1}(T)) \subseteq \tilde{Z}_k(T)$ for all $k$.\footnote{Notice that $\tilde{Z}_k \cap \mathbf{B}_k(T) = p_k(\tilde{Z}) \cap \mathbf{B}_k(T) = p_k(\tilde{Z} \cap p_k^{-1}(\mathbf{B}_k(T)))$. Thus $\tilde{Z}_k(T)$ equals $p_k(\tilde{Z}^k(T))$ for some complex analytic irreducible component $\tilde{Z}^k(T)$ of $\tilde{Z} \cap p_k^{-1}(\mathbf{B}_k(T))$. By definition of $\mathbf{B}_k(T)$, we have $p_{k+1}^{-1}(\mathbf{B}_{k+1}(T)) \subseteq p_k^{-1}(\mathbf{B}_k(T))$ for each $k$. Thus the $\tilde{Z}^k(T)$'s can be chosen such that $\tilde{Z}^{k+1}(T) \subseteq \tilde{Z}^k(T)$ for each $k$. For these choices, we then have $p_{k+1,k} (\tilde{Z}_{k+1}(T)) \subseteq \tilde{Z}_k(T)$.}
\end{itemize}


Finally for the purpose of lifting, we need to introduce the following sets, which generalize the set $\tilde{Z}(T)$ from \cite[proof of Thm.~5.2]{GaoAxSchanuel} (which handles the case where $m=1$). Let $k_0$ be such that $\dim \tilde{Z}_{k_0} > 0$, smallest for this property. 
For each $k \in \{k_0+1,\cdots,m\}$ and each real number $T > 0$ (the diagram \eqref{EqDiagQuotientAndRealPointsK} below, with $k$ replaced by $k-1$, is helpful to keep track of the notation):
\begin{itemize}
\item Let $\tilde{Z}(k,T):= \tilde{Z}_k \cap p_{k,k-1}^{-1}(\tilde{Z}_{k-1}(T)) \subseteq \cD_k$, and $\tilde{Z}(k,T)^+$ be a complex analytic irreducible component of $\tilde{Z}(k,T)$.
\item Similar to the $\tilde{Z}_k(T)$'s, we may choose such $\tilde{Z}(k,T)^+$'s that $p_{k+1,k} (\tilde{Z}(k+1,T)^+) \subseteq \tilde{Z}(k,T)^+$ for all $k$.
\end{itemize}
Notice that by definition, we have $p_{k,k-1}(\tilde{Z}(k,T)) = \tilde{Z}_{k-1} \cap \tilde{Z}_{k-1}(T) = \tilde{Z}_{k-1}(T) \subseteq \tilde{Z}_{k-1} \cap \mathbf{B}_{k-1}(T)$. 






\subsection{Sketch of the strategy of the proof of \Cref{ThmCounting}}\label{SubsectionSketchStrategyThmCounting}
For simplicity, we use the same notation $p_k$ to denote the projection $P \rightarrow P/W_{-k-1}$ and the projection $\cD \rightarrow \cD_k$. In the proof we need to work with many subscripts, and the following diagram is helpful to keep track of them.
\begin{equation}\label{EqDiagQuotientAndRealPointsK}
\xymatrix{
\tilde{Z}(k+1,T) \ar@{|->}[d] \subseteq & \cD_{k+1} \ar[r]^-{p_{k+1,k}} \ar[d]_{r_{k+1}} & \cD_k \ar[d]_{r_k} \\
\tilde{Z}(k+1,T)_{\R} \subseteq & \cD_{k+1,\R} \cong \cD_{k,\R} \times (W_{-k-1}/W_{-k-2})(\R)  \ar[r]_-{p_{k+1,k}|_{\cD_{k+1,\R}}} & \cD_{k,\R} 
}
\end{equation}
where the real-algebraic isomorphism $ \cD_{k+1,\R} \cong \cD_{k,\R} \times (W_{-k-1}/W_{-k-2})(\R) $ is from \eqref{EqXProdTrun}. Notice that $\tilde{Z}(k+1,T)_{\R}$ is a component of $\tilde{Z}_{k+1,\R} \bigcap (\prod_{i=0}^k B_i(T) \times (W_{-k-1}/W_{-k-2})(\R))$, and that $\tilde{Z}_k(T)_{\R}$ is a component of $\tilde{Z}_{k,\R} \bigcap \prod_{i=0}^k B_i(T)$. 


Suppose $\dim \tilde{Z}_0 = \dim p_0(\tilde{Z}) > 0$. By the results of Bakker and Tsimerman as explained above, we find $\#\{\gamma_0 \in p_0(\Gamma_P) : \gamma_0 \mathfrak{F}_0 \cap \tilde{Z}_0(T) \not= \emptyset, ~ H(\gamma_0) \le T \} \ge c_0T^{\epsilon_0}$. Consider the diagram \eqref{EqDiagQuotientAndRealPointsK} with $k=0$. We wish to lift at least polynomially many such $\gamma_0$'s to elements in $p_1(\Gamma_P)$ of height at most $T$ with the following property: each such lift $\gamma_1 \in p_1(\Gamma_P)$ satisfies $\gamma_1\mathfrak{F}_1 \cap \tilde{Z}(1,T) \not= \emptyset$, or equivalently $\gamma_1 r_1(\mathfrak{F}_1) \cap r_1(\tilde{Z}(1,T)) \not=\emptyset$ (since $\mathfrak{F}_1 = r_1^{-1}(\mathfrak{F}_{1,\R})$ by definition of $\mathfrak{F}$). This last condition, expressed with \Cref{NotationEstimates}, becomes $\gamma_1 \mathfrak{F}_{1,\R}\cap \tilde{Z}(1,T)_{\R} \not=\emptyset$. The intersection is taken in $\cD_{1,\R} \cong \cD_0 \times (W_{-1}/W_{-2})(\R)$. If the desired lifting can be realized, then we do similar liftings to $p_2(\Gamma_P)$, \textit{etc.}, under we obtain at least polynomially many elements $\gamma$ in $p_m(\Gamma_P) = \Gamma_P$ of height at most $T$ such that $\gamma\mathfrak{F}_{\R} \cap \tilde{Z}(T)_{\R} \not= \emptyset$.

At this stage, we can explain why the second bullet point in the constructions of the $\tilde{Z}(k,T)$'s is needed: in the lifting process, we need that $\tilde{Z}(k+1,T)_{\R}$ is mapped into $\tilde{Z}(k,T)_{\R}$ under $p_{k+1,k}$.

There is a problem in the procedure described above, namely it is possible that $\tilde{Z}_0$ is a point. In this case, we need to work with the smallest $k_0$ such that $\dim \tilde{Z}_{k_0} > 0$, which serves as the base step of the lifting process. Thus we need to introduce the set $\tilde{Z}_{k_0}(T)$, which is a complex analytic irreducible component of $\tilde{Z}_{k_0} \cap \mathbf{B}_{k_0}(T)$. We need to find at least polynomially many elements $\gamma_{k_0} \in p_{k_0}(\Gamma_P)$ of height at most $T$ such that $\gamma_{k_0} \mathfrak{F}_{k_0,\R} \cap \tilde{Z}_{k_0}(T)_{\R} \not= \emptyset$. Whereas this is guaranteed by the result of Bakker and Tsimerman when $k_0 = 0$, it is not known when $k_0 \ge 1$. We will prove this result in \Cref{SubsectionCountingBaseStep}, or more precisely \Cref{LemmaCountingBaseStep}.(ii).

Once we have established the base step, we need to realize the lifting. 
In view of \eqref{EqDiagQuotientAndRealPointsK}, 
in order to realize the lifting process from $k$ to $k+1$, we need to compare the growth of $\tilde{Z}(k+1,T)_{\R} \subseteq \cD_{k+1,\R}$ in the vertical direction $(W_{-k-1}/W_{-k-2})(\R)$ with its growth in the horizontal direction $\cD_{k,\R}$. This lifting process is done in \Cref{SubsectionEstimateLifting}. As in \cite[proof of Thm.~5.2]{GaoAxSchanuel}, we will divide into the two cases where $\tilde{Z}(k+1,T)_{\R}$ grows ``faster'' in the vertical direction $(W_{-k-1}/W_{-k-2})(\R)$ (\Cref{LemmaLiftingCaseVertical}) and where $\tilde{Z}(k+1,T)_{\R}$ grows ``faster'' in the horizontal direction $\cD_{k,\R}$ (the rest of \Cref{SubsectionEstimateLifting}).

\subsection{Proof of \Cref{ThmCounting}: the base step}\label{SubsectionCountingBaseStep}
The main goal of this subsection is to prove the base step for the lifting process, namely \Cref{LemmaCountingBaseStep}. At the end of this subsection we also state the result for the lifting process (\Cref{PropCounting}) and explain how it implies \Cref{ThmCounting}. The proof of the lifting process will be executed in the next subsection.

Let $k_0 \in \{0,\cdots,m\}$ be such that $\dim \tilde{Z}_{k_0} > 0$, smallest for this property. 
For simplicity, we introduce the following notation. For each real number $T \ge 0$, let 
\begin{align}\label{EqDefinitionXik_01time}
\Xi_{k_0}(T) & :=  \{g \in (W_{-k_0}/W_{-k_0-1})(\R) : g \mathfrak{F}_{k_0} \cap \tilde{Z}_{k_0}(T) \not=\emptyset\} \\
& = \{g \in (W_{-k_0}/W_{-k_0-1})(\R) : g \mathfrak{F}_{k_0,\R} \cap \tilde{Z}_{k_0}(T)_{\R} \not=\emptyset\}. \nonumber
\end{align}
Here the second equality holds true since $\mathfrak{F}_{k_0} = r_{k_0}^{-1}( \mathfrak{F}_{k_0,\R})$. 

We also denote by $\Gamma_{-k_0/-k_0-1} = (\Gamma_P \cap W_{-k_0}(\Q)) / (\Gamma_P \cap W_{-k_0-1}(\Q))$; it is a subgroup of $P/W_{-k_0-1}$ and acts on $\cD_{k_0} = \cD/W_{-k_0-1}$.

\begin{prop}\label{LemmaCountingBaseStep}
There exist constants $c_{k_0}, \epsilon_{k_0} > 0$ such that
\[
\#\{\gamma_{-k_0/-k_0-1} \in \Xi_{k_0}(T) \cap\Gamma_{-k_0/-k_0-1} : H(\gamma_{-k_0/-k_0-1}) \le T \} \ge c_{k_0} T^{\epsilon_{k_0}}\quad \text{ for all }T \gg 1.
\]
\end{prop}
\begin{proof}[Proof of \Cref{LemmaCountingBaseStep}] 


If $k_0 = 0$, this is proved in \cite{BTAS}. We refer to \cite[Prop.~6.3]{BTASLectureNotes} for a precise statement.

From now on, assume $k_0 \ge 1$. We use \eqref{EqDiagQuotientAndRealPointsK} with $k = k_0-1$. Now $\tilde{Z}_{k_0-1} = \bar{h}$ is a point in $\cD_{k_0-1}$. Thus $\tilde{Z}_{k_0} \subseteq p_{k_0,k_0-1}^{-1}(\bar{h})$. Notice that $r_{k_0}(p_{k_0,k_0-1}^{-1}(\bar{h}))$ can be identified with $(W_{-k_0}/W_{-k_0-1})(\R)$.



\begin{lemma}\label{LemmaApproximationXiT}
Recall $M > 0$ the real number in the definition of $\mathfrak{F}_{\R}$ from \Cref{ThmFundSet}. Denote for simplicity $\mathfrak{F}'_{k_0} = (-M,M)^{\dim (W_{-k_0}/W_{-k_0-1})(\R)} \subseteq  (W_{-k_0}/W_{-k_0-1})(\R)$. Then any $\gamma_{-k_0/-k_0-1} \in  \Xi_{k_0}(T) \cap \Gamma_{-k_0/-k_0-1}$ satisfies $H(\gamma_{-k_0/-k_0-1}) \le T+M$.
\end{lemma}

\begin{proof}[Proof of \Cref{LemmaApproximationXiT}] 
We have $\tilde{Z}_{k_0,\R} = r_{k_0}(\tilde{Z}_{k_0}) \subseteq r_{k_0}(p_{k_0,k_0-1}^{-1}(\bar{h})) =  (W_{-k_0}/W_{-k_0-1})(\R)$. Recall the definition $\mathbf{B}_{k_0}(T) = r_{k_0}^{-1}(\prod_{i=0}^{k_0} B_i(T))$. So 
\scriptsize
\begin{equation}\label{EqLemmaApproximationXiT}
 \Xi_{k_0}(T) \cap \Gamma_{-k_0/-k_0-1} = \{\gamma_{-k_0/-k_0-1} \in \Gamma_{-k_0/-k_0-1} : \left(\gamma_{-k_0/-k_0-1} + \mathfrak{F}'_{k_0} \right) \cap \tilde{Z}_{k_0}(T)_{\R} \not= \emptyset\} .
\end{equation}
\normalsize
Hence each $\gamma_{-k_0/-k_0-1} \in  \Xi_{k_0}(T) \cap \Gamma_{-k_0/-k_0-1}$ satisfies $\left(\gamma_{-k_0/-k_0-1} + \mathfrak{F}'_{k_0} \right) \cap B_{k_0}(T)_{\R} \not= \emptyset$. We are done.
\end{proof}

Now we are ready to finish the proof of \Cref{LemmaCountingBaseStep}.


Consider $\{\gamma_{-k_0/-k_0-1} \in \Gamma_{-k_0/-k_0-1} : (\gamma_{-k_0/-k_0-1} + \mathfrak{F}'_{k_0}) \cap \tilde{Z}_{k_0,\R} \not= \emptyset\}$. We claim that it is infinite. Indeed, assume otherwise, then $\tilde{Z}_{k_0,\R}$ is contained in a bounded subset of $(W_{-k_0}/W_{-k_0-1})(\R)$. But $p_{k_0,k_0-1}^{-1}(\bar{h}) \cong (W_{-k_0}/W_{-k_0-1})(\C)/F^0_{\bar{h}}(W_{-k_0}/W_{-k_0-1})_{\C}$ by part (ii) of \Cref{LemmaFibersOfVarphikPkkplus1}, and the composite ($\varphi_{k_0}$ is the natural projection)
\begin{align}\label{EqComplexStructurecDEstimate}
(W_{-k_0}/W_{-k_0-1})(\C) & \xrightarrow{\varphi_{k_0}}  (W_{-k_0}/W_{-k_0-1})(\C) / F^0_{\bar{h}}(W_{-k_0}/W_{-k_0-1})_{\C} = p_{k_0,k_0-1}^{-1}(\bar{h}) \nonumber \\
& \xrightarrow{r_{k_0}}  (W_{-k_0}/W_{-k_0-1})(\R)
\end{align}
is, up to an automorphism of $(W_{-k_0}/W_{-k_0-1})(\R)$ sending bounded sets to bounded sets, the projection to the real part.\footnote{Recall that $r_{k_0}$ is the retraction given by the $\mathfrak{sl}_2$-splitting. If $r_{k_0}$ is replaced by the retraction induced by the Deligne $\delta$-splitting, then this composite is precisely the projection to the real part. But the $\mathfrak{sl}_2$-splitting is defined by universal Lie polynomials in the Hodge components of the Deligne $\delta$-splitting, so this claim holds true.} So $\varphi_{k_0}^{-1}(\tilde{Z}_{k_0}) \subseteq \varphi_{k_0}^{-1}(r_{k_0}^{-1}(\tilde{Z}_{k_0,\R}))$ is contained in a set whose real part is bounded. But $\varphi_{k_0}^{-1}(\tilde{Z}_{k_0})$ is complex analytic, so $\varphi_{k_0}^{-1}(\tilde{Z}_{k_0})$ is a point, and so is $\tilde{Z}_{k_0}$. This contradicts $\dim \tilde{Z}_{k_0} > 0$.

Next we claim that $\tilde{Z}_{k_0}(T)_{\R}$ passes through the boundary of $B_{k_0}(T)$. Assume otherwise, then $\tilde{Z}_{k_0,\R} \setminus \tilde{Z}_{k_0}(T)_{\R}$ and $\tilde{Z}_{k_0}(T)_{\R}$ are disjoint. 
But $\tilde{Z}_{k_0,\R}$ is connected since $\tilde{Z}_{k_0}$ is irreducible. So we must have $\tilde{Z}_{k_0,\R} = \tilde{Z}_{k_0}(T)_{\R}$. Hence $\tilde{Z}_{k_0,\R}$ is contained in a bounded subset of $(W_{-k_0}/W_{-k_0-1})(\R)$. This yields a contradiction by the same argument in the previous paragraph.


Note that $\mathfrak{F}'_{k_0}$ is a fundamental set for the action of $\Gamma_{-k_0/-k_0-1}$ on the Euclidean space $(W_{-k_0}/W_{-k_0-1})(\R)$. The claims in the previous two paragraphs together immediately imply that 
\[
\#\{\gamma_{-k_0/-k_0-1} \in \Gamma_{-k_0/-k_0-1} : (\gamma_{-k_0/-k_0-1} + \mathfrak{F}'_{k_0}) \cap \tilde{Z}_{k_0}(T)_{\R} \not= \emptyset\} \ge T
\]
for all $T \gg 1$. Now the conclusion follows from \Cref{LemmaApproximationXiT} with $c_{k_0} =1/2$ and  $\epsilon_{k_0} = 1$.
\end{proof}

\begin{rmk}\label{RmkUseOfComplexStructureInTheEstimate}
The proof of \Cref{LemmaCountingBaseStep} is the only place in the proof of \Cref{ThmCounting} where we use the complex structure of $\cD$. More precisely, the complex structure of $\cD$ is used only in the citation of \cite{BTAS} (if $k_0 = 0$) and in the paragraph involving \eqref{EqComplexStructurecDEstimate} (if $k_0 \ge 1$).
\end{rmk}

\subsection{A preliminary lifting process}

Let $k \ge 0$. For simplicity denote by $W_0 := P$. The following diagram is useful to keep track of the notations.
\small
\begin{equation}\label{EqDiagQuotientAndRealPointsKExpand}
\xymatrix{
 (P/W_{-k-2}, \cD_{k+1}) \ar[r]^-{p_{k+1,k}} \ar[d]_{r_{k+1}} & (P/W_{-k-1}, \cD_{k}) \ar[d]^{r_{k}} \\
 \cD_{k+1,\R} \cong \cD_{k,\R} \times (W_{-k-1}/W_{-k-2})(\R) \ar[d]_{\lambda_{k+1}} \ar[r]^-{p_{k+1,k}|_{\cD_{k+1,\R}}} & \cD_{k,\R} \\
 (W_{-k-1}/W_{-k-2})(\R) 
}
\end{equation}
\normalsize
 where the real-algebraic isomorphism $ \cD_{k+1,\R} \cong \cD_{k,\R} \times (W_{-k-1}/W_{-k-2})(\R) $ is from \eqref{EqXProdTrun}, and  $\lambda_{k+1}$ is the natural projection.

Consider the isomorphism of $\Q$-varieties given by \eqref{EqDecompGroup} $P/W_{-k-2} \cong G \times (W_{-1}/W_{-2}) \times \cdots \times (W_{-k-1}/W_{-k-2})$. 
It induces
\begin{equation}\label{Eq}
P/W_{-k-2} \cong P/W_{-k-1} \times W_{-k-1/-k-2}.
\end{equation}

The group $(P/W_{-k-2})(\R)^+$ acts on $\cD_{k+1,\R} = (\cD/W_{-k-2})_{\R}$. Write
\[
\Gamma_{0/-k-2} = \Gamma_P/(\Gamma_P \cap W_{-k-2}(\Q)).
\]

\begin{lemma}\label{LemmaNormPolynomialGrowth}
There exists a constant $\beta_k > 0$ with the following property. 
Consider the Euclidean norm $| \cdot |$ on $(W_{-k-1}/W_{-k-2})(\R)$. Then for any $\gamma_{0/-k-2} \in \Gamma_{0/-k-2}$, the set $\lambda_{k+1}(\gamma_{0/-k-2} \mathfrak{F}_{k+1,\R})$ is contained in a $|\cdot|$-ball of radius $\le \beta_k H(\gamma_{0/-k-1})^{k}$ in $(W_{-k-1}/W_{-k-2})(\R)$. Here, $\gamma_{0/-k-1} \in \Gamma_{0/-k-1}$ is the projection of $\gamma_{0/-k-2}$ under the natural projection $P/W_{-k-2} \rightarrow P/W_{-k-1}$.

Moreover, if we denote by $(\gamma_{0/-k-1}, \gamma_{-k-1/-k-2})$ the image of $\gamma_{0/-k-2}$ under the isomorphism \eqref{Eq}, then the $|\cdot|$-ball mentioned above can be taken to be centered at $\gamma_{-k-1/-k-2}$.
\end{lemma}

Before proving \Cref{LemmaNormPolynomialGrowth}, let us see an application.
\begin{lemma}\label{LemmaHtNormLatticePts}
There exist constants $\alpha_k >0$ and $\alpha'_k > 0$ satisfying the following property. If $\gamma_{0/-k-1} \in \Gamma_{0/-k-1}$ satisfies $\gamma_{0/-k-1} \mathfrak{F}_{k,\R} \cap \prod_{i=0}^{k} B_i(T) \not= \emptyset$, then $H(\gamma_{0/-k-1}) \le \alpha'_k T^{\alpha_k}$.
\end{lemma}
\begin{proof}[Proof of \Cref{LemmaHtNormLatticePts}] We prove this lemma by upward induction on $k \ge 0$. The base step is $k = 0$, which is precisely \cite[Thm.~4.2]{BTAS}.

Use the notation from \eqref{EqDiagQuotientAndRealPointsKExpand}.


Assume \Cref{LemmaHtNormLatticePts} holds true for $k$, namely $H(\gamma_{0/-k-1}) \le \alpha'_k T^{\alpha_k}$ for each $\gamma_{0/-k-1} \in \Gamma_{0/-k-1}$ with $\gamma_{0/-k-1} \mathfrak{F}_{k,\R} \cap \prod_{i=0}^{k} B_i(T) \not= \emptyset$. We wish to prove the property for $k+1$.

Let $\gamma_{0/-k-2} \in \Gamma_{0/-k-2}$ be such that $\gamma_{0/-k-2} \mathfrak{F}_{k+1,\R} \cap \prod_{i=0}^{k+1} B_i(T) \not= \emptyset$. Denote by $(\gamma_{0/-k-1},\gamma_{-k-1/-k-2})$ the image of $\gamma_{0/-k-2}$ under the isomorphism \eqref{Eq}. In particular, $\gamma_{0/-k-1}$ is the image of $\gamma_{0/-k-2}$ under the projection $p_{k+1,k} \colon P/W_{-k-2} \rightarrow P/W_{-k-1}$.

Applying $p_{k+1,k}$ to both sides of $\gamma_{0/-k-2} \mathfrak{F}_{k+1,\R} \cap \prod_{i=0}^{k+1} B_i(T) \not= \emptyset$, we obtain that $\gamma_{0/-k-1} \mathfrak{F}_{k,\R} \cap \prod_{i=0}^{k} B_i(T) \not= \emptyset$. Thus by induction  hypothesis, we have $H(\gamma_{0/-k-1}) \le \alpha'_k T^{\alpha_k}$.

Next, 
applying $\lambda_{k+1}$ to both sides of $\gamma_{0/-k-2} \mathfrak{F}_{k+1,\R} \cap 
\prod_{i=0}^{k+1} B_i(T) \not=\emptyset$, we then have 
\begin{equation}\label{EqFundFiberTranslateIntersectBall}
\lambda_{k+1}(\gamma_{0/-k-2} \mathfrak{F}_{k+1,\R}) \cap B_{k+1}(T) \not=\emptyset.
\end{equation}
By \Cref{LemmaNormPolynomialGrowth},  $\lambda_{k+1}(\gamma_{0/-k-2} \mathfrak{F}_{k+1,\R})$ is contained in a $|\cdot|$-ball of radius $\beta_k H(\gamma_{0/-k-1})^{k}$ centered at $\gamma_{-k-1/-k-2}$. But $H(\gamma_{0/-k-1}) \le \alpha'_k T^{\alpha_k}$. Therefore the following two $|\cdot|$-balls in the Euclidean space $(W_{-k-1}/W_{-k-2})(\R)$ intersect: the one of radius $T$ centered at $0$, and the one of radius $\beta_k \alpha_k^{\prime k} T^{k \alpha_k}$ centered at $\gamma_{-k-1/-k-2}$. So $H(\gamma_{-k-1/-k-2}) \le T+\beta_k \alpha_k^{\prime k}  T^{k \alpha_k}$.

Thus the proposition holds true with $\alpha'_{k+1} := 1+\beta_k \alpha_k^{\prime k}$ and $\alpha_{k+1} = \max\{1,k\alpha_k\}$.
\end{proof}

We end this subsection with:

\begin{proof}[Proof of \Cref{LemmaNormPolynomialGrowth}]
Let $\gamma_{0/-k-2} \in \Gamma_{0/-k-2}$. Use the notation as in the lemma, namely $\gamma_{0/-k-2} \mapsto (\gamma_{0/-k-1} , \gamma_{-k-1/-k-2})$ under the isomorphism of $\Q$-varieties $P/W_{-k-2} \cong P/W_{-k-1} \times (W_{-k-1}/W_{-k-2})$ \eqref{Eq}.

Next, consider the isomorphism of $\Q$-varieties induced by \eqref{EqDecompGroup}
\[
P/W_{-k-2} \cong G \times (W_{-1}/W_{-2}) \times \cdots \times (W_{-k-1}/W_{-k-2}), 
\]
and suppose $\gamma_{0/-k-2} \mapsto (\gamma_{0/-1}, \gamma_{-1/-2}, \ldots, \gamma_{-k-1/-k-2})$ under this isomorphism. Then $H(\gamma_{0/-k-2}) = \max\{H(\gamma_{0/-1}), \ldots, H(\gamma_{-k-1/-k-2})\}$.

On the other hand, for the real-algebraic morphism induced by \eqref{EqXProdTrun}
\[
\cD_{k+1,\R} \cong \cD_0 \times (W_{-1}/W_{-2})(\R) \times \cdots \times (W_{-k-1}/W_{-k-2})(\R) \times (W_{-k-1}/W_{-k-2})(\R),
\]
we have defined, in \eqref{EqFundD}, $\mathfrak{F}_{k+1,\R}$ to be the inverse image of $\mathfrak{F}_0 \times \mathfrak{F}'_1 \times \cdots  \times \mathfrak{F}'_{k+1}$, where $M$ is a fixed real number and
\[
\mathfrak{F}'_i = (-M,M)^{\dim (W_{-i}/W_{-i-1})(\R)} \subseteq  (W_{-i}/W_{-i-1})(\R).
\]
The formula for the action of the group $(P/W_{-k-2})(\R)^+$ on $\cD_{k+1,\R}$ is given by \Cref{PropProductStrMHD}, or more precisely \eqref{EqFormulaActionBCH}. Thus
\scriptsize
\begin{align}\label{EqEstimateProjectionSet}
& \lambda_{k+1}(\gamma_{0/-k-2} \mathfrak{F}_{k+1,\R}) \nonumber \\		
 =&  \{ \gamma_{-k-1/-k-2} + \gamma_{0/-1} \tilde{x} + \mathrm{calb}_{k+1}(\gamma_{0/-k-1} , \gamma_{0/-1} \tilde{x}') : \tilde{x} \in \mathfrak{F}'_{k+1}, ~ \tilde{x}' \in \mathfrak{F}_0 \times \mathfrak{F}'_1 \times \cdots \times \mathfrak{F}'_{k}\} \\
= & \gamma_{-k-1/-k-2} + \gamma_{0/-1} \cdot \mathfrak{F}'_{k+1} + \mathrm{calb}_{k+1}(\gamma_{0/-k-1}, \gamma_{0/-1}  \mathfrak{F}_0 \times \gamma_{0/-1}  \mathfrak{F}'_1 \times \cdots \times \gamma_{0/-1} \mathfrak{F}'_{k} ),		\nonumber
\end{align}
\normalsize
where $\mathrm{calb}_{k+1}$ is a polynomial of degree at most $k$. 
Notice that $M$, $\mathfrak{F}_0$ and the $\mathfrak{F}'_i$'s are fixed, and that $H(\gamma_{0/-1}) \le H(\gamma_{0/-k-1})$. So
\small
\[
|\gamma_{0/-1} \cdot \mathfrak{F}'_{k+1} + \mathrm{calb}_{k+1}(\gamma_{0/-k-1}, \gamma_{0/-1} (\mathfrak{F}_0 \times \mathfrak{F}'_1 \times \cdots \mathfrak{F}'_{k}) )| \ll H(\gamma_{0/-1}) + H(\gamma_{0/-k-1})^{k} \ll H(\gamma_{0/-k-1})^{k}.
\]
\normalsize
Therefore, by \eqref{EqEstimateProjectionSet}, $\lambda_{k+1}(\gamma_{0/-k-2} \mathfrak{F}_{k+1,\R})$ is contained in the $|\cdot|$-ball of radius $\ll H(\gamma_{0/-k-1})^{k}$ centered at $ \gamma_{-k-1/-k-2}$. Hence we are done.
\end{proof}

\subsection{Setup and final conclusion of the lifting process}


For each $k \in \{k_0, \ldots, m\}$ and each real number $T \ge 0$:
\begin{itemize}
\item Let $\Gamma_{-k_0/-k-1} = (\Gamma_P \cap W_{-k_0}(\Q))/(\Gamma_P \cap W_{-k-1}(\Q))$. Then $\Gamma_{-k_0/-k-1}$ acts on $\cD_{k} = \cD/W_{-k-1}$.
\item If furthermore $k \ge k_0+1$, then set
\begin{align*} 
\qquad \Xi_{k}(T) & = \{g \in (W_{-k_0}/W_{-k-1})(\R) : g \mathfrak{F}_{k} \cap \tilde{Z}(k,T) \not= \emptyset\} \\
& =\{g \in (W_{-k_0}/W_{-k-1})(\R) : g \mathfrak{F}_{k,\R} \cap \tilde{Z}(k,T)_{\R} \not= \emptyset\}  \subseteq (P/W_{-k-1})(\R),
\end{align*} 
where the second equality follows from the construction $\mathfrak{F}_k = r_k^{-1}(\mathfrak{F}_{k,\R})$.
\end{itemize}

Notice that the above definition of $\Xi_k(T)$ for $k \ge k_0+1$ differs in an essential way from the definition for $\Xi_{k_0}(T)$  given in \eqref{EqDefinitionXik_01time}, because for $\Xi_{k_0}(T)$ we intersect with $\tilde{Z}_{k_0}(T)$ while in case $k \ge k_0+1$ the intersection is with $\tilde{Z}(k,T)_{\R}$ (which is larger than $\tilde{Z}_k(T)_{\R}$). This seemingly strange convention will make the induction step of the proof of \Cref{PropCounting} much cleaner.

Now we are ready to state the desired lifting proposition.

\begin{prop}\label{PropCounting}
For each $k \ge k_0$, 
there exist constants $c_k, \epsilon_k> 0$, and a sequence $\{T_i \in \R\}_{i \in \N}$ with $T_i \rightarrow \infty$, such that
\[
\#\{\gamma_{-k_0/-k-1} \in \Xi_{k}(T_i) \cap \Gamma_{-k_0/-k-1} : H(\gamma_{-k_0/-k-1}) \le T_i \} \ge c_k T_i^{\epsilon_k}.
\]
\end{prop}

Let us finish the proof of \Cref{ThmCounting} assuming this proposition. 
\begin{proof}[Proof of \Cref{ThmCounting} assuming \Cref{PropCounting}]
Apply \Cref{PropCounting} to $k = m$. As $W_{-(m+1)} = 0$, the conclusion of the proposition becomes: there exist constants $c=c_m, \epsilon=\epsilon_m > 0$ and a sequence $\{T_i \in \R\}_{i \in \N}$ with $T_i \rightarrow \infty$, such that
\[
\#\{\gamma \in \Gamma_{-k} : H(\gamma) \le T_i , ~ \gamma\mathfrak{F} \cap \tilde{Z}(m,T_i) \not= \emptyset\} \ge c T_i^{\epsilon}.
\]
But $\Gamma_{-k_0} \subseteq \Gamma_P$ and $\tilde{Z}(m,T_i) \subseteq \tilde{Z}$, and so
\[
\# \{\gamma \in \Gamma_P : H(\gamma) \le T_i, ~ \gamma\mathfrak{F} \cap \tilde{Z} \not= \emptyset\} \ge c T_i^{\epsilon}.
\]
Thus we can conclude by \eqref{EqDistSubsetOfTheta}.
\end{proof}

\subsection{Proof of \Cref{PropCounting}}\label{SubsectionEstimateLifting}
We prove in this subsection \Cref{PropCounting} with a lifting process. The following diagram helps to keep track of the lifting, with $k \ge k_0$:
\scriptsize
\begin{equation}\label{EqBigDiagramLiftingChase}
\xymatrix{
&  \Xi_{k+1}(T)  \ar@{}[d]|-*[@]{\subseteq} \ar@{|->}[r] & p_{k+1,k}(\Xi_{k+1}(T)) \ar@{}[d]|-*[@]{\subseteq} \\
\tilde{Z}(k+1,T) \subseteq \ar[d] & (P/W_{-k-2}, \cD_{k+1}) \ar[r]^-{p_{k+1,k}} \ar[d]_{r_k} & (P/W_{-k-1}, \cD_{k}) \ar[d]^{r_{k-1}}  & \ar@{}[l]|-*[@]{\subset} \mathfrak{F}_{k}, ~  \mathbf{B}_{k}(T) \ar@{|->}[d] \\
\tilde{Z}(k+1,T)_{\R} \subseteq & \cD_{k+1,\R} \cong \cD_{k,\R} \times (W_{-k-1}/W_{-k-2})(\R) \ar[d]^-{\lambda_{k+1}} \ar[r]^-{p_{k+1,k}|_{\cD_{\R}}} & \cD_{k,\R} & \ar@{}[l]|-*[@]{\subset} \mathfrak{F}_{k,\R}, ~  \prod_{i=0}^{k} B_i(T). \\
&  (W_{-k-1}/W_{-k-2})(\R) 
}
\end{equation}
\normalsize
where the vertical inclusions are inclusions in the $\R$-points of the underlying groups, and the horizontal inclusions are inclusions in the underlying spaces.

We start with the following observation:
\begin{equation}\label{LemmaProjectionXiLevel}
p_{k+1,k}(\Xi_{k+1}(T)) \subseteq \{g \in (W_{-k}/W_{-k-1})(\R) : g \mathfrak{F}_{k,\R} \cap \prod_{i=0}^{k} B_i(T) \not= \emptyset\}.
\end{equation}
Indeed, take $g \in p_{k+1,k}(\Xi_{k+1}(T))$. Then $g = p_{k+1,k}(g')$ for some $g' \in \Xi_{k+1}(T)$. Thus $g' \mathfrak{F}_{k+1,\R}\cap \tilde{Z}(k+1,T)_{\R} \not=\emptyset$. Applying $p_{k+1,k}$ to both sides, we get $g \mathfrak{F}_{k,\R} \cap p_{k+1,k}(\tilde{Z}(k+1,T)_{\R}) \not=\emptyset$. But $p_{k+1,k}(\tilde{Z}(k+1,T)_{\R}) = p_{k+1,k}(r_{k+1}(\tilde{Z}(k+1,T)) = r_k(p_{k+1,k}(\tilde{Z}(k+1,T))) \subseteq r_k(\tilde{Z}_k \cap \mathbf{B}_k(T)) \subseteq r_k(\mathbf{B}_k(T)) = \prod_{i=0}^{k} B_i(T)$. Hence we get \eqref{LemmaProjectionXiLevel}.

\vskip 0.5em

For each $\alpha_k$ from \Cref{LemmaHtNormLatticePts}, let us fix a number $\delta_k$ such that
\begin{equation}\label{EqDeltaCondition}
\delta_k > (k+1) \alpha_k.
\end{equation}

\vskip 0.5em

Now we proceed to the proof. We start with the following lemma, 
which handles the case where the vertical direction of $\tilde{Z}(k+1,T_i)$ grows faster than the horizontal direction.

\begin{lemma}\label{LemmaLiftingCaseVertical}
Assume there exists $k \ge k_0$ such that 
\begin{equation}\label{EqHypVerticalGrowth}
|\lambda_{k+1}(\tilde{Z}(k+1,T)_{\R})| > T^{\delta_k} \qquad\text{ for all }T \gg 1.
\end{equation}
Then 
 there exist constants $c_{k+1} >0$ and $\epsilon_{k+1} >0$, both independent of $T$, such that
\begin{equation}\label{EqLiftingVerticalGrowthCase}
\#\{\gamma_{-k_0/-k-2} \in \Xi_{k+1}(T) \cap \Gamma_{-k_0/-k-2} : H(\gamma_{-k_0/-k-2}) \le T\} \ge c_{k+1}T^{\epsilon_{k+1}}.
\end{equation}
\end{lemma}
\begin{proof}[Proof of \Cref{LemmaLiftingCaseVertical}]
Use the notation from \eqref{EqBigDiagramLiftingChase}. By definition \eqref{EqFundD}, we have $\mathfrak{F}_{k+1,\R} = \mathfrak{F}_{k,\R} \times (-M,M)^{\dim (W_{-k-1}/W_{-k-2})(\R)}$ for some fixed real number $M>0$.

Take $\gamma_{-k_0/-k-2} \in \Xi_{k+1}(T) \cap \Gamma_{-k_0/-k-2}$. Write $(\gamma_{-k_0/-k-1},\gamma_{-k-1/-k-2})$ for the image of $\gamma_{-k_0/-k-2}$ under the isomorphism \eqref{Eq} $P/W_{-k-2} \cong P/W_{-k-1}\times W_{-k-1/-k-2}$. Then $\gamma_{-k_0/-k-1} = p_{k+1,k}(\gamma_{-k_0/-k-2})$ for the group morphism $p_{k+1,k} \colon P/W_{-k-2} \rightarrow P/W_{-k-1}$. Thus $H(\gamma_{-k_0/-k-1}) \le \alpha'_k T^{\alpha_k}$ by \eqref{LemmaProjectionXiLevel} and \Cref{LemmaHtNormLatticePts}.

We temporarily work in the Euclidean space $(W_{-k-1}/W_{-k-2})(\R)$. Our hypothesis \eqref{EqHypVerticalGrowth} says that $\lambda_{k+1}(\tilde{Z}(k+1,T_i)_{\R})$ reaches the boundary of the $|\cdot|$-ball of radius $T_i^{\delta_k}$ centered at $0$.

We claim that $\lambda_{k+1}(\tilde{Z}(k+1,T_i)^+_{\R})$ also reaches the boundary of the $|\cdot|$-ball of radius $T_i^{\delta_k}$ centered at $0$ (which by our notation is $B_{k+1}(T_i^{\delta_k})$). Assume otherwise, then $\lambda_{k+1}(\tilde{Z}_{k+1,\R}) \setminus B_{k+1}(T_i^{\delta_k})$ and $\lambda_{k+1}(\tilde{Z}(k+1,T_i)^+_{\R})$ are disjoint. But $\lambda_{k+1}(\tilde{Z}_{k+1,\R})$ is connected because $\tilde{Z}$ is irreducible. So we must have $\lambda_{k+1}(\tilde{Z}_{k+1,\R}) = \lambda_{k+1}(\tilde{Z}(k+1,T_i)^+_{\R})$. By definition, we have $\tilde{Z}(k+1,T)_{\R} \subseteq \tilde{Z}_{k+1,\R}$. Thus $\lambda_{k+1}(\tilde{Z}(k+1,T_i)^+_{\R}) =  \lambda_{k+1}(\tilde{Z}_{k+1,\R})  \supseteq \lambda_{k+1}(\tilde{Z}(k+1,T)_{\R}) \supseteq \lambda_{k+1}(\tilde{Z}(k+1,T_i)^+_{\R}) $, and hence every inclusion is an equality. Now we are done by the assumption \eqref{EqHypVerticalGrowth}.

On the other hand by \Cref{LemmaNormPolynomialGrowth}, the subset $\lambda_{k+1}(\gamma_{-k_0/-k-2}\mathfrak{F}_{k+1,\R} \cap \tilde{Z}(k+1,T_i)_{\R})$ is contained in a $|\cdot|$-ball of radius $\ll H(\gamma_{-k_0/-k-1})^{k+1} \ll T_i^{(k+1)\alpha_k}$ centered at $\gamma_{-k-1/-k-2}$.

Since $(W_{-k-1}/W_{-k-2})(\R)$ is Euclidean and $\delta_k > (k+1)\alpha_k$ and that $\lambda_{k+1}(\tilde{Z}(k+1,T_i)^+_{\R})$ is connected, the previous two paragraphs together imply
\begin{align*}\label{}
\#\{\gamma_{-k-1/-k-2} \in & \Gamma_{-k-1/-k-2} :  H(\gamma_{-k-1/-k-2}) \le T^{\delta_k}, ~ (\gamma_{-k_0/-k-1}, \gamma_{-k-1/-k-2}) \in \Xi_k(T) \\
& \text{ for some }\gamma_{-k_0/-k-1} \in p_{k+1,k}(\Xi_{k+1}(T)) \cap \Gamma_{-k_0/-k-1} \} \gg T^{\delta_k - (k+1)\alpha_k}
\end{align*}
for all $T\gg 1$. As each $\gamma_{-k_0/-k-1} \in p_{k+1,k}(\Xi_{k+1}(T)) \cap \Gamma_{-k_0/-k-1}$ satisfies $H(\gamma_{-k_0/-k-1}) \le \alpha'_k T^{\alpha_k}$ by \eqref{LemmaProjectionXiLevel} and \Cref{LemmaHtNormLatticePts}, the counting above yields, for all $T \gg 1$,
\begin{align*}
\#\{(\gamma_{-k_0/-k-1},\gamma_{-k-1/-k-2}) \in \Xi_k(T) \cap \Gamma_{-k_0/-k-2} : & H(\gamma_{-k-1/-k-2}) \le T^{\delta_k}, \\
 & H(\gamma_{-k_0/-k-2}) \le \alpha'_k T^{\alpha_k} \}  \gg T^{\delta_k - (k+1)\alpha_k}.
\end{align*}
But the only assumption on $\delta_k$ is $\delta_k > (k+1) \alpha_k$. Hence we have proved \eqref{EqLiftingVerticalGrowthCase} by choosing appropriately $c_{k+1}$ and $\epsilon_{k+1}$. We are done.
\end{proof}

The next lemma is useful for the lifting in the case where the horizontal direction of $\tilde{Z}(k+1,T_i)$ grows faster than the vertical direction.

\begin{lemma}\label{LemmaHorizontalGrowth}
Let $k \ge k_0$. Assume there exists a sequence $\{T_i \in \R\}$, with $T_i \rightarrow \infty$, such that
\begin{equation}\label{EqHypHorizontalGrowth}
|\lambda_{k+1}(\tilde{Z}(k+1,T_i)_{\R})| \le T_i^{\delta_k}.
\end{equation}
Then $H(\gamma_{-k_0/-k-2}) \ll T_i^{\delta_k}$ for each $\gamma_{-k_0/-k-2} \in \Xi_{k+1}(T_i) \cap \Gamma_{-k_0/-k-2}$.
\end{lemma}
\begin{proof}
Use the notation from \eqref{EqBigDiagramLiftingChase}.

Take $\gamma_{-k_0/-k-2} \in \Xi_{k+1}(T_i) \cap \Gamma_{-k_0/-k-2}$. Write $(\gamma_{-k_0/-k-1},\gamma_{-k-1/-k-2})$ for the image of $\gamma_{-k_0/-k-2}$ under the isomorphism \eqref{Eq} $P/W_{-k-2} \cong P/W_{-k-1}\times W_{-k-1/-k-2}$. Then $\gamma_{-k_0/-k-1} \in p_{k+1,k}(\Xi_{k+1}(T_i) )$ for the group morphism $p_{k+1,k} \colon P/W_{-k-2} \rightarrow P/W_{-k-1}$. Thus $H(\gamma_{-k_0/-k-1}) \le \alpha'_k T_i^{\alpha_k}$ by \eqref{LemmaProjectionXiLevel} and \Cref{LemmaHtNormLatticePts}.

There exists a point $\tilde{z}_{k+1} \in \gamma_{-k_0/-k-2} \mathfrak{F}_{k+1,\R} \cap \tilde{Z}(k+1,T_i)_{\R}$. 
Write $\tilde{z}_{k+1,k} = \lambda_{k+1}(\tilde{z}_{k+1})$.

Now that $\tilde{z}_{k+1,k}$ is a point in $\lambda_{k+1}(\gamma_{-k_0/-k-2}\mathfrak{F}_{k+1,\R})$, which by  \Cref{LemmaNormPolynomialGrowth} is contained in a $|\cdot|$-ball of radius $\le \beta_k H(\gamma_{-k_0/-k-1})^{k+1} \le \beta_k \alpha_k^{\prime k} T_i^{(k+1)\alpha_k}$ centered at $\gamma_{-k-1/-k-2}$. On the other hand $\tilde{z}_{k+1,k}$ is a point in $\lambda_{k+1}(\tilde{Z}(k+1,T_i)_{\R})$. So the assumption \eqref{EqHypHorizontalGrowth} yields $|\tilde{z}_{k+1,k}| \le T_i^{\delta_k}$. Write $\beta'_k := \beta_k \alpha_k^{\prime k}$. This says that in the Euclidean space, the following two $|\cdot|$-balls intersect: the one of radius $T_i^{\delta_k}$ centered at $0$, and the one of radius $\beta'_k T_i^{(k+1)\alpha_k}$ centered at  $\gamma_{-k-1/-k-2}$. Thus we must have $H(\gamma_{-k-1/-k-2}) \le T_i^{\delta_k} + \beta'_k T_i^{(k+1)\alpha_k}$, which furthermore $\le (1+\beta'_k) T_i^{\delta_k}$ by assumption on $\delta_k$.

Therefore $H(\gamma_{-k_0/-k-2}) \le \max\{H(\gamma_{-k_0/-k-1}), H(\gamma_{-k-1/-k-2})\} \ll T_i^{\delta_k}$. 
\end{proof}

With these preparations, we are ready to prove \Cref{PropCounting}.
\begin{proof}[Proof of \Cref{PropCounting}]
We prove \Cref{PropCounting} by induction on $k \ge k_0$.

The base step is $k = k_0$. Recall that in this case, $\Xi_{k_0}(T)$ is defined in a different way than the other $\Xi_k(T)$'s (with $k \ge k_0+1$). Indeed
\begin{equation}\label{EqSetXik0}
\Xi_{k_0}(T) :=\{g \in (W_{-k_0}/W_{-k_0-1})(\R) : g \mathfrak{F}_{k_0,\R}\cap \tilde{Z}_{k_0}(T)_{\R}\not=\emptyset\}.
\end{equation}
The conclusion for the base step follows immediately from \Cref{LemmaCountingBaseStep}, which says: There exist constants $c_{k_0}, \epsilon_{k_0} > 0$ such that
\begin{equation}\label{EqCountingPropAddedInitialStep}
\#\{\gamma_{-k_0/-k_0-1} \in \Xi_{k_0}(T) \cap \Gamma_{-k_0/-k_0-1} : H(\gamma_{-k_0/-k_0-1}) \le T \} \ge c_{k_0} T^{\epsilon_{k_0}}.
\end{equation}

\vskip 0.5em

Assume \Cref{PropCounting} is proved for $k \ge k_0$, \textit{i.e.}  there exist constants $c_k, \epsilon_k> 0$, and a sequence $\{T_i \in \R\}_{i \in \N}$ with $T_i \rightarrow \infty$, such that
\begin{equation}\label{EqCountingPropInductionBaseStepHorizontal}
\#\{\gamma_{-k_0/-k-1} \in \Xi_{k}(T_i) \cap \Gamma_{-k_0/-k-1} : H(\gamma_{-k_0/-k-1}) \le T_i \} \ge c_k T_i^{\epsilon_k}.
\end{equation}
We wish to prove for $k+1$, \textit{i.e.} 
 up to replacing $\{T_i \in \R\}$ by a subsequence, find constants $c_{k+1}, \epsilon_{k+1} > 0$ such that 
\begin{equation}\label{EqCountingPropInductionInductionStep}
\#\{\gamma_{-k_0/-k-2} \in \Xi_{k+1}(T_i) \cap \Gamma_{-k_0/-k-2} : H(\gamma_{-k_0/-k-2}) \le T_i \} \ge c_{k+1} T_i^{\epsilon_{k+1}}.
\end{equation}

We have two alternatives:
\begin{enumerate}
\item[(i)] Either $|\lambda_{k+1}(\tilde{Z}(k+1,T)_{\R})| > T^{\delta_{k}}$ for all $T \gg 1$,
\item[(ii)] or $|\lambda_{k+1}(\tilde{Z}(k+1,T_i)_{\R})| \le T_i^{\delta_{k}}$ for some sequence $\{T_i\in \R\}$ with $T_i \rightarrow \infty$.
\end{enumerate}
For case (i), the conclusion follows from \Cref{LemmaLiftingCaseVertical}. 
So from now on, assume that we are in case (ii).

By \Cref{LemmaHorizontalGrowth} we have $H(\gamma_{-k_0/-k-2}) \ll T_i^{\delta_{k}}$ for each $\gamma_{-k_0/-k-2} \in \Xi_{k+1}(T_i) \cap \Gamma_{-k_0/-k-2}$. Thus, to establish \eqref{EqCountingPropInductionInductionStep}, it suffices to show that $\#(\Xi_{k+1}(T_i) \cap \Gamma_{-k_0/-k-2})$ is at least polynomial in $T_i$.

 Notice that
\[
\tilde{Z}(k+1,T_i)_{\R} = \bigcup_{\gamma_{-k_0/-k-2}\in \Xi_{k+1}(T_i) \cap \Gamma_{-k_0/-k-2}} \gamma_{-k_0/-k-2}\mathfrak{F}_{k+1,\R} \cap \tilde{Z}(k+1,T_i)_{\R}.
\]
Applying $p_{k+1,k}|_{\cD_{k+1,\R}}$ to both sides, we obtain
\[
p_{k+1,k}(\tilde{Z}(k+1,T_i)_{\R}) \subseteq \bigcup_{\gamma_{-k_0/-k-2}\in \Xi_{k+1}(T_i) \cap \Gamma_{-k_0/-k-2}} \gamma_{-k_0/-k-1}\mathfrak{F}_{k,\R}.
\]
Thus $p_{k+1,k}(\tilde{Z}(k+1,T_i)_{\R})$ hits $\le \#(\Xi_{k+1}(T_i) \cap \Gamma_{-k_0/-k-2})$ fundamental sets which are $\Gamma_{-k_0/-k-1}$-translates of $\mathfrak{F}_{k,\R}$.

On the other hand, we claim that $\tilde{Z}_k(T_i)_{\R}$ hits at least polynomially fundamental sets which are $\Gamma_{-k_0/-k-1}$-translates of $\mathfrak{F}_{k,\R}$. To show this, we need to divide into two cases: when $k = k_0$ or $k \ge k_0+1$. When $k = k_0$, this follows immediately from \eqref{EqCountingPropAddedInitialStep} and the definition of $\Xi_{k_0}(T_i)$ \eqref{EqSetXik0}. It remains to prove the claim for $k \ge k_0+1$. In this case, if $\gamma_{-k_0/-k-1} \in \Xi_{k}(T_i) \cap \Gamma_{-k_0/-k-1}$ with $H(\gamma_{-k_0/-k-1}) \le T_i$, then $\lambda_k(\gamma_{-k_0/-k-1}\mathfrak{F}_{k,\R})$ is contained in a $|\cdot|$-ball of radius $\ll T_i^k$ centered at $0$ by \Cref{LemmaNormPolynomialGrowth} (applied with $k$ replaced by $k-1$). Thus $\gamma_{-k_0/-k-1}\mathfrak{F}_{k,\R} \cap \tilde{Z}_k(T_i^k)_{\R} \not= \emptyset$. So \eqref{EqCountingPropInductionBaseStepHorizontal} yields the claim with $T_i$ replaced by $T_i^k$.
 

 By choice of $\tilde{Z}(k+1,T_i)$ and $\tilde{Z}_k(T_i)$, we have that  $p_{k+1,k}(\tilde{Z}(k+1,T)) = \tilde{Z}_k(T)$; see the end of \Cref{SubsectionPrepThmCounting} (or right above \Cref{SubsectionSketchStrategyThmCounting}). Thus $p_{k+1,k}(\tilde{Z}(k+1,T)_{\R}) = \tilde{Z}_k(T)_{\R}$. 
 Hence the previous two paragraphs imply that $\#(\Xi_{k+1}(T_i) \cap \Gamma_{-k_0/-k-2})$ is at least polynomial in $T_i$. Hence we are done.
\end{proof}

\section{Normality of the $\Q$-stabilizer: Part 1} \label{normality1}
Let $N$ be the connected algebraic monodromy group of the admissible VMHS $(\V_\Z,W_\bullet,\cF^\bullet)$ on $S$. Then $N \lhd P$ by \Cref{ThmAndreNormal}.

The goal of this section is to prove the following normality result.
\begin{prop} \label{normal1}
$H_{\mathscr{Z}^{\mathrm{Zar}}} \lhd N$.
\end{prop}

\subsection{Family associated with $\mathscr{Z}$}

Let $\mathbf{H}$ be the component of the Hilbert scheme of $S \times \cD^\vee$ which contains $[\mathscr{Z}^{\mathrm{Zar}}]$, the point representing $\mathscr{Z}^{\mathrm{Zar}}$. Then $\mathbf{H}$ is proper. Consider the (modified) universal family
\[
\mathbscr{B} = \{(x,\tilde{m} , [\mathscr{B}]) \in (S \times \cD) \times \mathbf{H} : (x,\tilde{m}) \in \mathscr{B}\}.
\]
The projection
\begin{equation}\label{EqProjPerpToHilb}
\psi \colon \mathbscr{B} \rightarrow S \times \cD
\end{equation}
is a proper map since $\mathbf{H}$ is proper.

Define
\[
\mathbscr{Z} = \{(\tilde{\delta}, [\mathscr{B}]) \in (\Delta \times \mathbf{H}) \cap \mathbscr{B} : \dim_{\tilde{\delta}}(\Delta \cap \mathscr{B}) \ge \dim \mathscr{Z} \}.
\]
Then $\mathbscr{Z}$ is a closed complex analytic subset of $\mathbscr{B}$. So $\psi(\mathbscr{Z})$ is closed complex analytic in $S \times \cD$ as $\psi$ is proper. Note that $\psi(\mathbscr{Z}) \subseteq \Delta$.

Let us summarize the notations in the following diagram.
\[
\xymatrix{
\mathbscr{B} \ar[d]_-{\psi} & \mathbscr{Z} \ar@{}[l]|{\supseteq} \ar[d]^-{\psi|_{\mathbscr{Z}}} & \\
S \times \cD  & \Delta \ar@{}[l]|-{\supseteq} \ar[r] \ar[d]_-{u_S} & \cD \ar[d]^-{u} \\
& S \ar[r]^-{[\Phi]} & \Gamma_P \backslash \cD
}
\]

Recall that the arithmetic group $\Gamma_P$ acts on $S \times \cD$ by its action on the second factor. We claim that $\Gamma_P \psi(\mathbscr{Z}) = \psi(\mathbscr{Z})$. Indeed, this action of $\Gamma_P$ on $S \times \cD$ induces an action of $\Gamma_P$ on $\mathbscr{B}$ by
\begin{equation}\label{EqActionOfGammaOnB}
\gamma(x,\tilde{m},[\mathscr{B}]) = (x,\gamma\tilde{m},[\gamma\mathscr{B}]).
\end{equation}
Thus $\Gamma_P \Delta = \Delta$ implies $\Gamma_P \mathbscr{Z} = \mathbscr{Z}$. But $\psi$ is $\Gamma_P$-invariant. So $\Gamma_P \psi(\mathbscr{Z}) = \psi(\mathbscr{Z})$.

As the map $u_S \colon \Delta  \rightarrow S$ is $\Gamma_P$-invariant (for the trivial action of $\Gamma_P$ on $S$), we have that $T: = u_S(\psi(\mathbscr{Z}))$ is closed complex analytic in $S$.


\begin{prop}\label{PropFoliAlg}
$T$ is an algebraic subvariety of $S$.
\end{prop}
\begin{proof} By definable Chow (\!\!\cite[Thm.~4.5]{PeterzilComplex-analyti} or \cite[Thm.~2.2]{MokAx-Schanuel-for}), it suffices to prove that $T$ is definable in $\R_{\mathrm{an},\exp}$. In the rest of the proof, when we say ``definable'' we mean \textit{definable in} $\R_{\mathrm{an},\exp}$.

Let $\mathfrak{F}_{\R}$ and $\mathfrak{F} = r^{-1}(\mathfrak{F}_{\R})$ be as in \Cref{ThmFundSet}.

Note that $u_S$ is the restriction of the natural projection $p_S \colon S \times \mathcal{D} \rightarrow S$ to ${\Delta}$. So $T = u_S(\psi(\mathbscr{Z})) = p_S(\psi(\mathbscr{Z})) = p_S(\psi(\mathbscr{Z}) \cap (S \times\mathfrak{F}))$. Thus it suffices to prove that $\psi(\mathbscr{Z}) \cap (S \times\mathfrak{F})$ is definable.

But $\psi(\mathbscr{Z}) \cap (S \times\mathfrak{F}) = \psi(\mathbscr{Z} \cap (S \times \mathfrak{F} \times \mathbf{H}))$. So it suffices to prove that $\mathbscr{Z} \cap (S \times \mathfrak{F} \times \mathbf{H})$ is definable.

By property (ii) of \Cref{ThmFundSet}, the period map $[\Phi]$ is definable if we endow $\Gamma_P \backslash\cD$ with the definable structure given by $u|_{\mathfrak{F}}$. So
\[
\Delta \cap (S \times \mathfrak{F}) = \{(x,\tilde{m}) \in S \times \mathfrak{F} : u(\tilde{m}) = [\Phi](x)\}
\]
is a definable subset of $S \times \cD$. So
\[
\Big( \big(\Delta \cap (S \times \mathfrak{F}) \big)  \times \mathbf{H} \Big) \cap \mathbscr{B}
\]
is a definable subset of $S \times \cD \times \mathbf{H}$. So
\[
\mathbscr{Z} \cap (S \times \mathfrak{F} \times \mathbf{H}) = \{(\tilde{\delta},[\mathscr{B}]) \in \Big( \big(\Delta \cap (S \times \mathfrak{F}) \big)  \times \mathbf{H} \Big) \cap \mathbscr{B} : \dim_{\tilde{\delta}}(\Delta\cap(S\times\mathfrak{F})\cap\mathscr{B}) \ge \dim \mathscr{Z}\}
\]
is definable. Hence we are done.
\end{proof}

\subsection{Monodromy} 
\begin{proof}[Proof of \Cref{normal1}]
Recall that $\Gamma_P \mathbscr{Z} = \mathbscr{Z}$. So $\Gamma_P \backslash \mathbscr{Z}$ is a complex analytic space. The proper map $\psi$ \eqref{EqProjPerpToHilb} induces
\[
\bar{\psi} \colon \Gamma_P \backslash \mathbscr{Z} \rightarrow \Gamma_P \backslash \psi(\mathbscr{Z}) = u_S(\psi(\mathbscr{Z})) = T,
\]
which is surjective and  proper.

We have  $\mathscr{Z} = \psi(\mathscr{Z} \times [\mathscr{Z}^{\mathrm{Zar}}]) \subseteq \psi(\mathbscr{Z})$. Applying $u_S$ to both sides, we get
\[
u_S(\mathscr{Z}) \subseteq T.
\]
Recall the assumption $S = u_S(\mathscr{Z})^{\mathrm{Zar}}$. So taking the Zariski closures of both sides, we get $T = S$ by \Cref{PropFoliAlg}.

Let $\mathbscr{Z}_0$ be an irreducible component of $\mathbscr{Z}$ which contains $\mathscr{Z} \times [\mathscr{Z}^{\mathrm{Zar}}]$. By abuse of notation, we use $\Gamma_P\backslash\mathbscr{Z}_0$ to denote the image of $\mathbscr{Z}_0$ under the map $\mathbscr{Z} \rightarrow \Gamma_P \backslash \mathbscr{Z}$. Then  $\bar{\psi}(\Gamma_P\backslash\mathbscr{Z}_0) = S$ because $T=S$ is irreducible.

Thus $\bar{\psi}$ induces a map $\bar{\psi}_* \colon \pi_1(\Gamma_P \backslash \mathbscr{Z}_0) \rightarrow \pi_1(S)$, and so a subgroup $\Gamma_0$ of $N(\Q)$. We have then $\Gamma_0 \mathbscr{Z}_0 = \mathbscr{Z}_0$. But $\mathrm{Im}(\bar{\psi}_*)$ has finite index in $\pi_1(S)$ (since $\bar{\psi}$ is proper), so $\Gamma_0^{\mathrm{Zar}} = N$.

Next denote by $\theta \colon \mathbscr{B} \subseteq (S \times \cD) \times \mathbf{H} \rightarrow \mathbf{H}$ the restriction of the natural projection. Let $\mathbscr{F} = \theta^{-1}(\theta(\mathbscr{Z}_0)) = \{(x,\tilde{m},[\mathscr{B}]) : [\mathscr{B}] \in \theta(\mathbscr{Z}_0), ~(x,\tilde{m}) \in \mathscr{B}\}$. Then $\mathbscr{F} \subseteq \mathbscr{B}$ is the family of algebraic varieties parametrized by $\theta(\mathbscr{Z}_0) \subseteq \mathbf{H}$, with the fiber over each $[\mathscr{B}] \in \theta(\mathbscr{Z}_0)$ being $\mathscr{B}$. Then we have
\[
\Gamma_0 \mathbscr{F} \subseteq \mathbscr{F}
\]
for the action of $\Gamma_P$ on $\mathbscr{B}$ defined by \eqref{EqActionOfGammaOnB}. Thus every $\gamma \in \Gamma_0$ sends a very general fiber of $\mathbscr{F}$ to a very general fiber of $\mathbscr{F}$.

Define
\[
\Gamma_{\mathbscr{F}} = \{ \gamma \in \Gamma_P  : \gamma \mathscr{B} \subseteq \mathscr{B}, ~ \text{for all }[\mathscr{B}] \in \theta(\mathbscr{Z}_0) \}.
\]
Then for a very general $[\mathscr{B}] \in \theta(\mathbscr{Z}_0)$, we have
\begin{equation}\label{EqZStabB}
\mathrm{Stab}_{\Gamma_P}(\mathscr{B}) = \Gamma_{\mathbscr{F}}.
\end{equation}
By construction of $\mathbscr{F}$, without loss of generality we may assume that $\mathscr{Z}^{\mathrm{Zar}}$ is a very general fiber of $\mathbscr{F}$. The conclusion of the last paragraph implies that any $\gamma \in \Gamma_0$ sends $\mathscr{Z}^{\mathrm{Zar}}$ to a very general fiber of $\mathbscr{F}$. By taking the stabilizers of the two fibers in consideration, we get $\Gamma_{\mathbscr{F}} = \gamma \Gamma_{\mathbscr{F}} \gamma^{-1}$ for all $\gamma \in \Gamma_0$. By taking the Zariski closures, we get
\[
(\Gamma_{\mathbscr{F}}^{\mathrm{Zar}})^{\circ} \lhd N.
\]

On the other hand \eqref{EqZStabB} implies $(\Gamma_{\mathbscr{F}}^{\mathrm{Zar}})^{\circ} = H_{\mathscr{Z}^{\mathrm{Zar}}}$. Hence we are done.
\end{proof}

\section{Normality of the $\Q$-stabilizer: Part 2} \label{normality2}
In this section, we finish the proof of the following proposition.
\begin{prop}\label{PropFinalNormality}
$H_{\mathscr{Z}^\Zar} \lhd P$.
\end{prop}

For simplicity we write $H$ for $H_{\mathscr{Z}^\Zar}$. By \Cref{ThmAndreNormal} and \Cref{normal1}, we have $H \lhd N \lhd P$. 

Recall $W_{-1} = \cR_u(P)$. Fix a suitable Levi decomposition $P = W_{-1} \rtimes G$. 
 In order to prove $H\lhd P$, it suffices to establish the following two assertions.
\begin{enumerate}
\item[(i)] $W_{-1}\cap H$ is a normal subgroup of $P$;
\item[(ii)] $H/(W_{-1}\cap H)$, as a subgroup of $G$, acts trivially on $W_{-1}/ (W_{-1}\cap H)$.
\end{enumerate}
Assume that (i) holds true. Then
 $N/(W_{-1}\cap N)$ acts trivially on $W_{-1}/(W_{-1}\cap N)$ since $N\lhd P$, and $H/(W_{-1}\cap H)$ acts trivially on $(W_{-1}\cap N)/(W_{-1}\cap H)$ since $H \lhd N$. But $H/(W_{-1}\cap H)$ is a subgroup of  $N/(W_{-1}\cap N)$. So part (ii) is established.

Now it remains to prove part (i). We will finish this in the rest of this section.

Recall the set $\Theta$ defined in \eqref{EqDefnSetTheta}
\[
\Theta = \{g \in P(\R) : \dim (g^{-1} \mathscr{Z}^{\mathrm{Zar}} \cap (S \times\mathfrak{F}) \cap \Delta ) = \dim \mathscr{Z} \},
\]
where $\mathfrak{F}$ is defined as follows; see \Cref{ThmFundSet}.
\begin{equation}\label{EqDiagramNormal2Prel}
\xymatrix{
\cD \ar[r]^-{r} & \cD_{\R} \cong  \cD_0 \times \prod_{1\le k \le m} (W_{-k}/W_{-k-1})(\R)  \\
\mathfrak{F}= r^{-1}(\mathfrak{F}_{\R}) \ar@{}[u]|-*[@]{\subseteq} \ar@{|->}[r]  & \mathfrak{F}_{\R} := \mathfrak{F}_0 \times \prod_{1 \le k \le m} (-M,M)^{\dim (W_{-k}/W_{-(k+1)})(\R)}  \ar@{}[u]|-*[@]{\subseteq} .
}
\end{equation}

Recall $p_0 \colon (P,\cD) \rightarrow (G,\cD_0)$. Denote by $\Gamma_G := p_0(\Gamma_P)$. 

Let $\gamma_0 \in p_0(\Theta) \cap \Gamma_G \supseteq \{\gamma_0 \in \Gamma_G : \gamma_0 \mathfrak{F}_0 \cap \tilde{Z}_0 \not= \emptyset\}$, and let $\tilde{Z}|_{\gamma_0\mathfrak{F}_0}^+$ be a complex analytic irreducible component of $\tilde{Z}|_{\gamma_0\mathfrak{F}_0} := \tilde{Z} \cap p_0^{-1}(\gamma_0\mathfrak{F}_0)$.

Write $\Gamma_{-1} := W_{-1}(\Q) \cap \Gamma_P$. The quotient $\cD \rightarrow \cD_0$ induces $\Gamma_{-1} \backslash \cD \rightarrow \cD_0$. Write $(\Gamma_{-1}\backslash\cD)|_{\gamma_0\mathfrak{F}_0} \subseteq \Gamma_{-1} \backslash \cD$ for the inverse image of $\gamma_0\mathfrak{F}_0$.

\begin{lemma}\label{LemmaClosedInUsualTopologyUnipotent}
For the map $\bar{u} \colon \cD \rightarrow \Gamma_{-1} \backslash \cD$, the set $\bar{u}(\tilde{Z}|_{\gamma_0\mathfrak{F}_0}^+)$ is closed in $(\Gamma_{-1}\backslash\cD)|_{\gamma_0\mathfrak{F}_0}$ in the usual topology.
\end{lemma}
\begin{proof}
Fix a Levi decomposition $P = W_{-1} \rtimes G$. 

Define the following set
\[
\Theta' := \{ w \in W_{-1}(\R) : \left( (w^{-1}, \gamma_0^{-1})\mathscr{Z}^{\Zar} \cap (S \times \mathfrak{F}) \cap \Delta \right) = \dim \mathscr{Z} \} \subseteq W_{-1}(\R).
\]
Then $\Theta' \cap \Gamma_P \supseteq \{ \gamma_{-1} \in \Gamma_{-1} : (\gamma_{-1},\gamma_0)\mathfrak{F} \cap \tilde{Z}|_{\gamma_0\mathfrak{F}_0}^+ \not= \emptyset\}$, and $(\Theta' , \gamma_0^{-1}) \subseteq \Theta$.  Denote by $\Gamma_{-1,H} := H(\Q) \cap \Gamma_{-1}$.

Denote by
\[
\lambda \colon \cD_{\R} \cong  \cD_0 \times \prod_{1\le k \le m} (W_{-k}/W_{-k-1})(\R) \rightarrow   \prod_{1\le k \le m} (W_{-k}/W_{-k-1})(\R)
\]
the natural projection. For each $\gamma_{-1} \in \Gamma_{-1}$, write $\gamma_{-k}$ the image of $\gamma_{-1}$ under the quotient $W_{-1} \rightarrow W_{-k}$ (with $k \le m$). We say that $\gamma_{-1}$ has reduced multi-height $\le (T_1, \ldots, T_{m-1})$ if $H(\gamma_{-k}) \le T_k$ for each $k \in \{1,\ldots,m-1\}$. Now if $\gamma_{-1} \in \Gamma_{-1}$ has reduced multi-height at most $(T^{1/\prod_{i=2}^m i}, \ldots ,T^{1/(m-1)m}, T^{1/m})$, then $\lambda((\gamma_{-1},\gamma_0)\mathfrak{F})$ is contained in a $|\cdot|$-ball of radius $\ll T$ by applying \Cref{LemmaNormPolynomialGrowth} iteratively (the constant involved may depend on $\gamma_0$).

Suppose that $\{\gamma_{-1} \in \Gamma_{-1} : (\gamma_{-1},\gamma_0)\mathfrak{F} \cap \tilde{Z}|^+_{\gamma_0\mathfrak{F}_0} \not= \emptyset\}$ is not contained in a finite union of $\Gamma_{-1,H}$-cosets.  Because $\tilde{Z}|^+_{\gamma_0\mathfrak{F}_0}$ is connected and $\prod_{1\le k \le m} (W_{-k}/W_{-k-1})(\R)$ is Euclidean, we get for free (when $T \gg 1$) that it contains $\ge T$ elements $\gamma_{-1} \in \Gamma_{-1}\setminus \Gamma_{-1,H}$ with the following property: the reduced multi-height of $\gamma_{-1}$ is $\le (T^{1/\prod_{i=2}^m i}, \ldots ,T^{1/(m-1)m}, T^{1/m})$ and $H(\gamma_{-m}) \ll T$. Notice that each such $\gamma_{-1}$ has height $\ll T$. Hence by Pila--Wilkie, there exist constants $c,\epsilon > 0$ with the following property: for each $T \gg 1$, $\Theta'$ contains a semi-algebraic block $B'$ which is not in any coset of $(W_{-1}\cap H)(\R)$ and which contains $\ge c T^{\epsilon}$ elements in $\Gamma_{-1}$ outside $\Gamma_{-1,H}$ of height at most $T$. Recall our assumption that every positive diemensional semi-algebraic block in $\Theta$ is contained in a left coset of $\mathrm{Stab}_{P(\R)}(\mathscr{Z}^{\Zar})$. In particular, $(B', \gamma_0^{-1}) \subseteq (\gamma_{-1}, \gamma_0^{-1}) \cdot \mathrm{Stab}_{P(\R)}(\mathscr{Z}^{\Zar})$ for some $\gamma_{-1} \in B' \cap \Gamma_{-1}$. Hence $(\gamma_{-1}^{-1} \cdot (\gamma_0 B') , 1) \subseteq \mathrm{Stab}_{P(\R)}(\mathscr{Z}^{\Zar})$. So 
\[
(\gamma_{-1}^{-1} \cdot (\gamma_0 B') \cap \Gamma_{-1} , 1) \subseteq \mathrm{Stab}_{P(\R)}(\mathscr{Z}^{\Zar}) \cap \Gamma \subseteq H(\Q).
\]
By letting $T \rightarrow \infty$ and varying $B'$ accordingly, we see that this inclusion cannot hold true because $B' \subseteq W_{-1}(\R)$ is not contained in any coset of $(W_{-1}\cap H)(\R)$.

Therefore $\{\gamma_{-1} \in \Gamma_{-1} : (\gamma_{-1},\gamma_0)\mathfrak{F} \cap \tilde{Z}|^+_{\gamma_0\mathfrak{F}_0} \not= \emptyset\}$ is contained in a finite union of $\Gamma_{-1,H}$-cosets. Hence we are done.
\end{proof}

From now on we work with the usual topology. 
Now $\tilde{Z}|^+_{\gamma_0\mathfrak{F}_0}$ is an open subset of $\tilde{Z}$ because $\gamma_0\mathfrak{F}_0$ is open in $\cD_0$. In particular, there exists $\gamma_0 \in \Gamma_G$ such that $\dim \tilde{Z}|^+_{\gamma_0\mathfrak{F}_0} = \dim \tilde{Z}$. Let $Y$ be the closure of $\bar{u}( \tilde{Z}|^+_{\gamma_0\mathfrak{F}_0} )$ in $\Gamma_{-1}\backslash \cD$. Then \Cref{LemmaClosedInUsualTopologyUnipotent} yields $Y^{\circ} \subseteq \bar{u}( \tilde{Z}|^+_{\gamma_0\mathfrak{F}_0} ) \subseteq Y$, where $Y^{\circ}$ is the interior of $Y$. Let $\tilde{Z}'$ be a complex analytic irreducible component of $\bar{u}^{-1}(Y)$ which contains $\tilde{Z}|^+_{\gamma_0\mathfrak{F}_0}$. Then $\tilde{Z}' \subseteq \tilde{Z}$ and $\dim \tilde{Z}|^+_{\gamma_0\mathfrak{F}_0}  \le \dim \tilde{Z}' \le \dim \tilde{Z}$. Therefore $\dim \tilde{Z}' = \dim \tilde{Z}$.

By analytic continuation, we then have $(\tilde{Z}')^{\Zar} = \tilde{Z}^{\Zar}$. Moreover, an irreducible component of $(S \times \tilde{Z}' ) \cap \Delta$, which we denote by $\mathscr{Z}'$, satisfies $(\mathscr{Z}')^{\Zar} = \mathscr{Z}^{\Zar}$.

Set
\[
\Gamma' := \mathrm{Im} \left( \pi_1(\bar{u}(\tilde{Z}')) \rightarrow \pi_1(\Gamma_{-1}\backslash\cD) = \Gamma_{-1} \right) \subseteq \Gamma_{-1}.
\]
Then $\Gamma'$ stabilizes $\tilde{Z}'$. So it also stabilizes $\mathscr{Z}'$, and hence $(\Gamma')^{\Zar}(\R)$ stabilizes $(\mathscr{Z}')^{\Zar} = \mathscr{Z}^{\Zar}$. Thus $(\Gamma')^{\Zar} \subseteq H$ because $H$ is the $\Q$-stabilizer of $\mathscr{Z}^{\Zar}$.

By logarithmic Ax \Cref{ThmLogAx} and its remark, $(\tilde{Z}')^{\Zar} = \tilde{Z}^{\Zar}$ is contained in an $N(\R)^+(W_{-1}\cap N)(\C)$-orbit in $\cD$. Call this orbit $\cD_N$. For each $k \ge 1$, let $W_{-k,N} := W_{-k}\cap N$. Then the top row of \eqref{EqDiagramNormal2Prel} induces $\cD_N \xrightarrow{r} \cD_{N,\R} \cong \cD_{N,0} \times \prod_{i=1}^m (W_{-k,N} / W_{-k-1,N})(\R)$.

Denote by $\Gamma_{-k,N} := \Gamma_P \cap W_{-k,N}$, and by $\Gamma_{-k/-k-1,N} := \Gamma_{-k,N}/\Gamma_{-k-1,N}$. Then we have the following diagram
\[
\xymatrix{
\cD_N \ar[r]^-{r} \ar[d]_-{\bar{u}} & \cD_{N,\R} \cong \cD_{N,0} \times \prod_{i=1}^m (W_{-k,N} / W_{-k-1,N})(\R)  \ar[d]^-{\bar{u}_{\R}}  \\
\Gamma_{-1,N} \backslash \cD_N \ar[r] & \Gamma_{-1,N} \backslash \cD_{N,\R} \cong  \cD_{N,0} \times \prod_{i=1}^m \Gamma_{-k/-k-1,N}\backslash (W_{-k,N} / W_{-k-1,N})(\R) 
}
\] 
and $\tilde{Z}' \subseteq \cD_N$. 

Since $W_{-1}\cap H \lhd N$ (because $W_{-1}\cap H = \cR_u(H)$), we can take the quotient of $\cD_{N,\R}$ by $(W_{-1}\cap H)(\R)$ and get a real manifold. Call this quotient $q$. Denote by $W_{-k, N/H} := W_{-k,N} / (W_{-k} \cap H)$, and by $\Gamma_{-k/-k-1, N/H}$ the image of $\Gamma_{-k,N}$ under the quotient $W_{-k, N} \rightarrow W_{-k, N/H} \rightarrow W_{-k, N/H}/W_{-k-1, N/H}$. Then the diagram above expands to (notice $\tilde{Z}' \subseteq \cD_N$)
\scriptsize
\[
\xymatrix{
\cD_N \ar[r]^-{\bar{u}} \ar[d]_-{r} & \Gamma_{-1,N} \backslash \cD_N \ar[d] \\
\cD_{N,\R} \cong \cD_{N,0} \times \prod_{i=1}^m (W_{-k,N} / W_{-k-1,N})(\R) \ar[d]_-{q} \ar[r]^-{\bar{u}_{\R}} & \Gamma_{-1,N} \backslash \cD_{N,\R} \cong  \cD_{N,0} \times \prod_{i=1}^m \Gamma_{-k/-k-1,N}\backslash (W_{-k,N} / W_{-k-1,N})(\R) \ar[d]^-{[q]} \\
 q(\cD_N) \cong \cD_{N,0} \times \prod_{i=1}^m (W_{-k,N/H} / W_{-k-1,N/H})(\R) \ar[r] \ar[d]_{\lambda_{N/H}} & \cD_{N,0} \times \prod_{i=1}^m  \Gamma_{-k/-k-1, N/H} \backslash(W_{-k,N/H} / W_{-k-1,N/H})(\R) \ar[d]^-{[\lambda_{N/H}]} \\
 \prod_{i=1}^m (W_{-k,N/H} / W_{-k-1,N/H})(\R) \ar[r]^-{u_{W,\R}} & \prod_{i=1}^m  \Gamma_{-k/-k-1, N/H} \backslash(W_{-k,N/H} / W_{-k-1,N/H})(\R).
}
\] 
\normalsize
Since $\Gamma' \subseteq H$, we have that $[\lambda_{N/H}] \circ [q] \circ \bar{u}_{\R}(r(\tilde{Z}'))$ is simply connected. But $\lambda_{N/H} \circ q \circ r(\tilde{Z}')$ is connected. So it is contained in a fundamental domain of $u_{W,\R}$, and hence is bounded. 
 So $\lambda_{N/H} \circ q \circ r(\tilde{Z}')$ is a point because $\tilde{Z}'$ is complex analytic.

Recall that $\mathscr{Z}' \subseteq \Delta$ and $\tilde{Z}'$ is the projection of $\mathscr{Z}'$ to $\cD$. Thus the previous paragraph implies that for $\tilde{z}_0 \in p_0(\tilde{Z}') \subseteq \cD_0$, the fiber of $(\tilde{Z}'
)^{\mathrm{Zar}}$ over $\tilde{z}_0$ is contained in an $(W_{-1}\cap H)(\C)$-orbit. Since $H(\R)(W_{-1}\cap H)(\C)$ stabilizes $\tilde{Z}^{\mathrm{Zar}} = (\tilde{Z}'
)^{\mathrm{Zar}}$, this fiber is indeed an $(W_{-1}\cap H)(\C)$-orbit. Call this fiber $\tilde{X}'$. We may furthermore assume that $\tilde{z}_0$ is Hodge generic in $\cD_0$.

As $W_{-1}\cap H$ is a $\QQ$-group, the set $u(\tilde{X}')$ is closed in $\Gamma\backslash \cD$ under the usual topology. It is a definable subset, and hence $[\Phi]^{-1}(u(\tilde{X}'))$ is a definable complex analytic subvariety of $S$; see \Cref{ThmFundSet}. So $[\Phi]^{-1}(u(\tilde{X}'))$ is algebraic by definable Chow. Its connected algebraic monodromy group is $W_{-1} \cap H$. Hence $W_{-1}\cap H$ is normal in $P$ by Andr\'{e} \Cref{ThmAndreNormal}.

\section{End of the proof}
In this section, we prove \Cref{ThmASDevissage}, which finishes the proof of \Cref{ThmAS}.

\medskip
Let $\mathscr{Z}$ as in \Cref{ThmASDevissage}. If $\dim
H_{\mathscr{Z}^{\mathrm{Zar}}} =0$ then we are done by
\Cref{PropBignessStab}. Thus we may assume $\dim
H_{\mathscr{Z}^{\mathrm{Zar}}} > 0$. For simplicity we write $H:=
H_{\mathscr{Z}^{\mathrm{Zar}}}$. 

\Cref{PropFinalNormality} says that $H \lhd P$. Thus we can take the quotient $\cD/H$ and obtain
\begin{equation}\label{EqEqQuotPerMapDiagEnd}
\xymatrix{
& \cD \ar[r]^-{p_H} \ar[d]_-{u} & \cD/H \ar[d]^-{u_{/H}} \\
S \ar[r]^-{[\Phi]} \ar@/_0.6cm/[rr]|-{[\Phi_{/H}]} & \Gamma_P\backslash\cD \ar[r]^-{[p_H]} & \Gamma_{P/H}\backslash(\cD/H) 
}.
\end{equation}
We can apply \Cref{PropBignessStab} to the new period map $[\Phi_{/H}] \colon S \rightarrow \Gamma_{P/H}\backslash(\cD/H)$ and
\[
\mathscr{Z}_{/H} := (\mathrm{id}_S, p_H)(\mathscr{Z}) \subseteq S \times_{\Gamma_{P/H}\backslash(\cD/H)} (\cD/H).
\]
But $H = H_{\mathscr{Z}^{\mathrm{Zar}}}$ is the $\Q$-stabilizer of $\mathscr{Z}^{\mathrm{Zar}}$, so the $\Q$-stabilizer of $\mathscr{Z}_{/H}^{\mathrm{Zar}}$ must be $1$. Thus \Cref{PropBignessStab} implies
\begin{equation}\label{EqASInduction}
\dim \mathscr{Z}_{/H}^{\mathrm{Zar}} - \dim \mathscr{Z}_{/H} \ge \dim p_{\cD/H}(\mathscr{Z}_{/H})^{\mathrm{ws}},
\end{equation}
where $p_{\cD/H} \colon S \times \cD/H \rightarrow \cD/H$ is the natural projection.

Let $\cR_u(H)$ be the unipotent radical of $H$. As $H(\R)^+\cR_u(H)(\C)\mathscr{Z}^{\mathrm{Zar}} = \mathscr{Z}^{\mathrm{Zar}}$, we have (for any $\tilde{s} \in \tilde{S}$)
\begin{equation}\label{EqDimZarInduction}
\dim \mathscr{Z}^{\mathrm{Zar}} = \dim \mathscr{Z}_{/H}^{\mathrm{Zar}} + \dim H(\R)^+\cR_u(H)(\C)\tilde{s}
\end{equation}
and
\begin{equation}\label{EqDimWSInduction}
\dim p_{\cD}(\mathscr{Z})^{\mathrm{ws}} = \dim p_{\cD/H}(\mathscr{Z}_{/H})^{\mathrm{ws}} + \dim H(\R)^+\cR_u(H)(\C)\tilde{s}.
\end{equation}
By \eqref{EqASInduction}, \eqref{EqDimZarInduction} and \eqref{EqDimWSInduction}, we then have
\begin{equation}
\dim \mathscr{Z}^{\mathrm{Zar}} - \dim \mathscr{Z}_{/H} \ge \dim p_{\cD}(\mathscr{Z})^{\mathrm{ws}}.
\end{equation}
So it remains to prove $\dim \mathscr{Z} = \dim \mathscr{Z}_{/H}$. Hence it remains to prove that each fiber of
\[
(\mathrm{id}_S, p_H) \colon S \times_{\Gamma_P\backslash \cD} \cD \rightarrow S \times_{\Gamma_{P/H}\backslash(\cD/H)} (\cD/H)
\]
is at most a countable set. This is true: Suppose $(s_1,\tilde{x}_1)$ and $(s_2,\tilde{x}_2)$ are in the same fiber, then $s_1 = s_2$. But any point $(s,\tilde{x}) \in S \times_{\Gamma_P\backslash \cD} \cD$ satisfies $[\Phi](s) = u(\tilde{x})$. So we have $u(\tilde{x}_1) = u(\tilde{x}_2)$, and hence $\tilde{x}_1 \in \Gamma_P \tilde{x}_2$. So each fiber of the map $(\mathrm{id}_S, p_H)$ above is contained in a $\Gamma_P$-orbit, and thus is at most a countable set.

\appendix
\renewcommand{\thesection}{\Alph{section}}
\setcounter{section}{0}

\section{Basic knowledge on Mumford--Tate domains}\label{Appendix}

\subsection{Some fundamental properties of Mumford--Tate domains}\label{Appendix2}
The goal of this subsection is to prove \Cref{PropMTDomainComplex} and \Cref{LemmaSmallestMTDomain}.

Let $V$ be a finite-dimensional $\Q$-vector space, and let $\cM$ be the classifying space of $\Q$-mixed Hodge structures constructed in \Cref{SubsectionClassifyingSpace}. We have seen that $\cM$ is a homogeneous space under $P^{\cM}(\R)^+W_{-1}^{\cM}(\C)$ for the $\Q$-algebraic group $P^{\cM}$ constructed in \eqref{EquationQGroupPM} and $W_{-1}^{\cM} = \cR_u(P^{\cM})$.

Let $h \in \cM$. Recall that the adjoint Hodge structure on $\lie P^{\cM}$ defined by $h$ has weight $\le 0$ by part (iii) of \Cref{propPink}. The following lemma is a rephrase of \cite[Thm.~3.13]{Pearlstein2000}.

\begin{lemma}\label{LemmaTangentSpaceOfClassifyingSpace}
The tangent space $T_h\cM$ can be canonically identified with
\[
\bigoplus_{r < 0, ~ r+s\le 0} (\lie P^{\cM}_{\C})^{r,s} = \bigoplus_{r < 0} (\lie P^{\cM}_{\C})^{r,s}.
\]
\end{lemma}

With this lemma, we are ready to prove \Cref{PropMTDomainComplex}.
\begin{proof}[Proof~of~\Cref{PropMTDomainComplex}]
Let $\cD = P(\R)^+W_{-1}(\C)h$ be a Mumford--Tate domain contained in $\cM$, where $P = \mathrm{MT}(h)$ and $W_{-1} = \cR_u(P)$.

Because $\cD$ and $\cM$ are homogeneous spaces, to prove that $\cD$ is a complex submanifold of $\cM$ it suffices to prove that $T_h\cD$ is a complex subspace of $T_h\cM$.

 $\lie P$ is a sub-Hodge structure of $\lie P^{\cM}$ for the adjoint Hodge structure on $\lie P^{\cM}$ induced by $h$. So $F^0 \lie P_{\C} = F^0 \lie P^{\cM}_{\C} \cap \lie P_{\C}$. By \Cref{LemmaTangentSpaceOfClassifyingSpace}, the complex structure on $T_h\cM$ is given by
\[
\lie P^{\cM}_{\C}/F^0\lie P^{\cM}_{\C} =  \bigoplus_{r < 0} (\lie P^{\cM}_{\C})^{r,s}.
\]
Thus $T_h\cD = \lie P_{\C}/ \left( F^0 \lie P^{\cM}_{\C} \cap \lie P_{\C}\right) = \lie P_{\C}/F^0\lie P_{\C}$ is a complex subspace of $T_h\cM$. Thus we can conclude that $\cD$ is a complex submanifold of $\cM$. Moreover we have shown that
\begin{equation}\label{EqTangentSpaceMTDomain}
T_h\cD = \bigoplus_{r<0}(\lie P_{\C})^{r,s}.
\end{equation}

The proof for weak Mumford--Tate domains is the similar. The only new input is to prove that $\lie N$ is a sub-Hodge structure of $\lie P^{\cM}$ for the normal subgroup $N$ of $P:=\mathrm{MT}(h)$ from \Cref{DefnMTDomain}.(2). This is true because the adjoint action of $P$ on $\lie  P$ leaves $\lie N$ stable (since $N\lhd P$), and the adjoint action $\mathrm{Ad} \colon P \rightarrow \mathrm{GL}(\lie P)$ is precisely the restriction of $\mathrm{Ad}^{\cM} \colon P^{\cM} \rightarrow \mathrm{GL}(\lie P^{\cM})$ restricted to $P$ (which leaves $\lie P$ stable).
\end{proof}

Next we turn to the Mumford--Tate group $\mathrm{MT}(h)$. For $m, n \in \Z_{\ge 0}$, denote by $T^{m,n}V := V^{\otimes m} \otimes (V^\vee)^{\otimes n}$. Then $h$ induces a $\Q$-mixed Hodge structure on $T^{m,n}V$, whose weight filtration we denote by $W_{\bullet}$ and Hodge filtration we denote by $F^{\bullet}$. 

The elements of $(T^{m,n}V_{\C})^{0,0} \cap T^{m,n}V = F^0(T^{m,n}V_{\C})  \cap W_0(T^{m,n}V)$, with $m$ and $n$ running over all non-negative integers, are called the \textit{Hodge tensors} for  $h$. Denote by $\mathrm{Hdg}_h$ the set of all Hodge tensors for $h$.

The following result is proved by Andr\'{e} \cite[Lem.~2.(a)]{AndreMumford-Tate-gr}, with pure case by Deligne.
\begin{lemma}\label{LemmaMTGroupCentralizerOfHodgeTensors}
We have
\begin{enumerate}
\item[(i)] Any element in some $T^{m,n}V$ fixed by $\mathrm{MT}(h)(\Q)$ is a Hodge tensor for $h$;
\item[(ii)] $\mathrm{MT}(h) = Z_{\mathrm{GL}(V)}(\mathrm{Hdg}_h)$.
\end{enumerate}
\end{lemma}

By dimension reasons, \Cref{LemmaMTGroupCentralizerOfHodgeTensors}.(ii) has the following consequence.
\begin{cor}\label{CorMumfordTateGroupCentralizerOfHodgeTensors}
There exists a finite set $\mathfrak{I} \subseteq \mathrm{Hdg}_h$  such that $\mathrm{MT}(h) = Z_{\mathrm{GL}(V)}(\mathfrak{I})$.
\end{cor}

Now we are ready to characterize Mumford--Tate domains contained in $\cM$ as irreducible components of \textit{Hodge loci}.


\begin{defn} For each $h \in \cM$, the \textbf{Hodge locus} at $h$ is defined as
\begin{equation}\label{EqDefnHodgeLocus}
\mathrm{HL}(h) = \{ h' \in \cM : \mathrm{Hdg}_h\subseteq \mathrm{Hdg}_{h'}\}.
\end{equation}
\end{defn}

\begin{lemma}\label{LemmaHodgeLocusInTermsOfMTGroup}
We have
\begin{enumerate}
\item[(i)]  $\mathrm{HL}(h) = \{h' \in \cM : \mathrm{MT}(h') < \mathrm{MT}(h)\}$.
\item[(ii)] $\mathrm{HL}(h) = \{ h' \in \cM : \mathfrak{I} \subseteq \mathrm{Hdg}_{h'}\}$ where $\mathfrak{I}$ is the finite set from \Cref{CorMumfordTateGroupCentralizerOfHodgeTensors}.
\end{enumerate}
\end{lemma}
\begin{proof}
\begin{enumerate}
\item[(i)] 
The inclusion $\subseteq$ is clear by \Cref{LemmaMTGroupCentralizerOfHodgeTensors}.(ii). Conversely suppose $\mathrm{MT}(h') < \mathrm{MT}(h)$. Then any $t \in \mathrm{Hdg}_h$ is fixed by $\mathrm{MT}(h)$ by \Cref{LemmaMTGroupCentralizerOfHodgeTensors}.(ii), and so is also fixed by $\mathrm{MT}(h')$, and thus is a Hodge tensor for $h'$ by \Cref{LemmaMTGroupCentralizerOfHodgeTensors}.(i). Therefore $\mathrm{Hdg}_h \subseteq \mathrm{Hdg}_{h'}$.
\item[(ii)] We first prove the inclusion $\subseteq$. Let $h' \in \mathrm{HL}(h)$. By \Cref{CorMumfordTateGroupCentralizerOfHodgeTensors} and (i), each $t \in \mathfrak{I}$ is fixed by $\mathrm{MT}(h')(\Q)$, and hence is a Hodge tensor for $h'$ by \Cref{LemmaMTGroupCentralizerOfHodgeTensors}.(i). So $\mathfrak{I}  \subseteq \mathrm{Hdg}_{h'}$. This proves the desired inclusion.

Conversely suppose that $h' \in \cM$ satisfies $\mathfrak{J} \subseteq \mathrm{Hdg}_{h'}$. Then $Z_{\mathrm{GL}(V)}(\mathrm{Hdg}_{h'}) \subseteq Z_{\mathrm{GL}(V)}(\mathfrak{I})$. Thus $\mathrm{MT}(h') < \mathrm{MT}(h)$ by \Cref{LemmaMTGroupCentralizerOfHodgeTensors}.(ii) and  \Cref{CorMumfordTateGroupCentralizerOfHodgeTensors}. So $h' \in \mathrm{HL}(h)$ by part (i) of the current lemma. This proves the inclusion $\supseteq$. Now we are done. \qedhere
\end{enumerate}
\end{proof}

By \Cref{LemmaHodgeLocusInTermsOfMTGroup}.(ii), $\mathrm{HL}(h)$ is the complex analytic subvariety of $\cM$ which parametrizes $\Q$-mixed Hodge structures (satisfying the properties (1)-(3) in \Cref{SubsectionClassifyingSpace}) together with the Hodge tensors in the finite set $\mathfrak{I}$.

\begin{prop}\label{PropMTDomainIsHodgeLocus}
Let $h \in \cM$, with $P=\mathrm{MT}(h)$ and $W_{-1}=\cR_u(P)$. Then $P(\R)^+W_{-1}(\C)h$ is the complex analytic irreducible component of $\mathrm{HL}(h)$ passing through $h$.
\end{prop}
\begin{proof} The proof is simply \cite[Prop.~17.1.2]{CMP} adapted to the mixed case. For completeness we include it here.

Denote by $\cD = P(\R)^+W_{-1}(\C) h$. Each $h' \in \cD$ equals $g\cdot h$ for some $g \in P(\R)^+W_{-1}(\C)$, and hence the homomorphism $h' \colon \S_{\C}\rightarrow \mathrm{GL}(V_{\C})$ factors through $g P_{\C} g^{-1} = P_{\C}$. Thus $\mathrm{MT}(h') < P$. So \Cref{LemmaHodgeLocusInTermsOfMTGroup}.(i) implies $h' \in \mathrm{HL}(h)$ for each $h' \in \cD$. Therefore
\begin{equation}\label{EqMTDomainContainedInHodgeLocus}
\cD  \subseteq \mathrm{HL}(h).
\end{equation}
Next we study $T_h(\mathrm{HL}(h)) \subseteq T_h\cM = \bigoplus_{r < 0} (\lie P^{\cM}_{\C})^{r,s}$; see \Cref{LemmaTangentSpaceOfClassifyingSpace} for the last equality. By \eqref{EqMTDomainContainedInHodgeLocus} and \eqref{EqTangentSpaceMTDomain}, to prove the proposition it suffices to prove
\begin{equation}\label{EquationTangentSpaceOfHodgeLocusContainedInMTDomain}
T_h (\mathrm{HL}(h)) \subseteq  \bigoplus_{r<0} (\lie P_{\C})^{r,s}.
\end{equation}
Indeed the action of $P^{\cM}$ on $T^{m,n}V$ induces an action of $T_h\cM$ on $T^{m,n}V$ in the following way: $\xi \cdot t = \frac{\mathrm{d}}{\mathrm{d}u}(e^{u \xi}\cdot t)|_{u=0}$, for $\xi \in T_h\cM = \bigoplus_{r < 0} (\lie P^{\cM}_{\C})^{r,s}$ and $t \in T^{m,n}V$. Then for any vector $\xi \in T_h\cM = \bigoplus_{r < 0} (\lie P^{\cM}_{\C})^{r,s}$, we have
\begin{equation}\label{EqCharacterizationOfTangentSpaceOfHodgeLocus}
\xi \in T_h(\mathrm{HL}(h))  \Leftrightarrow  \xi\cdot t \in \mathrm{Hdg}_h \text{ for each } t \in \mathrm{Hdg}_h.
\end{equation}
Now take $\xi \in T_h(\mathrm{HL}(h))$. Suppose $t \in T:=T^{m,n}V$ is a Hodge tensor, namely $t \in F^0T_{\C}  \cap W_0T \subseteq T_{\C}^{0,0}$.\footnote{Here the notation $T_{\C}^{0,0}$ means the $(0,0)$-constituent for the bi-grading of $T$ given by \Cref{PropBiGradingMHS}.} Then \eqref{EqCharacterizationOfTangentSpaceOfHodgeLocus} implies $\xi \cdot t \in F^0T_{\C}  \cap W_0T \subseteq T_{\C}^{0,0}$. On the other hand $\xi \in \bigoplus_{r < 0} (\lie P^{\cM}_{\C})^{r,s}$. Write $\xi = \sum_{r<0} \xi^{r,s}$. Then $\xi \cdot t =\sum_{r<0} \xi^{r,s}\cdot t \in \bigoplus_{r<0}T_{\C}^{r,s}$. Thus $\xi\cdot t \in T^{0,0} \cap \bigoplus_{r<0}T_{\C}^{r,s} = 0$. In summary
\begin{equation}
\xi \in T_h(\mathrm{HL}(h)) \Rightarrow \xi\cdot t = 0 \text{ for each }t \in \mathrm{Hdg}_h.
\end{equation}
But part (ii) of \Cref{LemmaMTGroupCentralizerOfHodgeTensors} implies that $\{\xi \in \lie P^{\cM}_{\C} : \xi\cdot t = 0 \text{ for each }t \in \mathrm{Hdg}_h\} \subseteq \lie P_{\C}$ with $P= \mathrm{MT}(h)$. Thus $T_h(\mathrm{HL}(h)) \subseteq \lie P_{\C}$. So
\[
T_h(\mathrm{HL}(h)) \subseteq \lie P_{\C} \cap \oplus_{r>0} (\lie P^{\cM}_{\C})^{r,s} = \oplus_{r>0}(\lie P_{\C})^{r,s}.
\]
This is precisely \eqref{EquationTangentSpaceOfHodgeLocusContainedInMTDomain}. Hence we are done.
\end{proof}

Now by \Cref{PropMTDomainIsHodgeLocus} and \Cref{LemmaHodgeLocusInTermsOfMTGroup}.(ii), the Mumford--Tate domains contained in $\cM$ are precisely the complex irreducible components of the moduli spaces parametrizing $\Q$-mixed Hodge structures (satisfying the properties (1)-(3) in \Cref{SubsectionClassifyingSpace}) together with a finite number of Hodge tensors.

\begin{proof}[Proof~of~\Cref{LemmaIntersectionMTDomains}] This is an immediate consequence of the moduli interpretation of Mumford--Tate domains above.
\end{proof}

Another application is as follows.
\begin{cor}\label{CorMTDomainCountablyMany}
There are at most countably many Mumford--Tate domains in $\cM$.
\end{cor}
\begin{proof} We have the moduli interpretation of Mumford--Tate domains above. On the other hand, every complex analytic variety has at most countably many irreducible components, and by definition there are countably many Hodge tensors. Hence there are at most countably many Mumford--Tate domains contained in $\cM$.
\end{proof}

This allows to prove a stronger version of \Cref{LemmaSmallestMTDomain}.
\begin{lemma}\label{LemmaDescriptionSmallestSpecialSubvariety}
Let $\cZ$ be a complex analytic irreducible subvariety  of $\cM$. Let $P = \mathrm{MT}(\cZ)$ be the generic Mumford--Tate group of $\cZ$. Then $\cZ^{\mathrm{sp}}$, the smallest Mumford--Tate domain which contains $\cZ$, is precisely $P(\R)^+W_{-1}(\C)h$ for some $h \in \cZ$, where $W_{-1} = \cR_u(P)$.
\end{lemma}
\begin{proof}
Denote by $\cZ^o$ the set of Hodge generic points in $\cZ$. Then $\cZ^o$ is the complement of the union of countably many proper complex analytic irreducible subvarieties of $\cZ$. In particular, $\cZ^o$ is irreducible since $\cZ$ is.

It is clearly true that $\cZ^o \subseteq \bigcup_{h \in \cZ^o} P(\R)^+W_{-1}(\C)h$. Each member in the union is by definition a Mumford--Tate domain, and hence the union is at most a countable union by \Cref{CorMTDomainCountablyMany}. Moreover two $P(\R)^+W_{-1}(\C)$-orbits either coincide or are disjoint. So $\cZ^o$ is contained in a countable disjoint union of some $P(\R)^+W_{-1}(\C)$-orbits. But $\cZ^o$ is irreducible, so it is contained some member in the union. Thus $\cZ^o \subseteq P(\R)^+W_{-1}(\C)h$ for some $h \in \cZ^o$. But then $\cZ \subseteq P(\R)^+W_{-1}(\C)h$. Hence we are done.
\end{proof}

Now we are ready to prove \Cref{LemmaSameOrbitUnderNormalSubgroup}.
\begin{proof}[Proof~of~\Cref{LemmaSameOrbitUnderNormalSubgroup}]
By assumption $\cD = P(\R)^+W_{-1}(\C)h$. From now on we fix $h' \in \cD$ Hodge generic, namely $\mathrm{MT}(h') = \mathrm{MT}(\cD)$.

By \Cref{LemmaDescriptionSmallestSpecialSubvariety} we have
\begin{equation}\label{Eq11111}
P(\R)^+W_{-1}(\C)h' = \cD \subseteq \mathrm{MT}(\cD)(\R)^+\cR_u(\mathrm{MT}(\cD))(\C)h'.
\end{equation}

Let us  prove $\mathrm{MT}(\cD)< P$. Indeed, each point $h' \in \cD$ is of the form $g \cdot h$ for some $g \in P(\R)^+W_{-1}(\C)$. The homomorphism $h' = g \cdot h \colon \S_{\C} \rightarrow \mathrm{GL}(V_{\C})$ factors through $g h(\S_{\C}) g^{-1} \subseteq g P_{\C} g^{-1} = P_{\C}$. Hence $\mathrm{MT}(h') < P$ for all $h' \in \cD$. So $\mathrm{MT}(\cD)< P$.

Next we show that $\mathrm{MT}(\cD)$ is normal in $P$. Indeed   for any $g \in P(\Q)$, we have
\[
\mathrm{MT}(\cD) \supseteq \mathrm{MT}(g\cdot h') = g \mathrm{MT}(h') g^{-1} = g\mathrm{MT}(\cD)g^{-1}.
\]
By comparing dimensions, we have $\mathrm{MT}(\cD)= g\mathrm{MT}(\cD) g^{-1}$. By letting $g$ run over elements in $P(\Q)$, we get $\mathrm{MT}(\cD) \lhd P$. In particular $\cR_u(\mathrm{MT}(\cD)) = W_{-1} \cap \mathrm{MT}(\cD)$.

This implies
\begin{equation}\label{Eq22222}
  \mathrm{MT}(\cD)(\R)^+\cR_u(\mathrm{MT}(\cD))(\C)h' \subseteq P(\R)^+W_{-1}(\C)h' .
\end{equation}
Thus $\cD = \mathrm{MT}(h')(\R)^+ \cR_u(\mathrm{MT}(h'))(\C) h'$ by \eqref{Eq11111} and \eqref{Eq22222}. So $\cD$ is a Mumford--Tate domain.
\end{proof}

\section{Underlying group}\label{SubsectionUnderlyingGroup}
Let $\cD$ be a Mumford--Tate domain in some classifying space $\cM$
with $P = \mathrm{MT}(\cD)$. Each $h \in \cD$
defines an adjoint Hodge structure on $\lie P$. Write $W_\bullet$ for
the weight filtration. By property (iii) of \Cref{propPink}
$W_\bullet$ does not depend on the choice of $h \in \cD$ and satisfies
$W_0(\lie P) = \lie P$ and $W_{-1}= \cR_u(P)$.

The weight filtration $0 = W_{-m-1}(\lie P) \subseteq W_{-m}(\lie P)
\subseteq \cdots \subseteq W_{-1}(\lie P)$ defines a sequence of
connected subgroups
\begin{equation}\label{EqDistSeqSubgp}
0 = W_{-(m+1)} \subseteq W_{-m} \subseteq \cdots \subseteq W_{-1}
\end{equation}
of $P$.  Each $W_{-k}$, $k \in \{1,\ldots,m\}$, is a normal
unipotent subgroup of $P$.

Write as before $G = P/W_{-1}$ the reductive part of $P$. We wish to
reconstruct $P$ from $G$ and the $W_{-k}$'s.

Let us start with the
unipotent radical $W_{-1}$. 

\begin{lemma}\label{LemmaGroupDecomp}
  \begin{itemize}
  \item[(a)]
    For each $k \in \{1,\ldots,m\}$, $W_{-k}/W_{-(k+1)}$ is a vector
    group.
    \item[(b)] There is an isomorphism of $\Q$-algebraic varieties  
\begin{equation} \label{EqDecompUnipotentRad}
\begin{array}{lll}
W_{-1} & \to & (W_{-1}/W_{-2}) \times \cdots \times (W_{-(m-1)}/W_{-m}) \times
W_{-m}\\
\mathbf{w} &\mapsto & (w_{1}, \cdots, w_{m-1}, w_{m}) \;\;.
\end{array}
\end{equation}
\end{itemize}
\end{lemma}
\begin{proof}
We first prove (a).
For each $k \in \{1,\ldots,m\}$, the algebraic group $W_{-k}/W_{-(k+1)}$ is
unipotent since $W_{-k}$ is unipotent. On the other hand $[\lie
W_{-k},\lie W_{-k}] \subseteq W_{-2k}$ by reason of weight, and
$W_{-2k} \subseteq W_{-(k+1)}$ as $k \ge 1$. Thus $\lie
W_{-k}/W_{-(k+1)}$ is a commutative Lie algebra, hence
$W_{-k}/W_{-(k+1)}$ is an abelian algebraic group. Finally the
algebraic group $W_{-k}/W_{-(k+1)}$
is a vector group as it is abelian and unipotent.

\medskip
We now turn to the description of the isomorphism~(\ref{EqDecompUnipotentRad}).
As $W_{-1}$ is unipotent, the exponential map $\exp \colon \lie W_{-1} \rightarrow W_{-1}$ is an
isomorphism of $\Q$-algebraic varieties.

Fix an isomorphism of $\Q$-vector spaces $\lie W_{-1} \cong \bigoplus_{j=1}^m
\lie W_{-j}/W_{-(j+1)}$. As part (a) provides a canonical
identification of $\Q$-algebraic varieties $\lie
W_{-k}/W_{-(k+1)} = W_{-k}/W_{-(k+1)}$ between a vector group and its
Lie algebra, we get the desired the isomorphism~(\ref{EqDecompUnipotentRad}) by
\[
W_{-1} \xleftarrow[\sim]{\exp}  \lie W_{-1} = \bigoplus_{j=1}^m \lie (W_{-j}/W_{-(j+1)}) = \prod_{j=1}^m  W_{-j}/W_{-(j+1)}.   \qedhere
\]
\end{proof}

Notice that this isomorphism (\ref{EqDecompUnipotentRad}) is not
canonical. In this paper, we fix such an isomorphism once and for all.


Next we give the formula for the group law on $W_{-1}$ under 
this identification given by (\ref{EqDecompUnipotentRad}).

\begin{defn} For $k \in \{1, \cdots, m\}$ we define the $k$-truncation $\mathbf{w}_k \in W_{-1}/
W_{-k-1} \simeq \prod_{j=1}^k  W_{-j}/W_{-(j+1)}$ of an element
$\mathbf{w} \in W_{-1}$ as follows. If $\mathbf{w} =  (w_{1}, \cdots, w_{m-1}, w_{m})$ under the
identification (\ref{EqDecompUnipotentRad}), then $\mathbf{w}_k=  (w_{1}, \cdots, w_{k}) $.
\end{defn}

\begin{lemma}\label{LemmaCalbFactorProdUni}
  For each $k \ge 2$, there exists a polynomial map
  $$\mathrm{calb}_k \colon  W_{-1}/
W_{-k-1} \times W_{-1}/
W_{-k-1} \to W_{-k}/
W_{-k-1}$$ of degree at most $k-1$ and constant term $0$
such that for any $\mathbf{w}, \mathbf{w}' \in W_{-1}$, their product
is given under the identification (\ref{EqDecompUnipotentRad}) by
\begin{equation}\label{EqGroupLawUnipotent}
\mathbf{w} \cdot \mathbf{w}' = (w_1+w'_1,w_2+w'_2+\mathrm{calb}_2(\mathbf{w}_{1},\mathbf{w}_{1}'),\ldots ,w_m+w'_m+\mathrm{calb}_m(\mathbf{w}_{m-1},\mathbf{w}'_{m-1})).
\end{equation}
\end{lemma}
\begin{proof}
Let $\mathbf{w} = (w_1,\ldots,w_m)$ and $\mathbf{w}'
= (w'_1, \ldots, w'_m)$ under (\ref{EqDecompUnipotentRad}). The Baker--Campbell--Hausdorff formula gives:
\begin{equation}\label{EqBCH}
\mathbf{w}\cdot \mathbf{w}' = \exp\left( (w_1,\ldots,w_m) + (w'_1, \ldots, w'_m) + \frac{1}{2} [(w_1, \ldots, w_m), (w'_1, \ldots, w'_m)] + \ldots \right),
\end{equation}
where all operations in the exponential are taken in $\lie W_{-1}$,
and the sum is finite as $\lie W_{-1}$ is nilpotent.
Noticing  that
\begin{equation*}\label{EqLieBracket}
[W_{-k}/W_{-(k+1)}, W_{-k'}/W_{-(k'+1)}] \subseteq W_{-(k+k')}/W_{-(k+k'+1)},
\end{equation*}
one can rewrite~(\ref{EqBCH}) as 
\[
\mathbf{w} \cdot \mathbf{w}' = \exp\left( (w_1+w'_1,w_2+w'_2+\mathrm{calb}_2(\mathbf{w}_{1},\mathbf{w}_{1}'),\ldots ,w_m+w'_m+\mathrm{calb}_m(\mathbf{w}_{m-1},\mathbf{w}'_{m-1})) \right),
\]
with polynomials $\mathrm{calb}_k$ for each $k\ge 2$ as required by the lemma.
\end{proof}

The next lemma explains how $G = P/W_{-1}$ acts on $W_{-1} = \cR_u(P)$
under the identification~(\ref{EqDecompUnipotentRad}).

\begin{lemma}
For each $k \ge 1$, $W_{-k}/W_{-(k+1)}$ is a $G$-module. Moreover this $G$-module structure is induced by the action of $G$ on $W_{-1}$.

As a consequence, for each $g_0 \in G$ and $\mathbf{w} = (w_1,\ldots,w_m) \in W_{-1}$ under (\ref{EqDecompUnipotentRad}), we have
\begin{equation}\label{EqActionOfGonW}
g_0\cdot \mathbf{w} = (g_0w_1,\ldots,g_0w_m).
\end{equation}
\end{lemma}
\begin{proof}
As $\mathrm{Gr}_0^{W_\bullet}(\lie P) = \lie G$ and $W_{-k}(\lie P) = \lie W_{-k}$ for each $k \ge 1$, we have $[\lie G, \lie W_{-k}] \subseteq \lie W_{-k}$. Hence the action of $G$ on $W_{-1}$ preserves $W_{-k}$ for each $k \ge 1$, and hence furthermore induces an action on $W_{-k}/W_{-(k+1)}$ which is a $\Q$-vector space. This concludes the lemma.
\end{proof}

We are now ready to state the result to reconstruct $P$ from $G$ and the $W_{-k}$'s. First let us fix a Levi decomposition $P = W_{-1} \rtimes G$.

\begin{prop}\label{PropUnderlyingGroup}
The fixed Levi decomposition $P = W_{-1} \rtimes G$ and the fixed isomorphism \eqref{EqDecompUnipotentRad} together induce an isomorphism as algebraic varieties defined over $\Q$
\begin{equation}\label{EqDecompGroup}
P \cong G \times (W_{-1}/W_{-2}) \times \cdots (W_{-(m-1)}/W_{-m}) \times W_{-m}.
\end{equation}

The group law on the right hand side of \eqref{EqDecompGroup} is given as follows. Let $(g_0,w_1,\ldots,w_m)$ and $(g_0',w_1',\ldots,w_m')$ be two elements in $P$ under the identification \eqref{EqDecompGroup}. Denote by $\mathbf{w} = (w_1,\ldots,w_m)$ and $\mathbf{w}' = (w_1',\ldots,w_m')$. Then
\scriptsize
\begin{equation}\label{EqFormulaGroupLaw}
(g_0, \mathbf{w})\cdot(g_0',\mathbf{w}') = (g_0 g_0', w_1 + g_0 w_1', w_2+g_0 w_2' + \mathrm{calb}_2(w_1, g_0 w_1'), \ldots, w_m + g_0 w_m' + \mathrm{calb}_m(\mathbf{w}_{m-1}, g_0\mathbf{w}'_{m-1}))
\end{equation}
\normalsize
where $\mathrm{calb}_2,\ldots,\mathrm{calb}_m$ are the $\Q$-polynomials from \Cref{LemmaCalbFactorProdUni}, $\mathbf{w}_k$ (resp. $\mathbf{w}'_k$) is the $k$-th truncation as in \Cref{LemmaCalbFactorProdUni}, and $g_0 \mathbf{w}_k' = (g_0w_1',\ldots,g_0w_k')$ for each $k \ge 1$.
\end{prop}
\begin{proof}
\eqref{EqDecompGroup} follows directly from the fixed Levi decomposition and \eqref{EqDecompUnipotentRad}.

To prove \eqref{EqFormulaGroupLaw}, first note that $(g_0,\mathbf{w}) = (1,\mathbf{w})\cdot (g_0,0)$ for $P = W_{-1}\rtimes G$. Similarly $(g_0',\mathbf{w}') = (1,\mathbf{w}')\cdot (g_0',0)$. So
\begin{align*}
(g_0,\mathbf{w}) \cdot (g_0',\mathbf{w}') & = (1,\mathbf{w})\cdot \left( (g_0,0) \cdot (1,\mathbf{w}') \right) \cdot (g_0',0) \\
& = (1,\mathbf{w}) \cdot (g_0, g_0\cdot \mathbf{w}') \cdot (g_0',0) \\
& = (1,\mathbf{w}) \cdot \left( (1,g_0\cdot \mathbf{w}') \cdot (g_0,0) \right) \cdot (g_0',0) \\
& = (1,w_1,\ldots,w_m) \cdot (1,g_0w_1', \ldots, g_0w_m') \cdot (g_0,0) \cdot (g_0',0) \quad\text{by \eqref{EqActionOfGonW}}\\
& = (1, w_1 + g_0 w_1', w_2+g_0 w_2' + \mathrm{calb}_2(w_1, g_0 w_1'), \ldots, \\
& \qquad \quad w_m + g_0 w_m' + \mathrm{calb}_m(\mathbf{w}_{m_1}, g_0\mathbf{w}'_{m-1})) \cdot (g_0g_0',0) \quad\text{by \eqref{EqGroupLawUnipotent}} \\
& = (g_0g_0', w_1 + g_0 w_1', w_2+g_0 w_2' + \mathrm{calb}_2(w_1, g_0 w_1'), \ldots, \\
& \qquad \qquad \qquad \qquad \qquad \qquad w_m + g_0 w_m' + \mathrm{calb}_m(\mathbf{w}_{m_1}, g_0\mathbf{w}'_{m-1})).   \qedhere
\end{align*}
\end{proof}

\bibliographystyle{alpha}
\bibliography{bibliography}

\begin{thebibliography}{BBKT20}

\bibitem[ACZ20]{ACZBetti}
Y.~Andr\'{e}, P.~Corvaja, and U.~Zannier.
\newblock The {B}etti map associated to a section of an abelian scheme (with an
  appendix by {Z.~Gao}).
\newblock {\em Inv. Math.}, 222:161--202, 2020.

\bibitem[And92]{AndreMumford-Tate-gr}
Y.~Andr\'{e}.
\newblock {M}umford-{T}ate groups of mixed {H}odge structures and the theorem
  of the fixed part.
\newblock {\em Compositio Mathematica}, 82(1):1--24, 1992.

\bibitem[BBKT20]{BBKT}
B.~Bakker, Y.~Brunebarbe, B.~Klingler, and J.~Tsimerman.
\newblock Definability of mixed period maps.
\newblock {\em JEMS, to appear}, 2020.

\bibitem[BKT20]{BKT}
B.~Bakker, B.~Klingler, and J.~Tsimerman.
\newblock Tame topology of arithmetic quotients and algebraicity of hodge loci.
\newblock {\em J. Amer. Math. Soc}, 2020.

\bibitem[BP13]{SL2Splitting}
P.~Brosnan and G.~Pearlstein.
\newblock On the algebraicity of the zero locus of an admissible normal
  function.
\newblock {\em Compos. Math.}, 149(11):1913--1962, 2013.

\bibitem[BT19]{BTAS}
B.~Bakker and J.~Tsimerman.
\newblock The {A}x-{S}chanuel conjecture for variations of {H}odge structures.
\newblock {\em Inv. Math.}, 217:77--94, 2019.

\bibitem[BT20]{BTASLectureNotes}
B.~Bakker and J.~Tsimerman.
\newblock Lectures on the ax--schanuel conjecture.
\newblock In Marc-Hubert Nicole, editor, {\em Arithmetic Geometry of
  Logarithmic Pairs and Hyperbolicity of Moduli Spaces: Hyperbolicity in
  Montr{\'e}al}, pages 1--68. Springer International Publishing, Cham, 2020.

\bibitem[BZ98]{brylinski1998overview}
J.-L. Brylinski and S.~Zucker.
\newblock An overview of recent advances in {H}odge theory.
\newblock {\em Complex manifolds}, pages 39--142, 1998.

\bibitem[Chi21]{KennethChiuAS}
K.C.T. Chiu.
\newblock {A}x-{S}chanuel for variations of mixed {H}odge structures.
\newblock {\em preprint, arXiv: 2101.10968}, 2021.

\bibitem[CMSP17]{CMP}
J.~Carlson, S.~M\"{u}ller-Stach, and C.~Peters.
\newblock {\em Period Mappings and Period Domains}, volume 168 of {\em
  Cambridge studies in advanced mathematics}.
\newblock Cambridge Univ. Press, 2017.

\bibitem[Del71]{DeligneHodgeII}
P.~Deligne.
\newblock Th\'{e}orie de {H}odge: {II}.
\newblock {\em Publications Math\'{e}matiques de l'IH\'{E}S}, 40:5--57, 1971.

\bibitem[Del87]{Del}
P.~Deligne.
\newblock Un th\'eor\`eme de finitude pour la monodromie.
\newblock In {\em Discrete groups in geometry and analysis}, volume~67 of {\em
  Progr. Math.} Birk\"auser, 1987.

\bibitem[DGH21]{DGHUnifML}
V.~Dimitrov, Z.~Gao, and P.~Habegger.
\newblock Uniformity in {M}ordell--{L}ang for curves.
\newblock {\em Annals of Mathematics}, 194(1):237--298, 2021.

\bibitem[DR18]{DawRenAppOfAS}
C.~Daw and J.~Ren.
\newblock Applications of the hyperbolic {A}x-{S}chanuel conjecture.
\newblock {\em Compositio Mathematica}, 154(9):1843--1888, 2018.

\bibitem[EZ86]{EZ}
F.~El~Zein.
\newblock Th\'eorie de hodge des cycles évanescents.
\newblock {\em Ann. Sci. École Norm. Sup.}, 19(1):107--184, 1986.

\bibitem[Gao16]{GaoAbout-the-mixed}
Z.~Gao.
\newblock About the mixed {A}ndr\'{e}-{O}ort conjecture: reduction to a lower
  bound for the pure case.
\newblock {\em Comptes rendus Mathematiques}, 354:659--663, 2016.

\bibitem[Gao17]{GaoTowards-the-And}
Z.~Gao.
\newblock Towards the {A}ndr\'{e}-{O}ort conjecture for mixed {S}himura
  varieties: the {A}x-{L}indemann-{W}eierstrass theorem and lower bounds for
  {G}alois orbits of special points.
\newblock {\em J.Reine Angew. Math (Crelle)}, 732:85--146, 2017.

\bibitem[Gao20a]{GaoBettiRank}
Z.~Gao.
\newblock Generic rank of {B}etti map and unlikely intersections.
\newblock {\em Compos. Math.}, 156(12):2469--2509, 2020.

\bibitem[Gao20b]{GaoAxSchanuel}
Z.~Gao.
\newblock Mixed {A}x-{S}chanuel for the universal abelian varieties and some
  applications.
\newblock {\em Compositio Mathematica}, 156(11):2263--2297, 2020.

\bibitem[Has21]{Hast}
D.~Hast.
\newblock Functional transcendence for the unipotent {A}lbanese map.
\newblock {\em Algebra \& Number Theory}, 15:1565--1580, 2021.

\bibitem[Kap95]{Kaplan}
A.~Kaplan.
\newblock Notes on the moduli spaces of {H}odge structures.
\newblock 1995.

\bibitem[Kas86]{KashiwaraA-study-of-vari}
M.~Kashiwara.
\newblock A study of variation of mixed {H}odge structure.
\newblock {\em Publ. RIMS Kyoto Univ.}, 22:991--1024, 1986.

\bibitem[Kli17]{Klingler-conj-}
B.~Klingler.
\newblock Hodge theory and atypical intersections: conjectures, to appear in
  the book {\em motives and complex multiplication}, birkha\"user.
\newblock {\em ArXiv: 1711.09387}, 2017.

\bibitem[KUY16]{KlinglerThe-Hyperbolic-}
B.~Klingler, E.~Ullmo, and A.~Yafaev.
\newblock The hyperbolic {A}x-{L}indemann-{W}eierstrass conjecture.
\newblock {\em Publ. math. IHES}, 123:333--360, 2016.

\bibitem[LS20]{LawrenceSawin}
B.~Lawrence and W.~Sawin.
\newblock The {S}hafarevich conjecture for hypersurfaces in abelian varieties.
\newblock {\em preprint, arXiv: 2004.09046}, 2020.

\bibitem[LV20]{LawrenceVenkatesh}
B.~Lawrence and A.~Venkatesh.
\newblock Diophantine problems and p-adic period mappings.
\newblock {\em Inv. Math.}, 221:893--999, 2020.

\bibitem[MPT19]{MokAx-Schanuel-for}
N.~Mok, J.~Pila, and J.~Tsimerman.
\newblock {A}x-{S}chanuel for {S}himura varieties.
\newblock {\em Annals of Mathematics}, 189:945--978, 2019.

\bibitem[Pea00]{Pearlstein2000}
Gregory~J. Pearlstein.
\newblock Variations of mixed {H}odge structure, {H}iggs fields, and quantum
  cohomology.
\newblock {\em manuscripta mathematica}, 102(3):269--310, 2000.

\bibitem[Pil11]{PilaO-minimality-an}
J.~Pila.
\newblock O-minimality and the {A}ndr\'{e}-{O}ort conjecture for
  $\mathbb{C}^n$.
\newblock {\em Annals Math.}, 173:1779--1840, 2011.

\bibitem[Pin89]{PinkThesis}
R.~Pink.
\newblock {\em Arithmetical compactification of mixed {S}himura varieties}.
\newblock PhD thesis, Bonner Mathematische Schriften, 1989.

\bibitem[PS09]{PeterzilComplex-analyti}
Y.~Peterzil and S.~Starchenko.
\newblock Complex analytic geometry and analytic-geometric categories.
\newblock {\em J.Reine Angew. Math (Crelle)}, 626:39--74, 2009.

\bibitem[PT13]{PilaAbelianSurfaces}
J.~Pila and J.~Tsimerman.
\newblock The {A}ndr\'{e}-{O}ort conjecture for the moduli space of {A}belian
  surfaces.
\newblock {\em Compositio Mathematica}, 149:204--216, February 2013.

\bibitem[PT14]{PilaAxLindemannAg}
J.~Pila and J.~Tsimerman.
\newblock {A}x-{L}indemann for $\mathcal{A}_g$.
\newblock {\em Annals Math.}, 179:659--681, 2014.

\bibitem[PT16]{PilaAx-Schanuel-for}
J.~Pila and J.~Tsimerman.
\newblock {A}x-{S}chanuel for the $j$-function.
\newblock {\em Duke Journal of Mathematics}, 165(13):2587--2605, 2016.

\bibitem[SZ85]{SteenbrinkVariation-of-mi}
J.~Steenbrink and S.~Zucker.
\newblock Variation of mixed {H}odge structure {I}.
\newblock {\em Inv. Math.}, 80:489--542, 1985.

\bibitem[Tsi18]{TsimermanA-proof-of-the-}
J.~Tsimerman.
\newblock A proof of the {A}ndr\'{e}-{O}ort conjecture for $\mathcal{A}_g$.
\newblock {\em Annals Math.}, 187:379--390, 2018.

\bibitem[UY14]{UllmoThe-Hyperbolic-}
E.~Ullmo and A.~Yafaev.
\newblock The hyperbolic {A}x-{L}indemann in the compact case.
\newblock {\em Duke Journal of Mathematics}, 163(2):433--463, 2014.

\end{thebibliography}




\end{document}